\documentclass{article}
\usepackage[utf8]{inputenc}
\usepackage[margin=1in]{geometry} 
\usepackage{amsmath,amsthm,amssymb,amsfonts}
\usepackage{tikz}
\usetikzlibrary{shapes.geometric, arrows}
\usepackage{verbatim}
\usepackage{dsfont}
\usepackage{cite}
\usepackage{algorithm}
\usepackage{algorithmic}
\usepackage{subcaption}
\usepackage[colorlinks=true]{hyperref}
\hypersetup{urlcolor=blue, citecolor=red, linkcolor=blue}
\usepackage[capitalise,noabbrev,nameinlink]{cleveref}
\usepackage[bottom]{footmisc}
\usepackage{multirow}
\usepackage[group-digits=false]{siunitx}
\usepackage{mwe}
\usepackage{mathtools}
\usepackage[inline]{enumitem}
\usepackage[]{mdframed}
\usepackage[skins,breakable]{tcolorbox}
\usepackage{booktabs}

\theoremstyle{plain}
\newtheorem{thm}{Theorem}[section] 
\newtheorem*{theorem*}{Main Result}
\newtheorem{lem}[thm]{Lemma}
\newtheorem{prop}[thm]{Proposition}

\newtheorem{cor}[thm]{Corollary}

\theoremstyle{definition}
\newtheorem{defn}[thm]{Definition}

\newtheorem{remark}[thm]{Remark}

\numberwithin{equation}{section}

\tikzstyle{dnc_step} = [rectangle, rounded corners, minimum width=3cm, minimum height=.7cm,text centered, text width = 4.5cm, draw=white]
\tikzstyle{arrow} = [draw, -latex',rounded corners]
\tikzstyle{point} = [circle, minimum width=.5cm, minimum height=.5cm, text centered, draw=black, fill=gray!30]
\tikzstyle{line} = [draw, -latex']

\usetikzlibrary{shapes,calc,arrows}

\begin{document}
\title{Fast and Stable Arbitrary Precision Approximate Matrix Multiplication}
\author{Ioana Dumitriu\footnote{University of California, San Diego 
} \and Olga Holtz\footnote{University of California, Berkeley 
} \and Abraham Hsu\footnote{Independent Researcher 
}  \and Nicholas Liskij\footnotemark[2] 
 \and Yoav Moran\footnote{The Hebrew University of Jerusalem 
 } \and Ryan Schneider\footnotemark[2] \footnote{Corresponding author (ryan.schneider@berkeley.edu)} 
 \and Oded Schwartz\footnotemark[4]
 \and Shuncheng Yuan\footnotemark[1] 
 }
\date{}

\maketitle

\begin{abstract}
    We present several new and surprising results on the theory and practice of  Arbitrary Precision Approximate (APA) matrix multiplication, which refers to algorithms that output merely approximations to matrix multiplication even in exact arithmetic. While APA algorithms have been around for nearly half a century, and despite the fact that they have potentially lower (asymptotic) cost than many fast exact matrix multiplication schemes, they have not been studied as closely, to the extent that relatively little is known about their numerical properties. In this paper, we shine a spotlight on these algorithms by presenting three important results. First, we provide a new search algorithm for APA routines that is easy to use, automatically generates encoding matrices with functional entries, and can match or beat the best known APA algorithms for several sizes. In particular, we improve the upper bound on the border rank of $\langle 2,4,5 \rangle$ matrix multiplication. Second, we combine this insight with alternative bases to minimize the number of additions in an APA algorithm, thereby improving its leading coefficient. Finally, we prove the same kind of stability that characterizes fast exact matrix multiplication for an algorithm of Bini [\textit{Calcolo} 1980] that computes an exact product of two matrices using a modified APA scheme (with or without alternative bases).

\end{abstract}

\small
\hspace{2mm} \textbf{Keywords:} Matrix multiplication, arbitrary precision, border rank, alternative bases, floating-point stability


\indent \hspace{2mm} \textbf{MSC Class:} 65F99
\normalsize

\vspace{.3cm}



\section{Introduction}\label{section: intro}
In 1969, Strassen \cite{Strassen1969} fundamentally changed the field of numerical linear algebra when he demonstrated that the classical algorithm for matrix multiplication 
is not optimal. Strassen's discovery sparked a still ongoing effort to derive, analyze, and implement \textit{fast} matrix multiplication algorithms, which are broadly defined as any matrix multiplication routine with computational complexity $O(n^{\omega})$ for $2 \leq \omega < 3$. At the time of writing, Dupont et al.\ \cite{dupont2026} hold the record for the smallest value of $\omega$ achieved by an algorithm, and it remains open whether the lower bound $\omega \geq 2$ is attainable. Research aimed at reaching this lower bound (or proving it cannot be done) is motivated by both the wide applicability of matrix multiplication in scientific computing/machine learning and the fact that many linear algebra operations have been shown to reduce to matrix multiplication \cite{fastLA}, including recently nonsymmetric eigenvalue problems \cite{Banks23,Demmel26}.   \\
\indent Most practical fast matrix multiplication algorithms, including Strassen's \cite{Strassen1969}, are based on routines for multiplying small matrices. The key idea here is straightforward:\ an algorithm for multiplying $m_0 \times k_0$ and $k_0 \times n_0$ matrices can be used to compute the product of arbitrary $A \in {\mathbb C}^{m \times k}$ and $B \in {\mathbb C}^{k \times n}$ by breaking the input into blocks
\begin{equation}
    A = \begin{pmatrix}
        A_{11}& \cdots &A_{1k_0} \\
    
        \vdots & & \vdots \\
        A_{m_01} & \cdots & A_{m_0k_0} 
        \end{pmatrix}
        \; \; \; \; \; \; B  = \begin{pmatrix}
        B_{11}& \cdots &B_{1n_0} \\
    
        \vdots & & \vdots \\
        B_{k_01} & \cdots & B_{k_0n_0}
        \end{pmatrix}
\end{equation}
and applying the algorithm at the block level. In this context, each scalar multiplication in the original ``base-case" algorithm becomes a matrix multiplication of size $(m /m_0) \times (k/k_0)$ by $(k/k_0) \times (n/n_0)$, assuming for simplicity that $m_0$, $k_0$, and $n_0$ divide $m$, $k$, and $n$, respectively. This procedure can then be used to handle the smaller block multiplications, continuing recursively either until the base case $m_0 \times k_0$ by $k_0 \times n_0$ is reached or, more commonly, until the block size is small enough to handle with classical matrix multiplication. The result is a (general) matrix multiplication algorithm of complexity $O((mkn)^{\log_{m_0k_0n_0}(t)})$, where $t$ is the number of scalar multiplications in the base-case routine. When $m = n = k$, this simplifies to $O(n^{\omega})$ for $\omega = 3\log_{m_0k_0n_0}(t)$. Throughout, we use $\langle m,k,n \rangle$ and $\langle m,k,n; t \rangle$ to refer, respectively, to $m\times k$ by $k \times n$ matrix multiplication and an algorithm for it that uses $t$ scalar multiplications.  \\
\indent The above observations underpin one branch of research on fast matrix multiplication, which is focused on searching for optimal routines for small matrices \cite{flip_graph,perminov2026,Smirnov2013,BALLARD2019,alphaevolve, Schwartz_Zwecher26,Khoruzhii26,Drevet2011,LADERMAN1992557,Pan,nature_MM}, where ``optimality" is measured by the number of scalar multiplications used. 
Algorithms derived this way are sometimes referred to as \textit{feasible} fast matrix multiplication because, while they typically do not achieve the best known upper bounds on $\omega$, they are applicable to inputs of more manageable size, making them candidates for practical use.\footnote{Indeed, algorithms that attain the best exponents -- e.g., \cite{dupont2026} -- are only faster on ``galactically" large matrices.} Accordingly, significant effort has been placed on optimizing these routines beyond their asymptotic complexity (see  \cite{alternative_bases}). \\
\indent The search for small matrix multiplication routines led to the discovery of Arbitrary Precision Approximate or APA matrix multiplication algorithms in the late 70's. First coined by Bini et al.\ \cite{BINI1979234}, APA matrix multiplication refers to a class of bilinear algorithms that, given input matrices $A$ and $B$ and an evaluation parameter $\varepsilon$, computes \begin{equation}\label{eqn: apa_output}
      AB + \varepsilon R_1 + \varepsilon^2 R_2 + \cdots + \varepsilon^{d} R_{d}, 
\end{equation}
where $R_1, \ldots, R_d$ are residual matrices (that depend on $A$ and $B$) and $d$ is the \textit{degree} of the algorithm. Note that the approximation obtained from an APA routine can be made arbitrarily accurate by taking $\varepsilon \rightarrow 0$. An example of an APA algorithm for multiplying $3 \times 2$ and $2 \times 2$ matrices -- alongside both classical and fast exact alternatives -- is shown in \cref{fig: alg_comp}, which in particular emphasizes how each algorithm obtains its output (exact or approximate) from a certain number of scalar multiplications. 

\begin{figure*}[t!]
    \centering
    \begin{subfigure}[t]{\textwidth}
        \centering 
        \[ \begin{pmatrix} a_{11} & a_{12} \\
        a_{21} & a_{22} \\ a_{31} & a_{32} \end{pmatrix}  \times \begin{pmatrix} b_{11} & b_{12} \\ b_{21} & b_{22} \end{pmatrix}  = \begin{pmatrix} c_{11} & c_{12} \\ c_{21} & c_{22} \\ c_{31} & c_{32} \end{pmatrix} \]
    \end{subfigure}
    
    \vspace{0.5cm}
    
    \begin{subfigure}[t]{0.3\textwidth}
        \centering
        \caption{Classical}
        \vspace{-.3cm}
        \[ \begin{aligned}
            m_1 &  = a_{11}b_{11} \\
            m_2 & = a_{12}b_{21} \\
            m_3 & = a_{11}b_{12} \\
            m_4 & = a_{12}b_{22} \\
            m_5 & = a_{21}b_{11} \\
            m_6 & = a_{22}b_{21} \\
            m_7 & = a_{21}b_{12} \\
            m_8 & = a_{22}b_{22} \\
            m_9 & = a_{31}b_{11} \\
            m_{10} & = a_{32}b_{21} \\
            m_{11} & = a_{31}b_{12} \\
            m_{12} & = a_{32}b_{22} 
        \end{aligned}\]
        \[ \begin{aligned}
            c_{11} &= m_1+m_2 \\
            c_{12} & = m_3+m_4 \\
            c_{21} & = m_5 +m_6 \\
            c_{22} & = m_7+m_8 \\
            c_{31}& = m_9+m_{10} \\
            c_{32} & = m_{11}+m_{12}
        \end{aligned}\]
    \end{subfigure}%
    \vrule
    \begin{subfigure}[t]{0.35\textwidth}
        \centering
        \caption{Fast exact \cite{Strassen1969}}
        \vspace{-.3cm}
        \[ \begin{aligned}
            m_1& = (a_{11}+a_{22})(b_{11}+b_{22}) \\
            m_2 & = (a_{21}+a_{22})b_{11} \\
            m_3 & = a_{11}(b_{12}-b_{22}) \\
            m_4 &= a_{22}(b_{21}-b_{11}) \\
            m_5 & = (a_{11}+a_{12})b_{22} \\
            m_6 & = (a_{21}-a_{11})(b_{11}+b_{12}) \\
            m_7 & = (a_{12}-a_{22})(b_{21}+b_{22}) \\
            m_8 & = a_{31}b_{11} \\
            m_9 & = a_{32}b_{21} \\
            m_{10} & = a_{31}b_{12} \\
            m_{11} & = a_{32}b_{22} \\
            \phantom{m_{12}} \\
        \end{aligned}\]
        \[ \begin{aligned}
            c_{11} &= m_1+m_4-m_5+m_7 \\
            c_{12} & = m_3+m_5 \\
            c_{21} &= m_2+m_4 \\
            c_{22} & = m_1-m_2+m_3+m_6 \\
            c_{31} &= m_8+m_9 \\
            c_{32} & = m_{10}+m_{11}
        \end{aligned}\]
    \end{subfigure}%
    \vrule
    \begin{subfigure}[t]{0.35\textwidth}
        \centering
        \caption{APA \cite{BINI1979234}}
        \vspace{-.3cm}
        \[ \begin{aligned}
            m_1 & =  (a_{21}+\varepsilon a_{11})b_{11} \\
            m_2 & = a_{12}b_{12} \\
            m_3 & = a_{21}(b_{11}+b_{12}+\varepsilon b_{21}) \\
            m_4 & = (a_{12}+a_{21})(b_{12}+\varepsilon b_{21}) \\
            m_5 & = (a_{12}+\varepsilon a_{11})(b_{12}+\varepsilon b_{22}) \\
            m_6 & = (a_{22}+\varepsilon a_{32})b_{21} \\
            m_7 & = a_{31}b_{22} \\
            m_8 & = a_{22}(b_{21}+b_{22}+ \varepsilon b_{11}) \\
            m_9 & = (a_{22}+a_{31})(b_{22}+\varepsilon b_{11}) \\
            m_{10} & = (a_{31}+\varepsilon a_{32})(b_{22} + \varepsilon b_{12}) \\
            \phantom{m_{11}} \\
            \phantom{m_{12}} \\
        \end{aligned}
        \]
        \[ \begin{aligned}
            c_{11} & = (m_1-m_2-m_3+m_4)/\varepsilon \\
            c_{12} & = (m_5-m_2)/\varepsilon \\
            c_{21} & = m_1+m_6 \\
            c_{22} & = -m_2+m_4-m_7+m_9 \\
            c_{31} & = (m_6-m_7-m_8+m_9)/\varepsilon  \\
            c_{32} & =  (m_{10}-m_7)/\varepsilon
        \end{aligned} \]
    \end{subfigure}
    \vspace{.2cm}
    \caption{Three algorithms for $\langle 3,2,2\rangle$ matrix multiplication using $12$, $11$, and $10$ scalar multiplications, respectively. The fast exact algorithm is obtained by applying Strassen's $\langle 2,2,2; 7\rangle$ routine to compute the upper $2 \times 2$ block of the output, while the APA algorithm is exact as $\varepsilon \rightarrow 0$.}
    \label{fig: alg_comp}
\end{figure*}

From \cref{fig: alg_comp}  we see that APA matrix multiplication can (at least in some cases) be accomplished with fewer scalar products than exact matrix multiplication. In fact, Hopcroft and Kerr \cite{Hopcroft_Kerr} showed that computing the exact product of a $3 \times 2$ and $2 \times 2$ matrix requires at least 11 scalar multiplications and hence that no exact algorithm is capable of beating the 10 multiplication APA routine of Bini et al.\cite{BINI1979234}. One way to understand this is through the lens of tensor decomposition; exact and APA algorithms for $m \times k$ by $k \times n$ matrix multiplication can be interpreted as exact and approximate decompositions of the $\langle m,k,n \rangle $ matrix multiplication tensor, which is defined as 
\begin{equation}\label{eqn: MM_tensor}
    {\mathcal M}_{\langle m,k,n \rangle} \coloneqq \sum_{i_1 = 1}^m \sum_{i_2 = 1}^k \sum_{i_3=1}^n E_{i_1i_2}^{(m,k)} \otimes E_{i_2i_3}^{(k,n)} \otimes E_{i_3i_1}^{(n,m)}.
\end{equation}
Here, $\otimes$ is the Kronecker product and  $E_{ij}^{(m,n)}$ is an $m \times n$ standard basis matrix, whose only nonzero entry is a one in row $i$ and column $j$.\footnote{Evaluating \eqref{eqn: MM_tensor} produces a \textit{matricized} version of the three-way tensor ${\mathcal M}_{\langle m,k,n \rangle}$, which has dimensions $mk \times kn \times mn$. The decomposition that defines ${\mathcal M}_{\langle m,k,n \rangle}$ corresponds to the classical algorithm for $\langle m,k,n \rangle$ matrix multiplication.}
Each scalar multiplication in an exact/APA algorithm contributes a rank-one tensor  (e.g., one term in the sum \eqref{eqn: MM_tensor}) to the corresponding decomposition of ${\mathcal M}_{\langle m,k,n \rangle}$. Accordingly, the number of scalar multiplications in an exact algorithm provides an upper bound on the \textit{rank} of ${\mathcal M}_{\langle m,k,n \rangle}$, that is, the minimum number of rank-one tensors needed to construct ${\mathcal M}_{\langle m,k,n \rangle}$. The same can be said of an APA algorithm and the \textit{border} rank of ${\mathcal M}_{\langle m,k,n \rangle}$, which corresponds to the minimum number of rank-one tensors in a sequence of decompositions whose limit is ${\mathcal M}_{\langle m,k,n\rangle}$. In this language, the above example can be summarized as follows:\ the rank and border rank of ${\mathcal M}_{\langle 3,2,2 \rangle}$ are 11 and 10 respectively,\footnote{Technically, the algorithm of Bini et al.\ \cite{BINI1979234} only establishes that the border rank is bounded above by 10. That this is in fact the true border rank was shown by Conner, Harper, and Landsberg \cite{Conner_Harper_Landsberg_2023}. } and corresponding decompositions can be obtained from the algorithms in \cref{fig: alg_comp}. 
For more background on tensor (border) rank and the notation used here, see \cite{Kolda_Bader} and \cite[Section 2]{Benson_Ballard}. \\
\indent  APA routines for small matrices can be used recursively as done above for exact algorithms. This yields APA algorithms for arbitrary $A \in {\mathbb C}^{m \times k}$ and $B \in {\mathbb C}^{k \times n}$ of degree $dL$, where 
\begin{equation}\label{eqn: L}
    L  \le \left\{ \log_{m_0}(m), \log_{n_0}(n), \log_{k_0}(k) \right\} 
\end{equation}
is the number of recursive steps applied. Since the complexity of these algorithms depends primarily on the number of scalar multiplications in the base case routine -- and this is potentially smaller than what can be achieved with a fast exact algorithm, as demonstrated in \cref{fig: alg_comp} -- APA matrix multiplication can be, asymptotically, faster than exact matrix multiplication. Despite this, APA algorithms are typically not used in practice. One reason is apparent from \cref{fig: alg_comp}:\  APA matrix multiplication requires dividing by $\varepsilon$, which means we cannot hope to use it in the limit $\varepsilon \rightarrow 0$ to obtain the exact product $C = AB$ (at least in the presence of floating-point roundoff).  \\
\indent Nevertheless, as outlined by Bini in \cite{Bini1980}, an APA algorithm \textit{can} be used to obtain an exact product of two matrices. 
Suppose that we evaluate a degree $d$ APA algorithm $d+1$ times with distinct evaluation parameters $\varepsilon_1, \ldots, \varepsilon_{d+1}$. If $C_i$ is the output of the $i$-th evaluation, and moreover if $\alpha_1, \ldots, \alpha_{d+1}$ solve the Vandermonde system
\begin{equation}\label{eqn: vandermonde_system_d}
        \begin{pmatrix} 1 & 1 & \cdots & 1 \\ 
        \varepsilon_1 & \varepsilon_2 & \cdots & \varepsilon_{d+1} \\
       \vdots & \vdots & & \vdots \\
       \varepsilon_1^{d} & \varepsilon_2^{d} & \cdots & \varepsilon_{d+1}^{d}
        \end{pmatrix} \begin{pmatrix} \alpha_1 \\ \alpha_2 \\ \vdots \\ \alpha_{d+1} \end{pmatrix} = \begin{pmatrix} 1 \\ 0 \\ \vdots \\ 0 \end{pmatrix},
\end{equation} 
then by construction
\begin{equation}\label{eqn: exactify}
    \alpha_1C_1 + \cdots + \alpha_{d+1}C_{d+1} = \left( \sum_{i = 1}^{d+1} \alpha_i \right)AB + \left(\sum_{i=1}^{d+1} \alpha_i\varepsilon_i \right) R_1 + \cdots +  \left(\sum_{i=1}^{d+1} \alpha_i\varepsilon_i^d \right) R_d = AB.
\end{equation}
In the recursive setting, \eqref{eqn: L} implies that the number of evaluations necessary to obtain the exact product of arbitrary $A \in {\mathbb C}^{m \times k}$ and $B \in {\mathbb C}^{k \times n}$ has logarithmic dependence on the size of the matrices.\footnote{Note that making the base case routine exact via \eqref{eqn: exactify} and using that recursively would lead to a slower algorithm.} That is, the complexity of an APA matrix multiplication algorithm 
is $O(L \cdot (mkn)^{\log_{m_0k_0n_0}(t)})$, where $L$ must satisfy \eqref{eqn: L} with equality to get the improved exponent. \\
\indent The extra logarithmic factor here implies that, while the exponent $\omega = \log_{m_0k_0n_0}(t)$ may be lower than what can be achieved by  fast exact $\langle m_0,k_0,n_0 \rangle$ matrix multiplication, APA algorithms, evaluated enough times to get an exact output, can only offer practical speedups on (possibly galactically) large matrices. For this reason, they have historically been considered mainly of theoretical interest, primarily as a source of upper bounds on the border rank of ${\mathcal M}_{\langle m,k,n \rangle}$. This is no doubt compounded by the fact that techniques for deriving APA algorithms are poorly documented in the literature. The most popular search algorithm applies some version of (nonlinear or alternating) least squares to the tensor decomposition problem. While least squares is similarly used to find fast exact algorithms (see e.g., \cite{Smirnov2013,BALLARD2019}), this approach has a unique limitation in the APA setting:\ the need to convert numerical decompositions into ones with functional entries. There is no known algorithmic way to do this. Indeed, Smirnov uses ``various choice heuristics" to obtain many of the current best-known APA algorithms in \cite{Smirnov2013}. Other open questions concern the practical details of using APA matrix multiplication. For example, recalling that fast exact matrix multiplication is well-known to be (norm-wise) stable \cite{fastLA}, can the same be said of the ``exactification" procedure introduced by Bini \cite{Bini1980}? Additionally, can we accelerate an APA algorithm using the same tools applied to fast exact matrix multiplication, specifically the alternative basis approach of Holtz et al.\ \cite{alternative_bases}? \\
\indent The goal of this work is to address all three of these questions. Our first contribution is a new search algorithm for APA matrix multiplication routines that can be easily run on small values of $m_0$, $k_0$, and $n_0$, requires minimal computing resources, and automatically generates decompositions with functional entries (e.g., of the form shown in \cref{fig: APA_rep}). Alongside this search algorithm, we also extend both the aforementioned numerical stability results and the alternative basis procedure to the APA setting. 
One strong motivation for this work  is the rapidly changing (and varied) set of demands placed on matrix multiplication kernels by machine learning applications (see the discussion in \cite{grading_the_blas}), which warrants revisiting every option in our algorithmic toolkit. Note in particular that an APA algorithm may only be slower than fast exact matrix multiplication in practice if we evaluate it enough times to get an exact product; in applications where only a low-precision approximate product is necessary -- and hence evaluating an APA algorithm $O(1)$ may be sufficient -- APA  matrix multiplication will be competitive. Indeed, Ballard et al.\ \cite{APA_NN} have already obtained promising results in this direction in the context of neural network training.

\subsection{Our Contributions}
The main contributions of this paper can be summarized as follows:
\begin{enumerate}
    \item In \cref{section: search} we present a new algorithm for generating APA matrix multiplication routines. The key insight is the following:\ by applying a symbolic De-Groote transformation and deleting one rank-one term in the corresponding tensor decomposition, a fast exact algorithm yields an APA routine with one fewer scalar multiplication (assuming the product of the factors in the deleted term is nilpotent). Using the recently introduced flip graph \cite{flip_graph,perminov2026} to find exact algorithms where this condition is satisfied, we obtain new APA routines for a variety of matrix sizes, including one that improves the previously best-known rank. 
    \item In \cref{section: alternative_basis}, we extend the alternative basis work of \cite{ALittleFaster2020,alternative_bases} to APA matrix multiplication. In doing so, we demonstrate that fast, symbolic basis transformations can be used to ``sparsify" APA algorithms and therefore improve the leading coefficient in their arithmetic complexities. 
    \item In \cref{section: stability}, we show that the approximate-to-exact procedure of Bini \cite{Bini1980} is stable (under some mild assumptions). In particular, we prove a weak forward stability bound for this procedure, in floating-point arithmetic and including the fast basis transformations from \cref{section: alternative_basis}, that matches the format of existing bounds for fast exact matrix multiplication \cite{fastMM,Schwartz2026_journal}. To our knowledge, these are the first stability bounds of any kind for APA matrix multiplication. We also comment on the stability of evaluating an APA algorithm a fixed number of times.  
    \item Finally, \cref{section: examples} presents a handful of numerical examples that validate the theoretical results from the previous sections. 
\end{enumerate}

Our results are captured by \cref{tab:results_summary,tab:results_summary2}, which show, respectively, the best ranks achieved in our search (vs.\ existing results in the literature) and  improvements to the leading coefficient of APA algorithms obtained by sparsifying with alternative bases. 

\renewcommand{\arraystretch}{1.2}
\begin{table}[t]
    \centering
    \begin{tabular}{|c|c|c|c|}
        \hline 
          \textbf{Size} & \textbf{Border Rank} & \textbf{Terms in Remainder} & \textbf{Exponent}   \\
        $\langle m_0,k_0,n_0\rangle $ & (vs.\ \cite{Smirnov2013,Smirnov2015,Smirnov2014}) & $i : R_i \not\equiv 0$ & $3\log_{m_0k_0n_0}(t)$ \\
        \hline
        $\langle 2,2,3\rangle $ & 10 (10) & 3, 6 & 2.7799  \\
        $\langle 2,3,3\rangle$ & 15 (14) & 3, 6 & 2.8108 \\
        $\langle 2,2,5\rangle $ & 17 (16) & 3, 4, 5, 6, 7, 8, 9 & 2.8372 \\
        $\langle 2,2,6\rangle$ & 20 (20) & 3, 4 & 2.8279  \\
        $\langle 3,3,3\rangle $ & 21 (20) &  1, 2, 3, 4, 5 & 2.7712  \\
        $\langle 2,3,5\rangle $ & 24 (23) & 3, 4, 5, 6, 7, 8, 9 & 2.8032  \\
        $\langle 2,4,4\rangle $ & 25 (24) & 3, 6 & 2.7863  \\
        $\langle 2,2,8\rangle $ & 27 (26) & 3, 4 & 2.8529  \\
        $\langle 3,3,4\rangle $ & 27 (27) & 3, 4, 5, 6, 7, 8, 9 & 2.7592  \\
        $\langle 2,3,6\rangle $ & 29 (28) & 3, 4, 6, 7 & 2.8190 \\
        $\langle 2,4,5\rangle $ & \textbf{31} (32) & 3, 4, 5, 6, 7, 8 & 2.7927  \\
        $\langle 2,3,7\rangle $ & 34 (33) & 3, 4, 5, 6 & 2.8304  \\
        $\langle 3,3,5\rangle $ & 35 (34) & 3, 4, 5, 6, 7, 8, 9 & 2.8019  \\
        $\langle 2,4,6\rangle $ & 37 (37) & 3, 4, 6, 7, 9, 10, 12, 13 & 2.7983 \\
        $\langle 2,5,5\rangle $ & 39 (39) & 3, 6 & 2.8095  \\
        $\langle 2,5,6\rangle $ & 46 (46) & 3, 4, 6, 7, 9, 10 & 2.8053  \\
        $\langle 3,3,7\rangle $ & 48 (48) & 3, 4, 5, 6, 7, 8, 9, 11 & 2.8031 \\
        \hline 
    \end{tabular}
    \caption{The best border rank $t$ achieved by the search algorithm discussed in \cref{section: search} for a variety of dimension triplets $\langle m_0,k_0,n_0 \rangle$. For reference, we list in parentheses the best border rank attained by Smirnov in \cite{Smirnov2013,Smirnov2015,Smirnov2014}. 
    For each size, we also describe the remainder for one $ \langle m_0,k_0,n_0;t \rangle$ scheme. Corresponding encoding matrices can be found at \href{https://github.com/nliskij/border-rank}{https://github.com/nliskij/border-rank}.}\label{tab:results_summary}
\end{table}
\renewcommand{\arraystretch}{1}

\begin{table}[p]   
    \centering
    \renewcommand{\arraystretch}{1.3}
    \begin{minipage}{15cm}
        \centering
        \begin{tabular}{|c|c|c|c|} 
        \hline 
          \textbf{APA Algorithm} & 
          \multicolumn{2}{c|}{\textbf{Leading Coefficient}} & \textbf{Saving}  \\
        $\langle m_0,k_0,n_0;t\rangle $ & 
        \multicolumn{1}{c}{Original} & \multicolumn{1}{c|}{Sparsified}& \\
        \hline
        $\langle 2,2,3; 10\rangle $ & 
        24.583 & 6.224  & 74.68\%\\
        \hline
        $\langle 2,3,3;15\rangle$ & 
        26.944 & 6.234 &76.86\% \\
        \hline
        $\langle 2,2,5;17\rangle $ & 
        35.242 & 9.366 & 73.42\%\\
        \hline 
        $\langle 2,2,6;20\rangle$ & 
        30.125 & 5.448 & 81.92\%\\
        \hline 
        $\langle 3,3,3; 20\rangle $ & 
        26.364 & 15.182  &42.41\% \\
        \hline 
        $\langle 3,3,3;21\rangle $ & 
        40.833 & 7.083 & 82.65\%\\
        \hline 
        $\langle 2,3,5;24\rangle $ & 
        45.587 & 7.048 & 84.54\%\\
        \hline
        $\langle 2,4,4;25\rangle $ & 
        56.294 & 9.808 & 82.58\%\\
        \hline 
        $\langle 2,2,8; 27\rangle $ & 
        25.249 & 5.259 & 79.17\%\\
        \hline
        \multirow{3}{*}{$\langle 3,3,4; 27 \rangle $} & 
        60.678 & 7.967 & 86.87\%\\
        \cline{2-4}
        & 59.878 & 7.967  & 86.69\%\\
          \cline{2-4}
        & 60.256 & 8.089  & 86.58\%\\
        \hline
        $\langle 2,3,6; 29\rangle $ & 
        48.227 & 10.041 & 79.18\%\\
        \hline 
        \multirow{4}{*}{$\langle 2,4,5;31\rangle $} & 
        50.218 & 7.546 & 84.97\%\\
         \cline{2-4}
         & 100.409 & 8.419  & 91.62\%\\
          \cline{2-4}
        & 64.990 & 7.638  & 88.25\%\\
         \cline{2-4}
        & 61.622 & 7.597  & 87.67\%\\
        \hline
        $\langle 2,3,7;34 \rangle $ & 
        80.157 & 8.606 & 89.26\%\\
        \hline 
        $\langle 3,3,5; 35\rangle $ & 
        49.627 & 7.537 & 84.81\%\\
        \hline
        \multirow{3}{*}{$\langle 2,4,6; 37\rangle $} & 
        89.476 & 7.925 & 91.14\%\\
         \cline{2-4}
        & 
        124.032 & 9.336  & 92.47\%\\
         \cline{2-4}
        & 
        83.858 & 8.095  & 90.35\%\\
        \hline
        $\langle 2,5,5; 39\rangle $ & 
        128.862 & 8.003 & 93.79\%\\
        \hline
        \multirow{3}{*}{$\langle 2,5,6; 46\rangle $} & 
        89.192 & 13.329 & 85.06\%\\
         \cline{2-4}
         & 
         104.323 & 15.395  & 85.24\%\\
          \cline{2-4}
        & 
        87.869 & 13.424  & 84.72\%\\
        \hline 
        \multirow{2}{*}{$\langle 3,4,5;46 \rangle $} & 
        51.516 & 7.801  & 84.86\%\\
         \cline{2-4}
        & 
        77.421 & 8.783  & 88.66\%\\
        \hline
        $\langle 3,3,7; 48 \rangle $ & 
        78.339 & 8.462 & 89.20\%\\
        \hline
        \multirow{2}{*}{$\langle 3,3,8;54 \rangle $} & 
        136.844 & 15.028  & 89.02\%\\
         \cline{2-4}
        & 
        109.089 & 14.278  & 86.91\%\\
        \hline 
    \end{tabular}
    \captionof{table}{Improvements to the leading coefficients of a variety of APA algorithms obtained by applying the alternative basis approach discussed in \cref{section: alternative_basis}. With the exception of the $\langle 3,3,3;20\rangle$ scheme -- which is due to Smirnov \cite{Smirnov2013} -- all algorithms were discovered via the search described in \cref{section: search} and can be found in our repository (see \href{https://github.com/nliskij/border-rank}{https://github.com/nliskij/border-rank}).}\label{tab:results_summary2}
    \end{minipage}
\end{table}
\renewcommand{\arraystretch}{1}

\subsection{Notation and Conventions}
In the remainder of the paper, the matrices $A$, $B$, and their product $C = AB$ are assumed to be complex. As mentioned above, we use $\langle m_0,k_0,n_0 \rangle$ to refer to $m_0 \times k_0$ by $k_0 \times n_0$ matrix multiplication (either exact or approximate) and $\langle m_0, k_0, n_0; t \rangle$ to refer to algorithms for this problem that use $t$ scalar multiplications. Note that $\langle m_0,k_0,n_0 \rangle$ is equivalent to $\langle m_0, n_0,k_0 \rangle$, $\langle k_0, m_0, n_0 \rangle$, $ \langle k_0, n_0, m_0 \rangle$, $\langle n_0, m_0, k_0 \rangle$, and $\langle n_0, k_0, m_0 \rangle$ in the sense that an algorithm for one can be adapted to an algorithm for any of the others (with the same number of scalar multiplications but possibly different additive complexity) \cite{Hopcroft_duality}. \\
\indent It is often convenient to talk about a specific $\langle m_0,k_0,n_0;t \rangle$ algorithm, APA or exact, in terms of a trio of \textit{encoding matrices} $U \in {\mathbb C}^{m_0k_0 \times t}$, $V \in {\mathbb C}^{k_0n_0 \times t}$, and $W \in {\mathbb C}^{m_0n_0 \times t}$, which record the dependence of each scalar multiplication on the entries of $A$ and $B$ (in the case of $U$ and $V$) and the dependence of each entry of $C$ on the scalar multiplications (in the case of $W$). To be specific, these encoding matrices satisfy
\begin{equation}\label{eqn: encoding_relation}
     c = W(U^Ta \odot V^Tb),
\end{equation}
where $\odot$ is the Hadamard product and $a$, $b$, and $c$ are (row-wise) vectorized copies of $A$, $B$, and $C$.\footnote{If, for example, $A$, $B$, and $C$ are the $3 \times 2$ and $2 \times 2$ matrices pictured in \cref{fig: alg_comp}, we have $a^T = (a_{11}, a_{12}, a_{21}, a_{22}, a_{31},a_{32})$, $b^T = (b_{11},b_{12},b_{21},b_{22})$, and $c^T = (c_{11},c_{12},c_{21},c_{22},c_{31},c_{32})$.} In the APA case, these encoding matrices are dependent on the evaluation parameter $\varepsilon$ and are therefore denoted $U(\varepsilon)$, $V(\varepsilon)$ and $W(\varepsilon)$. An example of these matrices for the $\langle 3,2,2;10\rangle$ APA algorithm of Bini et al.\ \cite{BINI1979234} is given in \cref{fig: APA_rep}. 

\begin{figure*}[t!]
    \centering
    \begin{subfigure}[t]{\textwidth}
        \centering 
        \[ U(\varepsilon) = \begin{pmatrix} 
         \varepsilon & 0 & 0&  0& \varepsilon & 0 & 0& 0& 0 & 0\\
         0 &  1& 0& 1 &  1 &  0& 0& 0& 0& 0\\
         1 &  0& 1& 1 &  0 &  0& 0& 0& 0& 0\\
         0 & 0 & 0& 0 &  0 & 1 & 0& 1& 1& 0\\
         0 & 0 & 0& 0&  0 & 0 & 1& 0& 1& 1\\
         0 & 0 & 0 & 0 & 0 & \varepsilon & 0& 0& 0& \varepsilon\\
        \end{pmatrix}  \; \; \; \; \; V(\varepsilon) =
        \begin{pmatrix} 
        1 & 0 & 1 & 0 & 0 & 0 & 0 & \varepsilon & \varepsilon & 0 \\
        0 & 1 & 1 & 1 & 1 & 0 & 0 & 0 & 0 & \varepsilon \\
        0 & 0 & \varepsilon & \varepsilon & 0 & 1 & 0 & 1 & 0 & 0 \\
        0 & 0 & 0 & 0 & \varepsilon & 0 &  1 & 1 & 1 & 1
        \end{pmatrix}
        \]
        \[ W(\varepsilon) = \begin{pmatrix}
            \varepsilon^{-1} & - \varepsilon^{-1} & - \varepsilon^{-1} & \varepsilon^{-1} & 0 & 0 & 0 & 0 & 0 & 0 \\
            0 & -\varepsilon^{-1} & 0 & 0 & \varepsilon^{-1} &  0 & 0 & 0 & 0 & 0 \\
            1 & 0 & 0 & 0 & 0 & 1 & 0 & 0 & 0 & 0 \\
            0 & -1 & 0 & 1 & 0 & 0 & -1 & 0 & 1 & 0 \\
            0 & 0 & 0 & 0   & 0 & \varepsilon^{-1} & - \varepsilon^{-1} & - \varepsilon^{-1} & \varepsilon^{-1} & 0 \\
            0 & 0 & 0 & 0 & 0 & 0 & -\varepsilon^{-1} & 0 & 0 & \varepsilon^{-1}
        \end{pmatrix}\]
        \caption{Parameterized encoding matrices}
    \end{subfigure}
    
    \begin{subfigure}[t]{\textwidth}
        \centering
        \[ \aligned {\mathcal M}_{\langle 3,2,2\rangle } \approx \begin{pmatrix} \varepsilon & 0 \\ 1 & 0 \\ 0 & 0 \end{pmatrix} &\otimes \begin{pmatrix} 1 & 0 \\ 0 & 0 \end{pmatrix} \otimes \begin{pmatrix} \varepsilon^{-1} & 1 & 0  \\ 0 & 0  & 0 \end{pmatrix} + \begin{pmatrix} 0 & 1 \\ 0 & 0 \\ 0 & 0 \end{pmatrix} \otimes \begin{pmatrix} 0 & 1 \\ 0 & 0 \end{pmatrix} \otimes \begin{pmatrix} -\varepsilon^{-1}& 0 & 0 \\
        -\varepsilon^{-1} & -1 & 0 \end{pmatrix} \\
        & + \begin{pmatrix} 0 & 0 \\ 1 & 0 \\ 0 & 0 \end{pmatrix} \otimes \begin{pmatrix} 1 & 1 \\ 
        \varepsilon & 0 \end{pmatrix} \otimes \begin{pmatrix} - \varepsilon^{-1} & 0 & 0 \\
        0 & 0 & 0 \end{pmatrix} + \begin{pmatrix}
            0 & 1 \\ 1 & 0 \\ 0 & 0 
        \end{pmatrix} \otimes \begin{pmatrix}
            0 & 1 \\ \varepsilon & 0 
        \end{pmatrix} \otimes \begin{pmatrix} 
        \varepsilon^{-1} & 0 & 0 \\ 0 & 1 & 0 
        \end{pmatrix} \\
        &  + \begin{pmatrix} \varepsilon & 1 \\ 0 & 0 \\ 0 & 0 \end{pmatrix} \otimes \begin{pmatrix} 0 & 1 \\ 0 & \varepsilon \end{pmatrix} \otimes \begin{pmatrix} 0 & 0 & 0 \\ \varepsilon^{-1} & 0 & 0 \end{pmatrix} + \begin{pmatrix}
            0 & 0 \\ 0 & 1 \\ 0 & \varepsilon 
        \end{pmatrix} \otimes \begin{pmatrix}
            0 & 0 \\ 1 & 0 \end{pmatrix} \otimes \begin{pmatrix}
            0 & 1 & \varepsilon^{-1} \\
            0 & 0 & 0 
        \end{pmatrix} \\
        &  + \begin{pmatrix} 0 & 0 \\ 0 & 0 \\ 1 & 0 \end{pmatrix} \otimes \begin{pmatrix}
            0 & 0 \\ 0 & 1 
        \end{pmatrix} \otimes \begin{pmatrix}
        0 & 0 & -\varepsilon^{-1} \\
        0 & -1 & -\varepsilon^{-1} 
        \end{pmatrix} + \begin{pmatrix} 0 & 0 \\ 0 & 1 \\ 0 & 0 \end{pmatrix} \otimes \begin{pmatrix} \varepsilon & 0 \\ 1 & 1 \end{pmatrix} \otimes \begin{pmatrix} 0 & 0 & -\varepsilon^{-1} \\
        0 & 0 & 0 \end{pmatrix} \\
        & + \begin{pmatrix} 0 & 0 \\ 0 & 1 \\ 1 & 0 \end{pmatrix} \otimes \begin{pmatrix}
            \varepsilon & 0 \\ 0 & 1 
        \end{pmatrix} \otimes \begin{pmatrix}
            0 & 0 & \varepsilon^{-1} \\  0 & 1 & 0 
        \end{pmatrix} + \begin{pmatrix}
            0 & 0 \\ 0 & 0 \\ 1 & \varepsilon
        \end{pmatrix} \otimes \begin{pmatrix} 0 & \varepsilon \\ 0 & 1 \end{pmatrix} \otimes \begin{pmatrix} 0 & 0 & 0 \\ 0 & 0 & \varepsilon^{-1} \end{pmatrix}
        \endaligned \]
        \caption{Tensor decomposition}
    \end{subfigure}%
    
    \caption{Two ways of representing the $\langle 3,2,2;10\rangle$ APA matrix multiplication algorithm of Bini et al.\ \cite{BINI1979234} (see \cref{fig: alg_comp}).}\label{fig: APA_rep}
\end{figure*}

\indent This example also highlights the connection between the encoding matrices of an exact (or APA) matrix multiplication algorithm and the corresponding tensor decomposition of ${\mathcal M}_{\langle m_0,k_0,n_0\rangle}$. In particular, we have
\begin{equation}
    {\mathcal M}_{\langle m_0,k_0,n_0 \rangle}  = \sum_{j = 1}^t U_j \otimes V_j \otimes W_j,
\end{equation}
where $U_j \in {\mathbb C}^{m_0 \times k_0}$, $V_j \in {\mathbb C}^{k_0 \times n_0}$, and $W_j^T \in {\mathbb C}^{m_0 \times n_0}$ are matricized copies of the columns of $U$, $V$, and $W$, respectively. Importantly, this identity is invariant under \textit{De-Groote transformations} \cite{DEGROOTE19781}, in which $U_j$, $V_j$, and $W_j$ are replaced by $R^{-1}U_jS$, $S^{-1}V_jT$ and $T^{-1}W_jR$ for any set of basis transformations $R$, $S$, and $T$. \\
\indent Note that the potential practical efficiency of an $\langle m_0,k_0,n_0; t \rangle$ matrix multiplication algorithm can be gleaned from properties of its encoding matrices. For example, a \textit{real} algorithm, whose encoding matrices are all real valued, in practice, is likely faster than a \textit{complex} one (of comparable sparsity), which requires complex arithmetic even to multiply real matrices. Similarly, an algorithm with sparse encoding matrices will likely be faster in practice -- hence the reference above to ``sparsifying" an APA algorithm, which really refers to sparsifying its encoding matrices (see \cite{BeniaminiSchwartzFastDecomposition}). \\
\indent Throughout, we use ${\mathbb C}(\varepsilon)$, ${\mathbb C}[\varepsilon]$ and ${\mathbb C}[\varepsilon,\varepsilon^{-1}]$, to denote, respectively, the field of rational functions, the ring of polynomials, and the ring of Laurent polynomials in the variable $\varepsilon$ with complex coefficients. We also use Matlab notation when convenient, e.g., $A(:,i)$ and $A(i,:)$  refer to a column/row of the matrix $A$. This extends to arbitrary index sets $I$. That is, $A(:,I)$ collects the columns of $A$ with indices $i \in I$.

\begin{remark}\label{rem: degree}
   In this paper, we assume that the remainder of an APA algorithm, meaning the difference between its output and the true product, takes the form $\varepsilon R_1 + \cdots + \varepsilon^d R_d$, where $R_i$ is assumed to be not identically zero for all $1 \leq i \leq d$. This does not necessarily hold for all APA algorithms, which may have remainders that skip terms (see \cref{tab:results_summary}). This is relevant for implementation purposes since, in practice, the number of evaluations necessary to recover the exact product $AB$ from an APA algorithm is not determined by the largest exponent in its remainder but rather by the number of nonzero terms it contains. 
   The catch here is that the number of terms in the remainder for a recursive implementation of a $\langle m_0,k_0,n_0 \rangle$ APA algorithm (that performs $L$ recursive steps) is \textit{not} necessarily $L$ times the number of terms in the base-case algorithm, and in fact in most cases it grows to $dL$ as $L$ increases.  For this reason 
   we stick with $\varepsilon R_1 + \cdots + \varepsilon^dR_d$ in our analysis. 
\end{remark}



\section{A New Search Algorithm for APA Schemes}\label{section: search}
We start by presenting a new algorithm for deriving APA schemes. As mentioned above, existing methods for generating APA algorithms, e.g., those of Smirnov and Bini, take (approximate) floating-point decompositions of the matrix multiplication tensor and convert them into symbolic ones. Our approach flips this idea; we start from an exact algorithm and produce an APA one by applying a symbolic De-Groote transformation and subsequently deleting terms that are $O(\varepsilon)$. That is, if $\sum_{j=1}^r U_j \otimes V_j \otimes W_j$ is a decomposition of the matrix multiplication tensor ${\mathcal M}_{\langle m,k,n \rangle}$ for $U_j \in {\mathbb C}^{m \times k}$, $V_j \in {\mathbb C}^{k \times n}$ and $W_j \in {\mathbb C}^{n \times m}$ (recall \cref{fig: APA_rep}), then 
\begin{equation}
    {\mathcal M}_{\langle m,k,n\rangle} = \sum_{j = 1}^r R^{-1}U_jS \otimes S^{-1}V_jT\otimes T^{-1}W_jR
\end{equation}
for any invertible $R$, $S$, and $T$. If additionally $R$, $S$, and $T$ have entries in ${\mathbb C}(\varepsilon)$, we obtain a symbolic decomposition that is exact for any value of $\varepsilon$. Deleting the terms for which every nonzero entry of $R^{-1}U_jS \otimes S^{-1}V_jT \otimes T^{-1}W_jR$ is divisible by a positive power of $\varepsilon$ yields an APA algorithm (possibly of lower rank) that can be made arbitrarily accurate as $\varepsilon \rightarrow 0$. \\
\indent The difficulty of this approach lies in finding an appropriate exact algorithm to use as a starting point (as we will see, not just any one will do). One way of doing this is by using recent work on generating exact decompositions \cite{flip_graph,perminov2026}, which allows us to find exact algorithms satisfying the necessary conditions to produce lower rank APA algorithms. In this section, we provide an overview of this approach, including pseudocode and a motivating example, and discuss the results of our search. 

\subsection{Sufficient Conditions for Rank Improvements}


Suppose that an exact algorithm for matrix multiplication is represented by the tensor decomposition ${\mathcal M}_{\langle m,k,n \rangle} = \sum_{j=1}^r U_j \otimes V_j \otimes W_j$. This algorithm is a good starting point for the approach outlined above if the following hold: (1) the product $U_jV_jW_j$ is nilpotent for at least one $j$ (required) and (2) $r$ is small (desired but not required). The second criterion is one of practical concern; since we obtain APA algorithms by removing terms from an exact one, we stand to obtain the best result by starting from an exact algorithm that has minimal rank. The first criterion, on the other hand, guarantees that a term can be deleted at all. We verify this below. 

\begin{thm}\label{thm: remove_one}
Suppose that the product $UVW$ of $ U \in {\mathbb C}^{m \times k}$, $V \in {\mathbb C}^{k \times n}$ and $W \in {\mathbb C}^{n \times m}$ is nilpotent. In this case, there exist $ R \in {\mathbb C}(\varepsilon)^{m \times m}$, $ S \in {\mathbb C}(\varepsilon)^{k \times k}$, and $T \in {\mathbb C}(\varepsilon)^{n \times n}$, each invertible over ${\mathbb C}(\varepsilon)$, such that all entries of $R^{-1}US$,  $S^{-1}VT$, and $T^{-1}WR$ are polynomials in $\varepsilon$ with zero constant term. 
\end{thm}
\begin{proof}
    Let $F_0=G_0=H_0={0}$ and inductively define $ F_{j+1}$, $G_{j+1}$, and $H_{j+1}$ to be the pre-images of $G_j$, $H_j$, and $F_j$, respectively, under the linear transformations represented by $W$, $V$, and $U$. For example, $ F_1 = \ker(W)$, $G_1 = \ker(V)$, and $H_1 = \ker(U)$. We note two important facts about these subspaces:
    \begin{enumerate}
        \item They are nested -- e.g., $F_j \subseteq F_{j+1}$, $G_j \subseteq G_{j+1}$, and $H_j \subseteq H_{j+1}$. This can be shown inductively by first noting that $F_0 \subseteq F_1$ (since $x \in F_0$ satisfies $Wx = W0 = 0 \in G_0$), $G_0 \subseteq G_1$, and $H_0 \subseteq H_1$. Now, $F_j \subseteq F_{j+1}$ implies $H_{j+1} \subseteq H_{j+2}$,  $G_{j} \subseteq G_{j+1}\Rightarrow F_{j+1} \subseteq F_{j+2}$, and $H_{j} \subseteq H_{j+1}\Rightarrow G_{j+1} \subseteq G_{j+2}$. 
        \item $F_{3i}$, $G_{3i}$, and $H_{3i}$ are the null spaces of $(UVW)^i$, $(WUV)^i$, and  $(VWU)^i$, respectively. 
    \end{enumerate}
    If $UVW$ is nilpotent, then so are $VWU$ and $WUV$,\footnote{This follows from the fact that nilpotence is equivalent to the trace powers of $UVW$ being zero, and trace is invariant under cyclic permutations.} and therefore $\bigcup_j F_j = {\mathbb C}^m$, $ \bigcup_j G_j = {\mathbb C}^n$, and $\bigcup_j H_j = {\mathbb C}^k$. We can now define $R$, $S$, and $T$ as follows:\ in the case of $R$, take a basis $x_1, x_2, \ldots, x_m$ for ${\mathbb C}^{m}$ constructed successively from the filtration 
    $F_1 \subseteq F_2 \subseteq \cdots $ and define $R \in {\mathbb C}(\varepsilon)^{m\times m}$ to be the matrix satisfying $Rx = \varepsilon^ix$ if $x \in F_i \setminus F_{i-1}$. Define $S$ and $T$ analogously using $H_j$ and $G_j$. It remains to show that $R^{-1}US$, $S^{-1}VT$, and $T^{-1}WR$ have entries in $\varepsilon {\mathbb C}[\varepsilon]$. To do this, suppose that $y \in {\mathbb C}^k$ is one of the basis vectors obtained from $H_1 \subseteq H_2 \subseteq \cdots$, where $y \in H_i\setminus H_{i-1}$ for some $i$. By construction, we have $R^{-1}USy = \varepsilon^iR^{-1}Uy $ and moreover $Uy \in F_{i-1}$, so $Uy = \alpha_1x_1 + \cdots + \alpha_lx_l$ for the basis vectors $x_1, \ldots, x_l$ of ${\mathbb C}^m$ coming from $F_{i-1}$. But in this case, 
    \begin{equation}\label{eqn: prove_polynomial_entries}
        R^{-1}USy = \varepsilon^{i} R^{-1} \left( \alpha_1x_1 + \cdots + \alpha_lx_l \right)  = \varepsilon^i \left( \alpha_1 R^{-1}x_1 + \cdots + \alpha_l R^{-1}x_l \right) = \alpha_1 \varepsilon^{i - \beta_1} x_1 + \cdots + \alpha_l \varepsilon^{i - \beta_l}x_l,
    \end{equation}
    where $\beta_1, \ldots, \beta_l$ are the exponents satisfying $Rx_j = \varepsilon^{\beta_j} x_j$. Since each $x_j$ belongs to $F_{i-1}$, we have $\beta_j \leq i-1$ for all $j$. Hence, each exponent $i - \beta_j$ in \eqref{eqn: prove_polynomial_entries} is at least one, which means $R^{-1}USy$ has entries in $\varepsilon {\mathbb C}[\varepsilon]$. Since this applies to any basis vector $y \in {\mathbb C}^k$, and letting $Y \in {\mathbb C}^{k \times k}$ be the matrix containing that basis, we conclude that $R^{-1}USY$ has entries in $\varepsilon {\mathbb C}[\varepsilon]$, which is sufficient to show the same for the entries of $R^{-1}US$, since multiplying on the right by $Y^{-1}$, which has constant entries, won't introduce any constant terms. The same argument can be used to show that $S^{-1}VT$ and $T^{-1}WR$ also have entries in $\varepsilon {\mathbb C}[\varepsilon]$.
\end{proof}


The key idea behind our search algorithm is now captured by the following result. The terms corresponding to $j \leq l$, which can be transformed to have entries that are polynomials in $\varepsilon$ with no constant term -- a property that carries over to the corresponding tensor product $R^{-1}U_jS \otimes S^{-1}V_j T \otimes T^{-1}W_jR$ -- are precisely those we referred to as $O(\varepsilon)$ above. 

\begin{thm}\label{thm: rank_improvement}
Let ${\mathcal M}_{\langle m,k,n \rangle} = \sum_{j=1}^{r} U_j \otimes V_j \otimes W_j$ be a tensor decomposition obtained from an exact $\langle m,k,n \rangle$ matrix multiplication algorithm. If there exist $ R \in {\mathbb C}(\varepsilon)^{m \times m}$, $ S \in {\mathbb C}(\varepsilon)^{k \times k}$, and $T \in {\mathbb C}(\varepsilon)^{n \times n}$, each invertible over ${\mathbb C}(\varepsilon)$, such that $R^{-1}U_jS$, $S^{-1}V_jT$, and $T^{-1}W_jR$ have entries in $\varepsilon {\mathbb C}[\varepsilon]$ for all $1 \leq j \leq l$, then \[\sum_{j=l+1}^r R^{-1}U_jS \otimes S^{-1}V_jT \otimes T^{-1}W_jR \] 
is an approximate, symbolic decomposition of ${\mathcal M}_{\langle m,k,n \rangle}$ that corresponds to an APA algorithm of rank $r-l$. In particular, the border rank of ${\mathcal M}_{\langle m,k,n \rangle}$ is bounded by $r - l$. 
\end{thm}

\cref{thm: rank_improvement} implies that we can remove more than one nilpotent term from an exact decomposition provided they can be made $O(\varepsilon)$ for the same (symbolic) basis transformations. Despite this, \cref{thm: remove_one} only provides a criterion for removing one term at a time, which is ultimately limited by the fact that, even if we have two terms that correspond to nilpotent products, a set of transformations that makes one $O(\varepsilon)$ may not necessarily make the other $O(\varepsilon)$ as well. Hence, we need other tools for identifying when more than one term is removable 
The following proposition gives an easy necessary condition. 

\begin{prop}
    If $ R^{-1} U_j S$, $S^{-1}V_j T$, and $T^{-1} W_jR $ have entries in $ \varepsilon \mathbb{C}[\varepsilon] $ for all $ 1 \leq j \leq l$ and $R$, $S$, and $T$  as in \cref{thm: rank_improvement}, then  
    \[ \text{\normalfont Trace}(U_{j^u_1}V_{j^v_1}W_{j^w_1} \cdots U_{j^u_p}V_{j^v_p} W_{j^w_p} ) = 0
    \]
    for all sequences of indices $j^u_1,j_1^v,j_1^w,\ldots,j^u_p,j_p^v, j_p^w \leq l$ with $p \geq 1$.
\end{prop}
\begin{proof}
    Since trace is basis invariant, all expressions of the form $ \text{Trace}(U_{j^u_1} V_{j^v_1} W_{j^w_1} \dots U_{j^u_p} V_{j^v_p} W_{j^w_p}) $ are invariant under De-Groote transformations. By assumption, we have a De-Groote transformation under which the given expression is a polynomial in $\varepsilon$ with zero constant term. But the original expression does not depend on $\varepsilon$, so it follows that it must be zero.
\end{proof}

Importantly, this necessary condition is also a  sufficient one. Moreover, we do not need to check arbitrary products of the $U_jV_jW_j$ (which of course is not possible in practice). This follows from classical results in invariant theory, namely Hilbert's basis theorem and Hilbert-Mumford, though an effective bound for $p$ is given by Le Bruyn and Procesi \cite[Theorem 1]{LeBruynProcesi1990}. In the literature, this theorem is usually stated in the language of quivers, which we reformulate in our context.\footnote{The relevant quiver for a rank $ r $ exact decomposition of ${\mathcal M}_{\langle m,k,n \rangle}$ is a three-vertex directed cycle with $ r $ copies of each arrow. We construct a representation of this quiver by assigning $ \mathbb{C}^m, \mathbb{C}^k, \mathbb{C}^n $ to the vertices and $ U_j, V_j, W_j$ to the arrows.}

\begin{thm}\label{thm: remove_more_than_one}
    Suppose that $ \text{\normalfont Trace}(U_{j^u_1} V_{j^v_1} W_{j^w_1} \cdots U_{j^u_p} V_{j^v_p} W_{j^w_p}) = 0 $ for all sequences of indices $ j^u_1, j^v_1,j^w_1, \ldots$
    $ \ldots, j^u_p, j^v_p,j^w_p \leq l $ with $ p \leq (m + k + n)^2 $. Then there exist $ R \in {\mathbb C}(\varepsilon)^{m \times m}$, $S \in {\mathbb C}(\varepsilon)^{k \times k}$, and $T \in {\mathbb C}(\varepsilon)^{n \times n}$, each invertible over ${\mathbb C}(\varepsilon)$, such that $ R^{-1}U_jS$, $S^{-1}V_jT$, and $T^{-1}W_jR$ have entries in $ \varepsilon \mathbb{C}[\varepsilon] $ for all $ 1 \leq j \leq l$.
\end{thm}

Unfortunately, unlike \cref{thm: remove_one}, the proof of this result is not constructive, meaning we cannot use it to find $R$, $S$, and $T$ algorithmically. As a result, it is typically easier to remove one term at a time. Possible generalizations capable of eliminating multiple terms simultaneously are discussed in \cref{subsection: generalization}.

\subsubsection{$\langle 2,2,3 \rangle$ Example}
To see our high-level approach in action, consider an exact $\langle2,2,3;11 \rangle$ algorithm represented by the tensor decomposition 
\begin{equation}\label{eqn: 223_example}
	\begin{aligned}
		{\mathcal M}_{\langle 2,2,3 \rangle} = 
		& \begin{pmatrix}
			1&1\\
			-1&-1
		\end{pmatrix} \otimes \begin{pmatrix}
			0&0&0\\
			0&1&0
		\end{pmatrix} \otimes 
		\begin{pmatrix}
			0&0\\
			1&0\\
			0&0
		\end{pmatrix} +
		\begin{pmatrix}
			0&0\\
			1&0
		\end{pmatrix} \otimes 
		\begin{pmatrix}
			1&0&-1\\
			0&0&-1
		\end{pmatrix} \otimes 
		\begin{pmatrix}
			1&1\\
			0&0\\
			0&0
		\end{pmatrix} \\
		& +
		\begin{pmatrix}
			1&-1\\
			0&0
		\end{pmatrix} \otimes 
		\begin{pmatrix}
			0&0&0\\
			0&0&1
		\end{pmatrix} \otimes 
		\begin{pmatrix}
			0&0\\
			0&0\\
			-1&-1
		\end{pmatrix} + 
		\begin{pmatrix}
			1&0\\
			-1&-1
		\end{pmatrix} \otimes
		\begin{pmatrix}
			0&1&-1\\
			0&-1&-1
		\end{pmatrix} \otimes 
		\begin{pmatrix}
			0&1\\
			1&1\\
			0&1
		\end{pmatrix} \\
		& + 
		\begin{pmatrix}
			1&-1\\
			-1&0
		\end{pmatrix} \otimes 
		\begin{pmatrix}
			0&0&0\\
			1&0&0
		\end{pmatrix} \otimes
		\begin{pmatrix}
			-1&0\\
			0&0\\
			0&1
		\end{pmatrix} + 
		{\color{red}
        \begin{pmatrix}
			0&0\\
			0&1
		\end{pmatrix} \otimes 
		\begin{pmatrix}
			0&1&-1\\
			1&-1&-1
		\end{pmatrix} \otimes 
		\begin{pmatrix}
			0&1\\
			0&0\\
			0&1
		\end{pmatrix}} \\
		& +
	       \begin{pmatrix}
			0&0\\
			1&1
		\end{pmatrix}\otimes 
		\begin{pmatrix}
			0&1&-1\\
			0&0&-1
		\end{pmatrix} \otimes 
		\begin{pmatrix}
			0&0\\
			1&1\\
			0&0
		\end{pmatrix} +
		\begin{pmatrix}
			1&0\\
			-1&0
		\end{pmatrix}\otimes 
		\begin{pmatrix}
			0&1&0\\
			0&-1&0
		\end{pmatrix} \otimes 
		\begin{pmatrix}
			1&-1\\
			0&-1\\
			0&-1
		\end{pmatrix} \\
		&+
		\begin{pmatrix}
			1&-1\\
			-1&1
		\end{pmatrix} \otimes 
		\begin{pmatrix}
			0&0&0\\
			1&0&-1
		\end{pmatrix} \otimes 
		\begin{pmatrix}
			0&0\\
			0&0\\
			0&-1
		\end{pmatrix} +
		\begin{pmatrix}
			1&0\\
			-1&0
		\end{pmatrix} \otimes 
		\begin{pmatrix}
			1&-1&-1\\
			1&1&-1
		\end{pmatrix} \otimes 
		\begin{pmatrix}
			1&0\\
			0&0\\
			0&0
		\end{pmatrix} \\
		& +
		\begin{pmatrix}
			1&0\\
			0&0
		\end{pmatrix} \otimes 
		\begin{pmatrix}
			0&0&1\\
			0&0&1
		\end{pmatrix} \otimes 
		\begin{pmatrix}
			1&1\\
			1&1\\
			1&1
		\end{pmatrix}.
	\end{aligned}
\end{equation}
Notice that the sixth term (highlighted in red) has a nilpotent product:
\begin{equation}
        \aligned
        U_6V_6W_6 &= \begin{pmatrix} 0 & 0 \\ 0 & 1 \end{pmatrix} \begin{pmatrix} 0 & 1 & -1 \\ 1 & -1 & -1 \end{pmatrix} \begin{pmatrix} 0 & 1 \\ 0 & 0 \\ 0 & 1 \end{pmatrix} = \begin{pmatrix} 0 & 0 \\ 0 & 0 \end{pmatrix} .
        \endaligned
\end{equation}
Hence, we can follow the proof of \cref{thm: remove_one} to construct symbolic basis transformations
\begin{equation}
	R = \begin{pmatrix}
		\varepsilon&0\\
		0&\varepsilon^3
	\end{pmatrix},
	\qquad
	S =
	\begin{pmatrix}
		\varepsilon&0\\
		0& \varepsilon^4
	\end{pmatrix}, \quad \text{and} \quad
	T= \begin{pmatrix}
		2&1&0\\
		1&1&0\\
		1&0&1
	\end{pmatrix}
    \begin{pmatrix}
		\varepsilon&0&0\\
		0&\varepsilon^2&0\\
		0&0&\varepsilon^5
	\end{pmatrix}\begin{pmatrix}
		2&1&0\\
		1&1&0\\
		1&0&1
	\end{pmatrix}^{-1}, 
\end{equation}
which, when applied to each term in \eqref{eqn: 223_example}, yield
\begin{equation}
    \begin{aligned}
    {\mathcal M}_{\langle 2,2,3\rangle} =
		& \begin{pmatrix}
			1& \varepsilon^3\\
			-\varepsilon^{-2}&-\varepsilon
		\end{pmatrix} \otimes 
		\begin{pmatrix}
			0&0&0\\
			-\varepsilon^{-2} + \varepsilon^{-3}& 2\varepsilon^{-2}-\varepsilon^{-3} &0
		\end{pmatrix}\otimes 
		\begin{pmatrix}
			2\varepsilon^{-1} - 2 &0\\
			2 \varepsilon^{-1} - 1 &0\\
			\varepsilon^{-4}-1&0
		\end{pmatrix} \\
		& +
		\begin{pmatrix}
			0&0\\
			\varepsilon^{-2}&0
		\end{pmatrix} \otimes 
		\begin{pmatrix}
			\varepsilon^4 - \varepsilon+1&-\varepsilon^4 + 2\varepsilon - 1& -\varepsilon^4\\
			\varepsilon - \varepsilon^{-3} &-\varepsilon+\varepsilon^{-3}&-\varepsilon
		\end{pmatrix} \otimes 
		\begin{pmatrix}
			2-\varepsilon^{-1}&2\varepsilon^2-\varepsilon\\
			1-\varepsilon^{-1}&\varepsilon^{2}-\varepsilon\\
			1-\varepsilon^{-4}&\varepsilon^2-\varepsilon^{-2}
		\end{pmatrix} \\
		& + 
		\begin{pmatrix}
			1&-\varepsilon^3\\
			0&0
		\end{pmatrix} \otimes 
		\begin{pmatrix}
			0&0&0\\
			-\varepsilon + \varepsilon^{-3}&\varepsilon - \varepsilon^{-3}&\varepsilon
		\end{pmatrix} \otimes 
		\begin{pmatrix}
        0&0\\
			0&0\\
			-\varepsilon^{-4}&-\varepsilon^{-2}
		\end{pmatrix} \\
        &+
		\begin{pmatrix}
			1&0\\
			-\varepsilon^{-2}&-\varepsilon
		\end{pmatrix} \otimes 
		\begin{pmatrix}
			\varepsilon^4 - \varepsilon &-\varepsilon^4 + 2\varepsilon& -\varepsilon^4\\
			\varepsilon + \varepsilon^{-2} - 2\varepsilon^{-3}& - \varepsilon -2\varepsilon^{-2} + 2\varepsilon^{-3}&-\varepsilon
		\end{pmatrix}\otimes
		\begin{pmatrix}
			2\varepsilon^{-1} -2 &\varepsilon\\
			2\varepsilon^{-1} -1& \varepsilon\\
			\varepsilon^{-4} - 1&\varepsilon^{-2}
		\end{pmatrix} \\
		& +
		\begin{pmatrix}
			1&-\varepsilon^3\\
			-\varepsilon^{-2}&0
		\end{pmatrix} \otimes 
		\begin{pmatrix}
			0&0&0\\
			-\varepsilon^{-2} + 2\varepsilon^{-3}&2\varepsilon^{-2} - 2\varepsilon^{-3}&0
		\end{pmatrix} \otimes 
		\begin{pmatrix}
        \varepsilon^{-1} -2&0\\
			\varepsilon^{-1} - 1&0\\
			\varepsilon^{-4} -1&\varepsilon^{-2}
		\end{pmatrix} \\
        &+ 
		{\color{red}
        \begin{pmatrix}
			0&0\\
			0&\varepsilon
		\end{pmatrix} \otimes 
		\begin{pmatrix}
			\varepsilon^4 - \varepsilon&-\varepsilon^4 + 2\varepsilon&-\varepsilon^4\\
			\varepsilon&-\varepsilon&-\varepsilon
		\end{pmatrix} \otimes 
		\begin{pmatrix}
			0&2\varepsilon^2 - \varepsilon\\
			0&\varepsilon^2-\varepsilon\\
			0&\varepsilon^2
		\end{pmatrix}} \\
        & + 
		\begin{pmatrix}
			0&0\\
			\varepsilon^{-2}&\varepsilon
		\end{pmatrix} \otimes 
		\begin{pmatrix}
			\varepsilon^4 - \varepsilon&-\varepsilon^4 + 2\varepsilon&-\varepsilon^4\\
			\varepsilon-\varepsilon^{-3}&-\varepsilon+\varepsilon^{-3}&-\varepsilon		\end{pmatrix} \otimes 
		\begin{pmatrix}
			2\varepsilon^{-1}-2& -2\varepsilon^2 + 2\varepsilon\\
			2\varepsilon^{-1}-1&-\varepsilon^2 + 2\varepsilon\\
			\varepsilon^{-4}-1&\varepsilon^{-2}-\varepsilon^2
		\end{pmatrix} \\
        &+
		\begin{pmatrix}
			1&0\\
			-\varepsilon^{-2}&0
		\end{pmatrix} \otimes
		\begin{pmatrix}
			1-\varepsilon&2\varepsilon-1&0\\
			\varepsilon^{-2}-\varepsilon^{-3}&-2\varepsilon^{-2} + \varepsilon^{-3}&0
		\end{pmatrix} \otimes 
		\begin{pmatrix}
			2 - \varepsilon^{-1}&-\varepsilon\\
			1-\varepsilon^{-1}&-\varepsilon\\
			1-\varepsilon^{-4}&-\varepsilon^{-2}
		\end{pmatrix} \\
		& +
		\begin{pmatrix}
			1&-\varepsilon^3\\
			-\varepsilon^{-2}&\varepsilon
		\end{pmatrix} \otimes
		\begin{pmatrix}
			0&0&0\\
			\varepsilon-\varepsilon^{-2} + \varepsilon^{-3}&-\varepsilon + 2\varepsilon^{-2} - \varepsilon^{-3}&-\varepsilon
		\end{pmatrix} \otimes
		\begin{pmatrix}
			0&0\\
			0&0\\
			0&-\varepsilon^{-2}
		\end{pmatrix} \\
        &+
		\begin{pmatrix}
			1&0\\
			-\varepsilon^{-2}&0
		\end{pmatrix} \otimes
		\begin{pmatrix}
			\varepsilon^4&-\varepsilon^4&-\varepsilon^4\\
			\varepsilon-2\varepsilon^{-2} + 2\varepsilon^{-3}&-\varepsilon+4\varepsilon^{-2} - 2\varepsilon^{-3}&-\varepsilon
		\end{pmatrix} \otimes
		\begin{pmatrix}
			2-\varepsilon^{-1}&0\\
			1-\varepsilon^{-1}&0\\
			1-\varepsilon^{-4}&0
		\end{pmatrix} \\
        & + 
		\begin{pmatrix}
			1&0\\
			0&0
		\end{pmatrix} \otimes
		\begin{pmatrix}
			1-\varepsilon^4&\varepsilon^4-1&\varepsilon^4\\
			-\varepsilon+\varepsilon^{-3}&\varepsilon-\varepsilon^{-3}&\varepsilon
		\end{pmatrix} \otimes
		\begin{pmatrix}
			\varepsilon^{-1}&\varepsilon\\
			\varepsilon^{-1}&\varepsilon\\
			\varepsilon^{-4}&\varepsilon^{-2}
		\end{pmatrix}.
	\end{aligned}
\end{equation}
As expected, the sixth term is now $O(\varepsilon^3)$ and can be deleted to obtain a $\langle 2,2,3; 10 \rangle$ APA algorithm. \\
\indent Note that the choice of $R$, $S$, and $T$ is not unique. Moreover, it is possible to construct suitable basis transformations without relying on the proof of \cref{thm: remove_one}. As an example, we could take instead 
\begin{equation}
	R = \begin{pmatrix}
		1&0\\
		0&\varepsilon^{-1}
	\end{pmatrix},
	\qquad
	S =
	\begin{pmatrix}
		\varepsilon^{-3}&0\\
		0&1
	\end{pmatrix},
	\qquad
	T=
	\begin{pmatrix}
		1&2&0\\
		0&1&1\\
		1&1&0
	\end{pmatrix}
    \begin{pmatrix}
		\varepsilon^{-2}&0&0\\
		0&1&0\\
		0&0&1
	\end{pmatrix}.
\end{equation}
These basis transformations, which were constructed by hand, transform the sixth term of \eqref{eqn: 223_example} to 
\begin{equation}
    \begin{pmatrix}
			0&0\\
			0&\varepsilon
		\end{pmatrix} \otimes 
		\begin{pmatrix}
			-\varepsilon&0&\varepsilon^3\\
			0&0&-1
		\end{pmatrix} \otimes 
		\begin{pmatrix}
			0&\varepsilon\\
			0&0\\
			0&0
		\end{pmatrix}.
\end{equation}
While in this case the entries of $R^{-1}U_6S$, $S^{-1}V_6T$, and $T^{-1}W_6R$ do \textit{not} all belong to $\varepsilon{\mathbb C}[\varepsilon]$, their tensor product is still $O(\varepsilon)$. In this way, \cref{thm: remove_one} and \cref{thm: remove_more_than_one} guarantee more than is necessary to remove one (or more) terms from an exact algorithm. We leave determining what constitutes an ``optimal" choice for these basis transformations to future work.

\subsection{Implementation Details and Results}\label{sec: Details and Results}

We can use the results from the previous section to determine when an APA algorithm can be derived from an exact one. Nevertheless, the conditions suggested by \cref{thm: remove_one} and \cref{thm: remove_more_than_one} are more rare than we might expect. Indeed, an initial search of the repository of Sedoglavic\footnote{\href{https://fmm.univ-lille.fr/}{https://fmm.univ-lille.fr/}} turned up no exact algorithms whose corresponding tensor decompositions contain a nilpotent term. Luckily we can generate more candidates by running a \textit{flip graph} search, which we describe below.

\begin{remark}
    There are many other search algorithms for exact matrix multiplication that could be used here in place of the flip graph -- e.g., algorithms that apply some version of least squares to the tensor decomposition problem \cite{Smirnov2013,BALLARD2019,Kaporin2024_2} or use reinforcement/machine learning \cite{alphaevolve,nature_MM}. We choose the flip-graph approach since it is particularly easy to use and requires minimal computing resources. 
\end{remark}

\subsubsection{Flip Graph Search for Exact Decompositions}

The flip graph is a recent tool for discovering (exact) matrix multiplication algorithms introduced by Kauers and Moosbauer \cite{flip_graph}. For given $m$, $k$, and $n$, the vertices of the $\langle m,k,n \rangle$ flip graph are $\langle m,k,n \rangle$ matrix multiplication algorithms, represented by their corresponding tensor decompositions. Edges in the graph connect schemes that can be obtained from one another via the \textit{flip} operation. In the simplest case, a flip replaces two rank-one factors $A \otimes B \otimes C$ and $A \otimes B' \otimes C'$ in a matrix-multiplication decomposition with $A \otimes (B+B') \otimes C $ and $A \otimes B' \otimes (C'-C)$. Since this leaves the sum of the terms unchanged, each flip results in a new decomposition and corresponding algorithm (of the same rank). Note that this only works because the terms $A \otimes B \otimes C$ and $A \otimes B' \otimes C'$ share the common factor $A$; to perform a flip in general, we need at least one factor to be shared between the rank-one terms, though it need not be the first. \\
\indent The key insight for Kauers and Moosbauer is that, under the right circumstances, an algorithm found in the flip graph can be \textit{reduced}, meaning two rank-one terms can be replaced by a single one, thereby producing a new algorithm of lower rank. Accordingly, they propose a search algorithm that runs a ``flip-graph walk," in which successive, random flips are applied to an initial decomposition, at each step checking for the conditions that make a reduction possible. In \cite{flip_graph}, they use this algorithm to improve the best-known ranks for several sizes $\langle m,k,n \rangle$. Subsequent work accelerated the original flip graph algorithm by leveraging symmetry and/or traversing multiple sizes in one search \cite{flip_graph_symmetry,Meta_Flip,Adaptive_Flip}. \\
\indent In this work, we are interested in using a flip graph walk to generate many exact algorithms rather than improve the rank of exact matrix multiplication. In this way, the flip graph gives a simple search algorithm for our application:\ run a flip-graph walk, checking not for possible reductions but instead for decompositions that satisfy the conditions of \cref{thm: remove_one} and/or \cref{thm: remove_more_than_one}. With this, we can now state our full algorithm for generating new APA schemes, which is shown here (in high-level pseudocode) as \cref{alg: APA_search}. Like \cref{thm: rank_improvement}, this pseudocode assumes that the terms that can be removed from an exact decomposition are leading. This is done to simplify the notation (and since we can always reorder the terms in a decomposition however we like). \\
\indent For better efficiency we use an implementation of Perminov \cite{perminov2026} to execute the flip-graph seach in line 1 of \cref{alg: APA_search}. As is typical for flip-graph algorithms, this search is conducted over a finite field (usually ${\mathbb F}_2$ or ${\mathbb F}_3$) with nilpotence checked by repeatedly multiplying each $U_jV_jW_j$.  Resulting decompositions are then converted to characteristic zero via Hensel lifting \cite{Gathen_Gerhard_2013}, which we modify to ensure that nilpotence is preserved.\footnote{This can be done since nilpotence of a product is simply a polynomial constraint in the entries of the corresponding matrices, which only requires adding additional rows to the Jacobian in the lifting step.} Of course, it is sometimes not possible to lift a decomposition with one (or more) nilpotent terms to characteristic zero, in which case we simply continue the search. 
\begin{algorithm}[t]
\caption{Search Algorithm for APA Matrix Multiplication\\
\textbf{Input:} An exact decomposition ${\mathcal M}_{\langle m,k,n \rangle} = \sum_{j=1}^r U_j \otimes V_j \otimes W_j$. \\
\textbf{Output:} A decomposition corresponding to an APA $\langle m,k,n \rangle$ algorithm with rank $r' = r-l$ for $l \geq 1$.}\label{alg: APA_search}
\begin{algorithmic}[1]
\STATE Run a flip-graph search (over ${\mathbb F}_2$ or ${\mathbb F}_3$, starting from the input decomposition) until a decomposition $\sum_{j=1}^r U_j' \otimes V_j' \otimes W_j'$ satisfying the conditions  of \cref{thm: remove_one} or \cref{thm: remove_more_than_one} is found. 
\STATE Lift the decomposition to characteristic zero. 
\IF{lifting fails}
\STATE Return to line 1 and continue the search.
\ELSE 
\STATE Construct symbolic basis transformations $R \in {\mathbb C}(\varepsilon)^{m \times m}$, $S \in {\mathbb C}(\varepsilon)^{k \times k}$, and $T \in {\mathbb C}(\varepsilon)^{n \times n}$ such that $R^{-1}U_j'S \otimes S^{-1}V_j'T \otimes T^{-1}W_j'R$ has entries in $\varepsilon {\mathbb C}[\varepsilon]$ for $1 \leq j \leq l$.
\ENDIF
\RETURN $\sum_{j=l+1}^r R^{-1}U_j'S \otimes S^{-1}V_j'T \otimes T^{-1}W_j'R $
\end{algorithmic}
\end{algorithm}

\subsubsection{Additional Results}
In practice, the initial decompositions used to start the search are particularly important. Perminov's implementation allows many flip graph walks to run in parallel, possibly with decompositions of different dimensions for each walk. Periodically, decompositions are combined or modified by meta-moves, as described in the meta-flip-graph paper \cite{Meta_Flip}. With human intervention, we also chose the most promising decompositions seen so far as seeds for a restarted search. Most new border rank decompositions were found over the course of a week using a 23 thread search. We ran the search on a Ryzen 9 9900X CPU. \\
\indent Our best result came in the case of $\langle 2,4,5 \rangle $, where we found the first rank 31 APA algorithm and therefore lowered the border rank upper bound for this size to 31. Encoding matrices for this algorithm are given in \cref{ap: encoding_matrices}. For the other schemes found in our search (including those listed in \cref{tab:results_summary,tab:results_summary2}), see \href{https://github.com/nliskij/border-rank}{https://github.com/nliskij/border-rank}.


\subsection{Generalization via Null Cone Membership Testing}\label{subsection: generalization}
Our approach to searching for suitable exact algorithms is not the most general possible. It may be the case that individual rank one terms cannot be made $ O(\varepsilon) $ but their sum can. In the representation theory literature, the sum of such terms would be called an element of the null cone of the De-Groote group. This may be promising for future work since recent papers have addressed algorithms for testing null cone membership \cite{BurgisserEtAl2018,BurgisserEtAl2019}. While such an approach can produce decompositions more general than described here, in practice we were not able to obtain any further border-rank improvements.
\section{Alternative Basis APA Matrix Multiplication}\label{section: alternative_basis}
In this section, we demonstrate that APA matrix multiplication algorithms can be accelerated by reducing their additive complexity, accomplished by applying fast, symbolic basis transformations to reduce the number of nonzero entries in the encoding matrices $U(\varepsilon)$, $V(\varepsilon)$ and $W(\varepsilon)$ (recall \cref{fig: APA_rep}). This has the effect of lowering the leading coefficient in an APA algorithm's arithmetic \textit{and} IO complexities.\footnote{While not the focus of this paper, IO complexity refers to the amount of data movement or communication performed by an algorithm.} In particular, generalizing \cite{alternative_bases}, we present a greedy algorithm for constructing these basis transformations, which depend on an evaluation parameter $\varepsilon$.

\subsection{Background and Motivation}
To outperform classical matrix multiplication in practice, the arithmetic and IO complexities of fast matrix multiplication algorithms must have small leading coefficients. To promote this, Karstadt and Schwartz~\cite{ALittleFaster2020} introduced alternative basis matrix multiplication, which takes an existing, exact algorithm and computes a set of basis transformations that sparsify the corresponding encoding matrices. That is, if $U \in {\mathbb C}^{m_0k_0 \times t}$, $V \in {\mathbb C}^{k_0n_0 \times t}$, and $W \in {\mathbb C}^{m_0n_0 \times t}$ are the encoding matrices of a $\langle m_0,k_0,n_0;t \rangle$ matrix multiplication scheme, the alternative basis method computes automorphisms $\phi: {\mathbb C}^{m_0k_0} \rightarrow {\mathbb C}^{m_0k_0}$, $\psi: {\mathbb C}^{k_0n_0} \rightarrow {\mathbb C}^{k_0n_0}$, and $\nu : {\mathbb C}^{m_0n_0} \rightarrow {\mathbb C}^{m_0n_0}$ and factors $U$, $V$ and $W$ as follows:
\begin{equation}\label{eqn: AB_encoding matrices}
    U = \phi \cdot U_{\phi}, \; \; \; V = \psi \cdot  V_{\psi}, \; \; \; \text{and} \; \; \; W = \nu^{-1} \cdot  W_{\nu},
\end{equation}
where $U_{\phi}$, $V_{\psi}$, and $W_{\nu}$ are sparse (or at least sparser than $U$, $V$, and $W$). The key insight is the following:\ rewriting \eqref{eqn: encoding_relation} as
\begin{equation}\label{eqn: alternative_basis_relation}
    c = W(U^Ta \odot V^Tb) = \nu^{-1} \left[ W_{\nu} \left(U_{\phi}^T(\phi^Ta)\odot V_{\psi}^T(\psi^Tb) \right)\right] 
\end{equation}
for $a$, $b$, and $c$ vectorized copies of $A \in {\mathbb C}^{m_0 \times k_0}$, $B \in {\mathbb C}^{k_0 \times n_0}$, and $C = AB$, we observe that $\nu(C) \in {\mathbb C}^{m_0 \times n_0}$ can be computed by applying the bilinear form represented by $U_{\phi}$, $V_{\psi}$ and $W_{\nu}$ to the transformed inputs $\phi(A) \in {\mathbb C}^{m_0 \times k_0}$ and $\psi(B) \in {\mathbb C}^{k_0 \times n_0}$. Here, we use $\phi(A)$, $\psi(B)$, and $\nu(C)$ to denote the matrices that, when vectorized, correspond to $\phi^Ta$, $\psi^Tb$, and $\nu c$. Assuming that $\phi(A)$, $\psi(B)$ and $C$ can be obtained cheaply from $A$, $B$, and $\nu(C)$, respectively, this offers a potentially much more efficient way of computing the output $C$. Note additionally that the basis transformations, like the corresponding encoding matrices, can be applied recursively for use on matrices of arbitrary size.  \\
\indent In \cite{ALittleFaster2020}, Karstadt and Schwartz used the alternative basis approach to reduce the leading coefficient in the arithmetic complexity of Strassen's $\langle 2, 2,2;7\rangle$ algorithm to 5 without increasing communication costs, making it more practical than~\cite{Cenk}. Efforts to apply this technique to other fast exact algorithms quickly followed \cite{BeniaminiSchwartzFastDecomposition,SparsifyingMatrixMultiplication,Hadas}.
Corresponding implementations show CPU and GPU speedups for the resulting alternative basis algorithms (e.g., \cite{Vaknin, Weirnik, Karppa, Buker26,DUMAS2026102524}).  In all of these, fast basis transformations are used, which can be applied recursively to $A \in {\mathbb C}^{m \times k}$, $B \in {\mathbb C}^{k \times n}$ and $C \in {\mathbb C}^{m \times n}$ at the modest overhead cost of $O(mk\log(mk))$, $O(kn\log(kn))$, and $O(mn\log(mn))$ flops, respectively (see \cite[Section 2.1]{ALittleFaster2020}).

\indent Extending this work to the APA setting is relatively straightforward. Given a set of encoding matrices $U(\varepsilon)$, $V(\varepsilon)$, and $W(\varepsilon)$ for $\langle m_0,k_0,n_0;t \rangle$ APA matrix multiplication, we seek symbolic basis  transformations $\phi_{\varepsilon}$, $\psi_{\varepsilon}$, and $\nu_{\varepsilon}$ such that 
\begin{equation}\label{eqn: APA_AB_factorization}
    U(\varepsilon) = \phi_{\varepsilon} \cdot U_{\phi}(\varepsilon), \; \; \; V(\varepsilon) = \psi_{\varepsilon} \cdot V_{\psi}(\varepsilon), \; \; \; \text{and} \; \; \; W(\varepsilon) = \nu_{\varepsilon}^{-1} \cdot W_{\nu}(\varepsilon),
\end{equation}
where $U_{\phi}(\varepsilon)$, $V_{\psi}(\varepsilon)$, and $W_{\nu}(\varepsilon)$ are the (sparser) encoding matrices of an APA bilinear form in the bases $\phi_{\varepsilon}$, $\psi_{\varepsilon}$, $\nu_{\varepsilon}$. These can then be used to generate an Alternative Basis APA Matrix Multiplication algorithm, which is shown here as \cref{alg:alt_basis_APA}. 

\begin{algorithm}[h]
\caption{Alternative Basis APA Matrix Multiplication (\textbf{AltBasis-Recursive-APA}$_{\langle \phi, \psi, \nu \rangle}$)\\
\textbf{Input:} $A \in {\mathbb C}^{m \times k}$, $B \in {\mathbb C}^{k \times n}$, and $\varepsilon \in {\mathbb C}$.\\
\textbf{Output:} $M = AB + \varepsilon R_1 + \varepsilon^2R_2 + \cdots + \varepsilon^{p-1} R_{p-1}$ for residual matrices $R_1, \ldots, R_{p-1} \in {\mathbb C}^{m \times n}$ and $p = dL+1$.}\label{alg:alt_basis_APA}
\begin{algorithmic}[1]
\STATE $\widetilde{A} = \phi_{\varepsilon}(A)$
\STATE $\widetilde{B} = \psi_{\varepsilon}(B)$
\STATE{$\widetilde{M} = $ \textbf{Recursive-APA$_{\langle U_{\phi}(\varepsilon),V_{\psi}(\varepsilon),W_{\nu}(\varepsilon)\rangle}$}$(\widetilde{A},\widetilde{B})$}
\STATE $M = \nu_{\varepsilon}^{-1}(\widetilde{M})$
\RETURN{$M$}
\end{algorithmic}
\end{algorithm}

In this pseudocode, \textbf{Recursive-APA}$_{\langle U_{\phi}(\varepsilon), V_{\psi}(\varepsilon), W_{\nu}(\varepsilon) \rangle}(\widetilde{A},\widetilde{B})$ evaluates the recursive bilinear form corresponding to $U_{\phi}(\varepsilon)$, $V_{\psi}(\varepsilon)$, and $W_{\nu}(\varepsilon)$ at the transformed inputs $\widetilde{A}$ and $\widetilde{B}$, obtained by applying $\phi_{\varepsilon}$ and $\psi_{\varepsilon}$ to $A$ and $B$. Following the notation from the previous section, $d$ is the degree of the underlying APA matrix multiplication algorithm and $L$ is the number of recursive steps applied. The complexity of \cref{alg:alt_basis_APA}, in terms of $t$ and the sparsity of $U_{\phi}(\varepsilon)$, $V_{\psi}(\varepsilon)$, and $W_{\nu}(\varepsilon)$ and including the cost of the basis transformations, follows from the exact case -- e.g., \cite{ALittleFaster2020,alternative_bases}.

\subsection{Symbolic Sparsification}
\indent One question remains:\ how do we obtain the symbolic basis transformations $\phi_{\varepsilon}$, $\psi_{\varepsilon}$ and $\nu_{\varepsilon}$? In this work, we use a symbolic version of the greedy sparsification algorithm first proposed by Holtz et al.\ \cite{alternative_bases}, which, in the context of fast exact matrix multiplication, is capable of producing optimal (in terms of resulting sparsity) basis transformations and has been shown to reduce the leading coefficient of an algorithm's arithmetic complexity by up to 89\%. Our symbolic version is shown here as \cref{alg: APA_sparsificiation}. For simplicity we restrict our focus to $U(\varepsilon)$, but the algorithm applies identically to $V(\varepsilon)$ and $W(\varepsilon)$.

\begin{algorithm}[h]
\caption{Symbolic Sparsification Algorithm\\
\textbf{Input:} $U(\varepsilon) \in {\mathbb C}[\varepsilon,\varepsilon^{-1}]^{m_0k_0 \times t}$ with $\text{rank}_{{\mathbb C}(\varepsilon)}(U(\varepsilon)) = m_0k_0$ and a constant $r\in \mathbb{R}^+$. \\
\textbf{Output:} Invertible $ \phi_{\varepsilon}^{-1} \in {\mathbb C}[\varepsilon, \varepsilon^{-1}]^{m_0k_0 \times m_0k_0}$ such that $U_{\phi}(\varepsilon) = \phi_{\varepsilon}^{-1}U(\varepsilon)$ is as sparse as possible. }\label{alg: APA_sparsificiation}
\begin{algorithmic}[1]
\STATE $\phi_{\varepsilon}^{-1} = \emptyset$ 
\STATE $w_c =$ \textbf{Generate-Constraints}$(\phi_{\varepsilon}^{-1})$
\FOR{$i = 1:r \cdot t$}
\WHILE{\textbf{Next-Vector}$(i,w_c)\neq \emptyset$}
\STATE $v = $ \textbf{Next-Vector}$(i,w_c)$
\STATE $\phi_{\varepsilon}^{-1} = [\phi_{\varepsilon}^{-1}; v^T]$ 
\STATE $w_c = $ \textbf{Generate-Constraints} $(\phi_{\varepsilon}^{-1})$
\IF{\text{size}$(\phi_{\varepsilon}^{-1},1)= m_0k_0$}
\RETURN $\phi_{\varepsilon}^{-1}$
\ENDIF
\ENDWHILE
\ENDFOR
\end{algorithmic}
\end{algorithm}

\cref{alg: APA_sparsificiation} progressively constructs a set of vectors $\phi_{\varepsilon}^{-1}$ in ${\mathbb C}[\varepsilon, \varepsilon^{-1}]^{m_0k_0}$ according to the following rules:\ (1) each newly added vector $v$ should be linearly independent to the current collection and (2) $v^TU(\varepsilon)$ should be as sparse as possible. This is done via the following subroutines:
\begin{enumerate}
    \item \textbf{Generate-Constraints}$(\phi_{\varepsilon}^{-1})$ --  returns a set of constraints $w_c$ that, when satisfied by a vector $v \in {\mathbb C}[\varepsilon,\varepsilon^{-1}]^{m_0k_0}$, implies that it is linearly independent to every (row) vector in $\phi_{\varepsilon}^{-1}$. 
    \item \textbf{Next-Vector}$(i,w_c)$ --  returns a vector $v \in {\mathbb C}[\varepsilon,\varepsilon^{-1}]^{m_0k_0}$ that satisfies both the constraints $w_c$ and weight$(v^TU(\varepsilon)) = i$. 
\end{enumerate} 
For the latter, the function $\text{weight}(\cdot )$ is a measure of sparsity defined as
\begin{equation}\label{eqn: weight_defn}
    \text{weight}(v^TU(\varepsilon)) \coloneqq \text{nnz}(v^TU(\varepsilon)) + \text{nns}(v^TU(\varepsilon)) + \text{nnc}(v^TU(\varepsilon))\cdot p,
\end{equation}
where nnz is the number of nonzero entries, nns is the number of constant, nonsingleton entries (that is, not $-1$, $0$, or $1$), and nnc is the number of variable entries,  i.e., that contain $\varepsilon$, for $p$ a tunable parameter. 
\eqref{eqn: weight_defn} naturally extends to the matrix $U(\varepsilon)$ as follows
\begin{equation}\label{eqn: weight_matrix}
    \begin{aligned}
    \text{weight}(U(\varepsilon)) &\coloneqq \text{nnz}(U(\varepsilon)) + \text{nns}(U(\varepsilon)) + \text{nnc}(U(\varepsilon))\cdot p-\text{ncols}(U(\varepsilon)),\\
    \end{aligned}
\end{equation}
where ncols is the number of columns.\footnote{In the case of $\text{weight}(W(\varepsilon))$, ncols is replaced by nrows, the number of rows in $W(\varepsilon)$.} 
By repeatedly calling \textbf{Next-Vector} for increasing values of $i$, \cref{alg: APA_sparsificiation} promotes optimal sparsity; that is, each newly added vector $v$ has the smallest value of weight$(v^TU(\varepsilon))$ possible. 
Accordingly, the constant $r$ should be chosen so that $r\cdot t$ is the maximum possible value of $\text{weight}(v^TU(\varepsilon))$, specifically to guarantee that the algorithm terminates.  \\
\indent Importantly, we have made several simplifications here both for convenience and to accelerate \cref{alg: APA_sparsificiation}. We summarize them below.
\begin{enumerate}
    \item First, we assume that the encoding matrix $U(\varepsilon)$ belongs to ${\mathbb C}[\varepsilon,\varepsilon^{-1}]^{m_0k_0 \times t}$, equivalently that its entries are Laurent polynomials in $\varepsilon$ as opposed to arbitrary rational functions in $\varepsilon$. Similarly, \textbf{Next-Vector} returns $v \in {\mathbb C}[\varepsilon,\varepsilon^{-1}]^{m_0k_0}$. We focus on Laurent polynomials for the following reasons:\ (1) essentially all APA algorithms in the literature -- including those derived in \cref{section: search} -- have encoding matrices with Laurent polynomial entries and (2) restricting the search to ${\mathbb C}[\varepsilon,\varepsilon^{-1}]$ greatly improves its efficiency without sacrificing sparsity (see \cref{thm:independence}). Note that, while this ultimately implies that $\phi_{\varepsilon}^{-1}$ will have Laurent-polynomial entries, its inverse $\phi_{\varepsilon}$ (which recall is necessary to evaluate \cref{alg:alt_basis_APA}) may contain rational entries. 
    \item Second, we choose not to differentiate between different kinds of variable entries in our measure of sparsity -- i.e., a simple monomial entry like $\varepsilon^2$ and a much more complicated Laurent polynomial contribute the same to \eqref{eqn: weight_defn} and \eqref{eqn: weight_matrix}. We justify this choice by noting that, when running \cref{alg:alt_basis_APA}, any variable entries in the encoding matrices/basis transformations can be computed ahead of time at minor additional cost (that will not contribute to the leading-term complexity of the algorithm). Of course, \cref{alg: APA_sparsificiation} can be run with a different definition of $\text{weight}(\cdot)$, if desired.
\end{enumerate}

As mentioned above, the parameter $p$ is tunable and in particular determines how much to penalize variable entries in $v^TU(\varepsilon)$. Choosing a large value of $p$ can often lead to basis transformations that eliminate $\varepsilon$ terms entirely, in which case $U_{\phi}(\varepsilon)$, $V_{\psi}(\varepsilon)$, and $W_{\nu}(\varepsilon)$ in \eqref{eqn: APA_AB_factorization} are constant (see, e.g., \cref{subsec: section_two_special_case}). \\
\indent In the remainder of this subsection, we discuss implementation details for the key subroutines of \cref{alg: APA_sparsificiation} before discussing its optimality. In \cref{subsection: AB_results}, we present the results of applying alternative bases to the APA routines found in \cref{section: search}.

\subsubsection{Implementation Details}
There are a handful of technical questions concerning \cref{alg: APA_sparsificiation} that we address below:
\begin{enumerate}
    \item How does \textbf{Generate-Constraints} work? In \cref{prop:certificate-construction}, we demonstrate that it is sufficient to define constraints in terms of a basis for the kernel of $\phi_{\varepsilon}^{-1}$.
    \item Is \cref{alg: APA_sparsificiation} well-posed -- i.e., can \textbf{Next-Vector} always find a solution with Laurent-polynomial entries? Moreover, can this be done without sacrificing sparsity? \cref{thm:independence} verifies that the answer to both is yes (under certain conditions).
\end{enumerate}

We start with the subroutine \textbf{Generate-Constraints}, which can be implemented by computing a basis for the kernel of $\phi_{\varepsilon}^{-1}$ (in the vector space ${\mathbb C}(\varepsilon)^n$) and subsequently defining the set of constraints as those that guarantee that the next vector $v$ has a nontrivial inner product with at least one of these basis vectors. This follows from a basic linear algebra result (see e.g., \cite[Theorem 3.2]{alternative_bases}), which we state here for ${\mathbb C}(\varepsilon)$.

\begin{prop}\label{prop:certificate-construction}
Let $v_1, \ldots, v_m \in {\mathbb C}(\varepsilon)^n$ and let $A \in {\mathbb C}(\varepsilon)^{m \times n}$ be the matrix whose $i$-th row is $v_i^T$. Suppose  that $q_1, \ldots, q_r$ form a basis for the kernel of $A$ and let $v \in {\mathbb C}(\varepsilon)^n$. In this case, $v \notin \text{\normalfont span}\left\{ v_1, \ldots, v_m \right\}$ if and only if $q_i^Tv \neq 0$ for at least one $i$. Hence, kernel basis vectors of $A$ can be used to certify the independence of $v$ from the rows of $A$.
\end{prop}

With this, \textbf{Next-Vector} can be implemented (symbolically) via the Z3 satisfiability solver in Python, which takes the constraints implied by \cref{prop:certificate-construction} and searches  for admissible vectors in ${\mathbb C}[\varepsilon, \varepsilon^{-1}]$ such that $\text{weight}(v^TU(\varepsilon))$ is minimal. In practice, this corresponds to solving $v^TU(\varepsilon) = u$ for $u$ a vector with Laurent polynomial entries and a pre-set weight value. The following technical result guarantees that such a system -- subject to the linear-independence constraints obtained from \textbf{Generate-Constraints} --  can be solved over ${\mathbb C}[\varepsilon,\varepsilon^{-1}]^{m_0k_0}$ if it is solvable over ${\mathbb C}(\varepsilon)^{m_0k_0}$, thereby guaranteeing that \cref{alg: APA_sparsificiation} is well-posed. For details on how the Z3 solver chooses candidates $u$ and/or decides to terminate the search for a given value of $\text{weight}(\cdot)$, we refer readers to \cite{demoura2008z}. 

\begin{thm}\label{thm:independence}
Let $A\in {\mathbb C}[\varepsilon,\varepsilon^{-1}]^{m \times n}$ have $\text{\normalfont rank}_{{\mathbb C}(\varepsilon)}(A)=\rho\le n$, and assume every nonzero entry of $A$ has exponent support contained in $[\alpha,\beta]$ for $\alpha \leq 0$. Let $d=\beta-\alpha$ and let $b\in\mathbb {\mathbb C}[\varepsilon, \varepsilon^{-1}]^n $ have entries with exponents in $[B_-,B_+]$, again with $B_- \leq 0$. Define $A_0=\varepsilon^{-\alpha}A$, whose entries are polynomials in $\varepsilon$ of degree at most $d$. Finally, let $N_i$ be the full-rank $\rho\times\rho$ minors of $A_0$ with $\Delta_i(\varepsilon) = \det(N_i)$ and let $S \subseteq {\mathbb C}[\varepsilon, \varepsilon^{-1}]^m$ be a collection of vectors. Suppose that the following hold:
\begin{enumerate}
    \item There exist polynomials $h_i(\varepsilon)\in\mathbb {\mathbb C}[\varepsilon]$ with $\deg h_i\le H$ for all $i$ such that $\sum_i h_i(\varepsilon)\Delta_i(\varepsilon)=c\varepsilon^k$ for an integer $k \geq 0$ and a nonzero constant $c \in {\mathbb C}$.
    \item There exists $y\in {\mathbb C}(\varepsilon)^m$ such that $y^TA=b^T$ and $y \notin \text{\normalfont span}_{{\mathbb C}(\varepsilon)}(S)$.
\end{enumerate}
Then there exists a Laurent-polynomial solution $y'\in {\mathbb C}[\varepsilon,\varepsilon^{-1}]^m$ such that
$(y')^TA = b^T$, $y' \notin \text{\normalfont span}_{{\mathbb C}(\varepsilon)}(S)$, and every nonzero entry of $y'$ has exponent support contained in 
\[
I = \left[ \min \left\{ 0, \;  B_- -\alpha-k \right\} , \; \; \max \left\{B_+ -\alpha -k + H + (\rho-1)(\beta - \alpha), \; \rho(\beta-\alpha) \right\} \right].
\]
\end{thm}
\begin{proof} We start by noting $\text{rank}_{{\mathbb C}(\varepsilon)}(A_0) = \rho$ and $A = \varepsilon^{\alpha}A_0$. Hence, if $y^TA = b^T$ then $z^TA_0 = b^T$ for $z = \varepsilon^{\alpha} y \in {\mathbb C}(\varepsilon)^m$. With this in mind, let $I_i$ and $J_i$ collect, respectively, the row and column indices  corresponding to the minor $N_i$ of $A_0$. If $\text{adj}(N_i)$ is the adjugate of $N_i$, then we have $\text{adj}(N_i)N_i = \Delta_i(\varepsilon)I_{\rho}$ and 
therefore
\begin{equation}
    b_i^T\text{adj}(N_i)N_i = \Delta_i(\varepsilon)b_i^T,
\end{equation}
where $b_i = b\big|_{J_i} \in {\mathbb C}(\varepsilon)^\rho$ is the vector obtained by restricting $b$ to the indices in $J_i$. Define now $u_i^T = b_i^T\text{adj}(N_i)$ and let $\tilde{u}_i\in\mathbb C[\varepsilon,\varepsilon^{-1}]^m$ such that 
\begin{equation}
    \tilde{u}_i\big|_{I_i} = u_i \; \; \; \;  \text{and} \; \; \; \;  \tilde{u}_i \big|_{I_i^C} = 0.
\end{equation}
In this case
\begin{equation}
    \tilde{u}_i ^TA_0 \big|_{J_i} = u_i^TN_i = \Delta_i(\varepsilon) b_i^T = \left( \Delta_i(\varepsilon) b \big|_{J_i} \right)^T.
\end{equation}
Now since $N_i$ is a full rank minor of $A_0$ and $\text{rank}_{\mathbb C(\varepsilon)}(A_0) = \rho$, we can express each column of $A_0$ as a linear combination of the columns with indices in $J_i$. Letting $A_0(:,r) = \sum_{u \in J_i} \lambda_u(\varepsilon) A_0(:,u)$ for some $\lambda_u(\varepsilon) \in {\mathbb C}(\varepsilon)$ and recalling $z^TA_0 = b^T$, we have
\begin{equation} \label{eqn: u_identity}
    \aligned 
    \tilde{u}_i^TA_0(:,r) &= \tilde{u}_i^T \left( \sum_{u \in J_i} \lambda_u(\varepsilon) A_0(:,u) \right)  = \sum_{u \in J_i} \lambda_u(\varepsilon) (\tilde{u}_i^T A_0(:,u)) = \sum_{u \in J_i} \lambda_u(\varepsilon) \Delta_i(\varepsilon) b_u  \\
    & = \Delta_i(\varepsilon) \sum_{u \in J_i}\lambda_u(\varepsilon) (z^TA_0(:,u))  = \Delta_i(\varepsilon) z^T \left( \sum_{u \in J_i} \lambda_u (\varepsilon) A_0(:,u)\right)  = \Delta_i(\varepsilon) z^T A_0(:,r) \\
    & = \Delta_i(\varepsilon) b_r,
    \endaligned 
\end{equation}
where $b_r$ is the $r$-th entry of $b$. Hence, $\tilde{u}_i^TA_0 = \Delta_i(\varepsilon) b^T$. Suppose now 
\begin{equation}\label{eqn: y_prime}
    y' = \frac{1}{c\varepsilon^{k+\alpha}} \sum_{i} h_i(\varepsilon)\tilde{u}_i,
\end{equation}
which we note has Laurent-polynomial entries by construction. Combining \eqref{eqn: u_identity} with $\sum_i h_i(\varepsilon) \Delta_i(\varepsilon) = c\varepsilon^k$ and $A = \varepsilon^{\alpha} A_0$, it is easy to see that $(y')^TA = b^T$, as desired. Moreover, since the degree of any entry in $\text{adj}(N_i)$ is at most $(\beta-\alpha)(\rho-1)$, the exponent support for entries of $y'$ is a subset of
\begin{equation}\label{eqn: initial_interval}
I = \left[ B_- -\alpha-k , \; \;B_+ -\alpha -k + H + (\beta-\alpha)(\rho-1) \right].
\end{equation}
It remains to show that $y'$ does not belong to the span of $S$ in ${\mathbb C}(\varepsilon)^m$. To do this, let $v_1, \ldots v_{m-\rho} \in {\mathbb C}[\varepsilon, \varepsilon^{-1}]^m$ be a basis for the \textit{left} null space of $A$. By construction, $y' - y$ belongs to this space and hence can be written as a linear combination of $v_1, \ldots, v_{m-\rho}$. Suppose now that $y' \in \text{span}_{\mathbb{C}(\varepsilon)}(S)$. Since $y \notin \text{span}_{\mathbb{C}(\varepsilon)}(S)$ we must have $y'-y \notin \text{span}_{\mathbb{C}(\varepsilon)}(S)$ and therefore $v_j \notin \text{span}_{\mathbb{C}(\varepsilon)}(S)$ for at least one $j$. But in this case, we can simply replace $y'$ by $y'+v_j$, which satisfies
\begin{equation}
    (y'+v_j)^TA = (y')^TA + v_j^TA = b^T + 0 = b^T
\end{equation}
and does not belong to $\text{span}_{\mathbb{C}(\varepsilon)}(S)$. To verify that this will not destroy the structure of $y'$, we next show that $v_1, \ldots, v_{m-\rho}$ can be constructed to have Laurent polynomial entries. We start by noting that $A$ and $A_0$ have the same left null spaces over ${\mathbb C}(\varepsilon)$, so it suffices to construct a basis for $A_0$. Let $I_i$ and $J_i$ again correspond to the row/column indices of the minor $N_i$ of $A_0$. Since the columns of $N_i$ are linearly independent, we have
\begin{equation}
    q^TA_0 = 0 \; \; \iff \; \; q^TA_0(:,J_i)  = 0
\end{equation}
For $j \in I_i^C$ let $c_j = A_0(j,J_i)$ and define $v_j \in {\mathbb C}[\varepsilon]^m$ to be the vector satisfying the following:
\begin{equation}
    v_j\big|_{I_i} = -\text{adj}(N_i)^Tc_j^T, \; \; v_j(j) = \Delta_i(\varepsilon), \; \; \text{and} \; \;  v_j \big|_{\left\{I_i\cup j \right\}^C} = 0.
\end{equation}
This vector has polynomial entries because $\operatorname{adj}(N_i)$, $c_j$, and $\Delta_i(\varepsilon)$ all have entries in ${\mathbb C}[\varepsilon]$.

Then, by construction, we have
\begin{equation}
    \aligned
    v_j^TA_0(:,J_i) &=  (v_j\big|_{I_i^C})^TA_0(I_i^C,J_i) + (v_j\big|_{I_i})^TN_i = \Delta_i(\varepsilon)c_j - c_j \text{adj}(N_i) N_i  =  \Delta_i(\varepsilon) c_j - \Delta_i(\varepsilon) c_j = 0.
    \endaligned
\end{equation}
Moreover, $v_1, \ldots, v_{m-\rho}$ are clearly linearly independent, as each contains the nonzero entry $\Delta_i(\varepsilon)$ where the others are zero. Since there are $m-\rho$ such vectors, and 
$\operatorname{rank}_{\mathbb C(\varepsilon)}(A_0)=\rho$, $v_1,\ldots,v_{m-\rho}$ must form a
$\mathbb C(\varepsilon)$-basis of the left null space of $A_0$. \\
\indent Finally, to complete the proof we update \eqref{eqn: initial_interval} -- i.e., since $y'$ (as defined in \eqref{eqn: y_prime}) might need to be replaced by $y' + v_j$, where each entry of $v_j$ is a polynomial of degree at most $\rho(\beta- \alpha)$. Shifting \eqref{eqn: initial_interval} accordingly gives the listed exponent range. 
\end{proof}

\cref{thm:independence} not only guarantees that \textbf{Next-Vector} can find Laurent-polynomial-based solutions but also bounds the corresponding range of exponents in the its entries. This can be used to further accelerate the \cref{alg: APA_sparsificiation} -- e.g., by providing Z3 with a narrower range of exponents to search over. Of course, for this to be usable we need to bound $H$ and $k$. This can be done by applying the following identity to the $\Delta_i(\varepsilon)$. 

\begin{prop}[Hilbert nullstellensatz]\label{lemma:bezout-bound}
Let $f_1,\dots,f_m\in \mathbb {\mathbb C}[\varepsilon]$, and let $D=\max_i \deg f_i$.
If
\begin{equation}\label{eqn: gcd is one}
\gcd(f_1,\dots,f_m)=1
\end{equation}
then there exist polynomials $a_1(\varepsilon),\dots,a_m(\varepsilon)\in {\mathbb C}[\varepsilon]$ such that $
\sum_{i=1}^m a_i(\varepsilon)f_i(\varepsilon)=1 $ where $\deg a_i\le \max \left\{ 3, D \right\}$ for every $i$.
\end{prop}
\begin{proof}
    This is a classical result in algebraic geometry. For the degree bound, see \cite{Kollar}.
\end{proof}

In particular, suppose that $g(\varepsilon)=\gcd(\Delta_1(\varepsilon),\dots,\Delta_m(\varepsilon))$. Since \cref{thm:independence} assumes that the ideal generated by the $\Delta_i(\varepsilon)$ contains a monomial, $g(\varepsilon)$ divides a monomial and is therefore one itself (up to a nonzero scalar). With this in mind, write $g(\varepsilon) = c\varepsilon^r$ for some $c \neq 0$ and $r \geq 0$. Since $g(\varepsilon)$ divides the $\Delta_i(\varepsilon)$, and recalling that each $\Delta_i(\varepsilon)$ has degree at most $\rho(\beta-\alpha)$, we have
\begin{equation}\label{eqn: r bound}
0\le r\le \rho(\beta-\alpha).
\end{equation}
For each $i$, write $\Delta_i(\varepsilon)=c\varepsilon^r f_i(\varepsilon)$.
By construction, $\gcd(f_1,\dots,f_m)=1$
and
\begin{equation}
\deg f_i=\deg \Delta_i-r\le \rho(\beta-\alpha)-r,
\end{equation} 
so \cref{lemma:bezout-bound} implies that there exist polynomials $a_i(\varepsilon)\in\mathbb C[\varepsilon]$ such that $\sum_i a_i(\varepsilon)f_i(\varepsilon)=1$
and $\deg a_i\le \rho(\beta-\alpha)-r$.
Multiplying by $c\varepsilon^r$, we obtain
\begin{equation}
\sum_i a_i(\varepsilon)\Delta_i(\varepsilon)=c\varepsilon^r.
\end{equation}
Thus, in \cref{thm:independence}, we can take $k=r$ and $H\le \max \left\{ 3, \rho(\beta-\alpha)-r\right\}$. Notice that, because we simply need to keep the constants the same to have a vector with same or better weight, we can set $B_- = B_+ = 0$. 
\subsubsection{Optimality}
With the above results, \cref{alg: APA_sparsificiation} is guaranteed to achieve optimal sparsity, as measured by \eqref{eqn: weight_defn} and \eqref{eqn: weight_matrix}, under the assumption that \textbf{Next-Vector} returns an admissible vector whenever one exists (guaranteed as long as the first condition of \cref{thm:independence} holds for the corresponding encoding matrix). This can be shown by simply replicating the optimality proof from~\cite{alternative_bases}, interpreting the notations of span, rank, and independence as defined over ${\mathbb C}(\varepsilon)$. Note that this still holds under alternative definitions of $\text{weight}(\cdot)$, meaning \cref{alg: APA_sparsificiation} is adaptable to other notions of sparsity. 

\subsubsection{Sparsifying Algorithms Obtained by Term Deletion}\label{subsec: section_two_special_case}

We close this subsection by considering the sparsification of the APA algorithms obtained from the search procedure discussed in \cref{section: search}. These algorithms have special structure:\ their parameter dependence comes from a De-Groote transformation applied to an exact decomposition. As a result, verified below, we can always construct alternative bases whose corresponding encoding matrices $U_{\phi}$, $V_{\psi}$, and $W_{\nu}$ are independent of $\varepsilon$. Moreover, this does not come at the expense of sparsity (again, as measured by $\text{weight}(\cdot)$).

\begin{prop}\label{prop:constant-core-sparsification}
Let $U(\varepsilon)$, $V(\varepsilon)$, and $W(\varepsilon)$ be obtained from an exact algorithm as in \cref{section: search}. Consider any invertible, symbolic alternative-basis transformations $\phi_{\varepsilon}$, $\psi_{\varepsilon}$, and $\nu_{\varepsilon}$ for which the encoding matrices
\[
U_{\phi}(\varepsilon)=\phi_{\varepsilon}^{-1}\cdot U(\varepsilon),\qquad
V_{\psi}(\varepsilon)=\psi_{\varepsilon}^{-1} \cdot V(\varepsilon),\qquad
W_{\nu}(\varepsilon) =\nu_{\varepsilon} \cdot W(\varepsilon)
\]
have Laurent-polynomial entries. The following hold:
\begin{enumerate}
\item There are equivalent alternative-basis transformations $\tilde\phi_{\varepsilon}$, $\tilde\psi_{\varepsilon}$, and $\tilde\nu_{\varepsilon}$ for which $U_{\tilde\phi}(\varepsilon)$, $V_{\tilde\psi}(\varepsilon)$, and $W_{\tilde\nu}(\varepsilon)$ are constant. The maps $\tilde\phi_{\varepsilon}^{-1}$, $\tilde\psi_{\varepsilon}^{-1}$, and $\tilde\nu_{\varepsilon}$ have Laurent-polynomial entries.
\item If the parameter $p$ in \eqref{eqn: weight_defn} satisfies $p\geq 1$, the weight of each constant encoding matrix is no greater than that of its original counterpart.
\end{enumerate}
\end{prop}

\begin{proof}
We prove both claims for $U(\varepsilon)$. Let $U$ be the encoding matrix of the parent decomposition after the selected terms have been deleted. With the row-wise vectorization used here, set $D_{\varepsilon}=R^{-1}\otimes S^T$, where $R$ and $S$ are the De-Groote transformations from \cref{section: search}. Then $U(\varepsilon)=D_{\varepsilon}U$, and $D_{\varepsilon}$ and $D_{\varepsilon}^{-1}$ have Laurent-polynomial entries. Set
\begin{equation}
    G_{\varepsilon}=\phi_{\varepsilon}^{-1}D_{\varepsilon}.
\end{equation}
Because the coefficient field is infinite and $G_{\varepsilon}$ is invertible over ${\mathbb C}(\varepsilon)$, we can choose a nonzero $\varepsilon_0$ at which $G_{\varepsilon_0}$ is defined and invertible. Define
\begin{equation}
    \tilde\phi_{\varepsilon}^{-1}=G_{\varepsilon_0}D_{\varepsilon}^{-1}.
\end{equation}
This map has Laurent-polynomial entries, and
\begin{equation}
U_{\tilde\phi}(\varepsilon)
=\tilde\phi_{\varepsilon}^{-1}U(\varepsilon)
=G_{\varepsilon_0}U
\end{equation}
is constant. 
Moreover, $U_{\tilde\phi}(\varepsilon)$ is obtained by evaluating $U_{\phi}(\varepsilon)=G_{\varepsilon}U$ at $\varepsilon_0$. Hence, every constant entry retains its value, and every nonconstant entry becomes a constant. This implies that the weight cannot increase. 
For $V(\varepsilon)$ and $W(\varepsilon)$, the proof is the same, using $\psi_{\varepsilon}^{-1}$ and $\nu_{\varepsilon}$, respectively, as the maps that produce their encoding matrices.
\end{proof}

The benefit of sparsifying this way (when possible) is straightforward:\ if $U_{\phi}$, $V_{\psi}$, and $W_{\nu}$ have no $\varepsilon$ dependence, then we can perform the approximate-to-exact procedure of Bini \cite{Bini1980} by applying several basis transformations in parallel before evaluating the bilinear algorithm corresponding to $U_{\phi}$, $V_{\psi}$, and $W_{\nu}$ in a batched fashion.

\subsection{Results}\label{subsection: AB_results}

    To evaluate the arithmetic complexity of the resulting algorithms, we compute the leading coefficients using the framework of~\cite{alternative_bases}. For each sparsified algorithm, we first compute the base-case linear-operation counts $q_U$, $q_V$, and $q_W$ of its encoding and decoding matrices. In the APA setting, these counts are evaluated symbolically in $\varepsilon$, and we use $p=2$ for the weight calculation in~\eqref{eqn: weight_defn}. We then apply the recursion-order optimization of~\cite{alternative_bases}, which uses dynamic programming to evaluate recursive compositions of the algorithm and its permutations for square matrix multiplication. Among the resulting recursion orders, we report the minimum leading coefficient returned by this procedure. An example is given in \cref{ap: sparsification_example}. 

    The sparsification results recorded in \cref{tab:results_summary2} were obtained on a laptop with an Intel Core Ultra 9. 
    Most encoding matrices could be successfully sparsified after running \cref{alg: APA_sparsificiation} for less than an hour. Only three required a run time of more than a day:\ two completed in two days and the last in three. While, in general, larger matrices took longer to sparsify, this is not the only factor determining the run time of \cref{alg: APA_sparsificiation}. In fact, it seems that the complexity of the matrix entries matter more than the size of the matrix. 

\section{APA Matrix Multiplication is Stable}\label{section: stability}

In this section, we show that alternative basis APA matrix multiplication is stable by proving a (weak) forward error bound for a floating-point implementation of the approximate-to-exact procedure represented by \cref{alg:alt_basis_APA_exact}. In line three of this algorithm, \textbf{AltBasis-Recursive-APA}$_{\langle \phi, \psi, \nu \rangle}$ is a recursive implementation of an alternative basis APA matrix multiplication routine (i.e, \cref{alg:alt_basis_APA}), indexed by the corresponding symbolic basis transformations $\phi$, $\psi$, and $\nu$. Following the previous section, this algorithm is based on an alternative basis  $\langle m_0,k_0,n_0; t \rangle$ scheme defined by parameterized encoding matrices $U_{\phi}(\varepsilon)$, $V_{\psi}(\varepsilon)$ and $W_{\nu}(\varepsilon)$.  

\begin{algorithm}[h]
\caption{Approximate-to-Exact Alternative Basis APA Matrix Multiplication\\
\textbf{Input:} $A \in {\mathbb C}^{m \times k}$, $B \in {\mathbb C}^{k \times n}$ and parameters $ \varepsilon_1, \ldots,\varepsilon_p \in {\mathbb C}$, where $p = dL+1$. \\
\textbf{Output:} $C = AB$. }\label{alg:alt_basis_APA_exact}
\begin{algorithmic}[1]
\STATE{$[\alpha_1, \ldots, \alpha_p] = $ \textbf{Vandermonde-Solve}$(\varepsilon_1, \ldots, \varepsilon_p)$}
\FOR{$j = 1:p$}
    \STATE $C_j = $ \textbf{AltBasis-Recursive-APA}$_{\langle \phi, \psi, \nu \rangle}(A,B,\varepsilon_j)$
\ENDFOR
\STATE{$C = \alpha_1C_1 + \alpha_2C_2 + \cdots + \alpha_pC_p$}
\RETURN{$C$}
\end{algorithmic}
\end{algorithm}

\indent As discussed in the introduction, \cref{alg:alt_basis_APA_exact} obtains the exact product of $A$ and $B$ by evaluating \textbf{AltBasis-Recursive-APA}$_{\langle \phi, \psi, \nu \rangle}$ $p = dL+1$ times with different values of $\varepsilon$ and taking a linear combination of the output (after applying the corresponding basis transformation $\nu_{\varepsilon}$). The coefficients in this linear combination solve the Vandermonde system \eqref{eqn: vandermonde_system_d} with $p$ replacing $d$ 
and are computed in line 1 of \cref{alg:alt_basis_APA_exact} by the routine \textbf{Vandermonde-Solve}. Note that each $\alpha_j$ can be expressed in terms of $\varepsilon_1, \ldots, \varepsilon_p$ as follows:
\begin{equation}\label{eqn: vandermonde_solution}
   \alpha_j = (-1)^{p-1} \prod_{k \neq j} \frac{\varepsilon_k}{\varepsilon_j-\varepsilon_k} = \prod_{k\neq j} \frac{\varepsilon_k}{\varepsilon_k-\varepsilon_j}.
\end{equation}
We also have the special case $\alpha_j = 1/p$ for all $j$ when the $\varepsilon_i$ are the $p$-th roots of unity. \\
\indent To obtain a forward error bound for a floating-point implementation of \cref{alg:alt_basis_APA_exact} we extend work of Schwartz et al. \cite{Schwartz2026_journal} to the APA setting. The results of \cite{Schwartz2026_journal} are the most recent in a long line of research on the stability of fast matrix multiplication, which traces its roots back to work of Brent \cite{Brent70}  and Bini and Lotti \cite{Bini_Lotti}.
The normwise stability bounds derived by these authors for specific algorithms -- i.e., the same notion of weak forward stability established for APA algorithms in this paper (see \cref{thm: main_APA_error_bound}) -- was eventually shown to apply to the broad class of fast (exact) matrix multiplication routines by Demmel et al.\ \cite{fastMM}. The contribution of Schwartz et al.\  in \cite{Schwartz2026_journal}, itself a generalization of \cite{Ballard2016}, was to extend this framework to include alternative basis matrix multiplication  \cite{alternative_bases}, and hence their bounds are a natural fit for our setting in light of the previous section. To the best of our knowledge, the results presented here are the first stability bounds derived for APA matrix multiplication. 

\begin{remark} The bounds presented in this section are stated in terms of the maximum-entry norm -- i.e., $\|A\| = \max_{ij}|A(i,j)|$. This matches most of the results in the literature on the stability of fast matrix multiplication. Dumas, Pernet, and Sedoglavic \cite{DUMAS2026102524}, however, recently derived new bounds in terms of arbitrary norms, which are tighter than what can be obtained by simply applying norm equivalence. They demonstrate the efficacy of their bounds by using them to obtain algorithms that are more stable in practice (including with respect to the maximum entry norm). Since our goal in this section is to demonstrate stability for \cref{alg:alt_basis_APA_exact}, we leave an APA extension of these results to future work. Similarly, we do not consider scaling techniques (e.g., \cite{Dumitrescu1998}), even though this is a well-documented way to further improve the stability of fast exact matrix multiplication (see the discussion in \cite{grading_the_blas}). 
\end{remark}

\subsection{Preliminaries}
\indent We begin by summarizing the building blocks of our analysis. Throughout, we adopt a standard model for floating-point arithmetic \cite{Higham_book,Overton}, in which each binary floating-point operation $\circ \in \left\{ +,-,\times,\div \right\}$ satisfies
\begin{equation}\label{eqn: floating-point assumption}
    fl(x \circ y) = (x \circ y)(1+\delta) \; \;  \text{for some} \; \; |\delta| \leq {\bf u},
\end{equation}
where $fl(x \circ y)$ and $x \circ y$ are the floating-point and exact outputs, respectively, and ${\bf u}$ is the \textit{unit roundoff}. Every operation in \cref{alg:alt_basis_APA_exact}, including the evaluation of the symbolic encoding matrices and the application of the fast basis transformations, incurs this floating-point error. In addition, we assume that the coefficients $\alpha_1, \ldots, \alpha_p$ are computed to high \textit{relative} accuracy in line one  -- i.e., if $\widehat{\alpha}_1, \ldots,\widehat{\alpha}_p$ are the outputs of \textbf{Vandermonde-Solve}$(\varepsilon_1, \ldots, \varepsilon_p)$ in floating-point arithmetic and $\alpha_1, \ldots, \alpha_p$ are the corresponding exact quantities, then
    \begin{equation}\label{eqn: Vandermonde_error_assumption}
        \frac{|\widehat{\alpha}_j - \alpha_j|}{|\alpha_j|} \leq  \delta 
    \end{equation}
for some small $\delta > 0$. 
\cref{prop: evaluate_vandermonde_directly} demonstrates that this holds for $\delta = O(\mathbf{u})$, under some mild conditions, if the $\alpha_j$ are obtained by evaluating \eqref{eqn: vandermonde_solution} directly. Of course, this is not the only way to handle the first step of \cref{alg:alt_basis_APA_exact} (see for example \cite{Higham_vandermonde,Higham_book}) though it is sufficient for our purposes. \\  
\indent Ultimately, the performance of \cref{alg:alt_basis_APA_exact} does not depend critically on how the Vandermonde system is solved. After all, the coefficients $\alpha_1, \ldots, \alpha_p$ could be computed ahead of time in higher precision if needed. Nevertheless, \cref{prop: evaluate_vandermonde_directly} foreshadows how the stability of alternative basis APA matrix multiplication depends on the choice of $\varepsilon_1, \ldots, \varepsilon_p$.  In the worst cases -- for example, when the $\varepsilon_i$ are large/small or tightly clustered -- it will be difficult to avoid under/overflow when evaluating \eqref{eqn: vandermonde_solution}. On the flip side, if the $\varepsilon_i$ are roots of unity then we can skip \textbf{Vandermonde-Solve} altogether, as the corresponding coefficients $\alpha_j = 1/p$ are known exactly. 

\indent The final ingredient of our analysis is an error bound of Schwartz et al.\ \cite{Schwartz2026_journal}, which covers a single evaluation of an alternative basis APA matrix multiplication routine (that is, \cref{alg:alt_basis_APA}). To state this result, we need the following definitions. Note that \cref{defn:stability_factor} and \cref{defn:stability_prefactor} generalize \cite[Definition 3.4]{Schwartz2026_journal} and \cite[Definition 3.2]{Schwartz2026_journal}, respectively.

\begin{defn}\label{defn:star_operation}
    The \textit{error accumulation operator} $\star$ takes as input a vector $ v \in {\mathbb C}^m$ and a matrix $M \in {\mathbb C}^{m \times n}$ and outputs $v \star M \in {\mathbb C}^n$ satisfying 
    \[v \star M = \begin{pmatrix}
        \text{nnz}(M(:,1) ) + \|v \odot {\mathbb I}(M(:,1)) \|_{\infty}\\
        \text{nnz}(M(:,2)) +  \|v \odot {\mathbb I}(M(:,2))\|_{\infty}\\
        \vdots \\
        \text{nnz}(M(:,n))+
       \| v \odot {\mathbb I}(M(:,n)) \|_{\infty}
    \end{pmatrix}, \]
    where $\odot$ is the Hadamard product and where  $\text{nnz}(w) \in {\mathbb N}$ and ${\mathbb I}(w) \in {\mathbb R}^n$ are, respectively, the number of nonzero entries in the vector $w \in {\mathbb C}^n$ and a vector that records whether each entry of $w$ is nonzero -- that is, the $i$-th entry of ${\mathbb I}(w)$ is one if $w_i$ is nonzero and zero otherwise. Following \cite{Schwartz2026_journal}, we also use the shorthand notation $\star M \coloneqq 0 \star M$ for  $0$ the zero vector. 
\end{defn}

\begin{defn}\label{defn:stability_factor}
    Let $U_{\phi}(\varepsilon)$, $V_{\psi}(\varepsilon)$, and $W_{\nu}(\varepsilon)$ be the encoding matrices of an alternative basis APA matrix multiplication algorithm with associated basis transformations $\phi_{\varepsilon}$, $\psi_{\varepsilon}$, and $\nu_{\varepsilon}^{-1}$. Define the vectors $a_{\varepsilon}$ and $b_{\varepsilon}$ as follows:
    \[a_{\varepsilon} = \left(|\phi_{\varepsilon}| \cdot |U_{\phi}(\varepsilon)|\right)^T \mathds{1} \; \; \; \; \text{and} \; \; \; \; \; b_{\varepsilon}= \left(|\psi_{\varepsilon}| \cdot |V_{\psi}(\varepsilon)|\right)^T \mathds{1}, \]
    where $\mathds{1}$ is a vector of all ones (of the appropriate dimension). The algorithm's \textit{stability vector} $e_{\varepsilon}$ is then 
    \[ e_{\varepsilon} = \left(|\nu_{\varepsilon}^{-1}| \cdot |W_{\nu}(\varepsilon)| \right) \left(a_{\varepsilon} \odot b_{\varepsilon}\right)\]
    where $\odot$ represents the Hadamard product. The corresponding \textit{stability factor} is $E(\varepsilon) = \|e_{\varepsilon}\|_{\infty}$.
\end{defn}

\begin{defn}\label{defn:stability_prefactor}
    Let $U_{\phi}(\varepsilon)$, $V_{\psi}(\varepsilon)$, and $W_{\nu}(\varepsilon)$ be the encoding matrices of an alternative basis APA matrix multiplication algorithm with associated basis transformations $\phi_{\varepsilon}$, $\psi_{\varepsilon}$, and $\nu_{\varepsilon}^{-1}$. Define the vector
    \[ q_{\varepsilon} = \left( \left( \star \phi_{\varepsilon} \star U_{\phi}(\varepsilon) \right) + \left( \star \psi_{\varepsilon} \star V_{\psi}(\varepsilon)\right) \right) \star W_{\nu}(\varepsilon)^T \star \nu_{\varepsilon}^{-T}.\]
     The \textit{stability prefactor} of the algorithm is then $Q(\varepsilon) = \|q_{\varepsilon}\|_{\infty}$.
\end{defn}

Note that the stability factor $E(\varepsilon)$ can be used to bound the norm of the output of a recursive APA algorithm. Indeed, if $C$ is the output of \cref{alg:alt_basis_APA} (which defaults to exact matrix multiplication after $L$ recursive steps) in exact arithmetic, then 
\begin{equation}\label{eqn: output_norm_bound}
    \|C\| \leq (k/k_0^L)E(\varepsilon)^L\|A\|\|B\|
\end{equation}
for $\|\cdot \|$ the maximum-entry norm. This follows from \eqref{eqn: alternative_basis_relation}, applied recursively, and was similarly observed by Schwartz et al.\ in the proof of \cite[Theorem 4.5]{Schwartz2026_journal}.\\
\indent The focus of \cite{Schwartz2026_journal} is  alternative basis \textit{exact} matrix multiplication (in the same model of floating-point arithmetic and using analogous definitions of stability factor/prefactor). Since their analysis does not assume that the bilinear form being evaluated is exact matrix multiplication, their bound immediately extends to the APA setting. Only one modification is necessary:\ the encoding matrices, basis transformations, and stability factor/prefactor now depend on the evaluation parameter $\varepsilon$. We state a modified version of their main result below. 

\begin{thm}[Schwartz et al.\ \cite{Schwartz2026_journal}] \label{thm: single_evaluation_bound}
Let $A \in {\mathbb C}^{m \times k}$ and $B \in {\mathbb C}^{k \times n}$ and suppose that $\widehat{C}$ and $C$ are the floating-point and exact arithmetic outputs of a recursive, alternative basis APA matrix multiplication algorithm (i.e., \cref{alg:alt_basis_APA}) evaluated at $\varepsilon \in {\mathbb C}$. If the base case of this algorithm multiplies $m_0 \times k_0$ and $k_0 \times n_0$ matrices, then we have
\[ \|\widehat{C} - C\| \leq \left( k/k_0^L + Q(\varepsilon)L\right) \left(k/k_0^L \right) E(\varepsilon)^L\|A\|\|B\| \text{\normalfont \bf u} + O(\text{\normalfont \bf u}^2),\]
where $L$ is the required number of recursive steps, $\|\cdot\|$ is the maximum entry norm, and $E(\varepsilon)$ and $Q(\varepsilon)$ are the stability factor and prefactor of the algorithm, respectively.     
\end{thm}

\subsection{Main Error Bound}

We now prove the main result of this section, which exhibits (weak) forward stability for \cref{alg:alt_basis_APA_exact} under some mild assumptions on the inputs and the two subroutines \textbf{Vandermonde-Solve} and \textbf{AltBasis-Recursive-APA}$_{\langle \phi, \psi, \nu \rangle}$. The main insight of the proof is that \cref{thm: single_evaluation_bound} can be used to bound the error in a single evaluation of \textbf{AltBasis-Recursive-APA}$_{\langle \phi, \psi, \nu \rangle}$ and moreover that this implies stability for \cref{alg:alt_basis_APA_exact} when combined with the following observations:
\begin{itemize}
    \item In the APA setting, the encoding matrices and basis transformations have functional entries, which must be evaluated at each evaluation point $\varepsilon_j$. Nevertheless, if these entries are computed to high relative accuracy then only minor additional error is introduced (see \cref{prop: evaluating an approximate bilinear algorithm}).
    \item The sum reduction in line 5 is similarly unproblematic if \textbf{Vandermonde-Solve} satisfies \eqref{eqn: Vandermonde_error_assumption}. 
\end{itemize}
Accordingly, \cref{thm: main_APA_error_bound} is stated in terms of two parameters $\delta_1$ and $\delta_2$, which denote how accurately (in terms of relative error) the solution to the Vandermonde system and the entries of the encoding matrices/basis transformations are computed. We choose to leave these as tunable parameters since, in practice, both the $\alpha_j$ and the entries of $U_{\phi}(\varepsilon_j)$, $V_{\psi}(\varepsilon_j)$, $W_{\nu}(\varepsilon_j)$, $\phi_{\varepsilon_j}$, $\psi_{\varepsilon_j}$ and $\nu_{\varepsilon_j}^{-1}$ can be computed offline to whatever precision is desired (without changing the asymptotic cost of \cref{alg:alt_basis_APA_exact}). Recall that \cref{prop: evaluate_vandermonde_directly} implies that $\delta_1$ can be replaced with $3p\mathbf{u} + O(\mathbf{u^2})$ if  we simply evaluate \eqref{eqn: vandermonde_solution} directly.

\begin{thm}\label{thm: main_APA_error_bound}
    Let $A \in {\mathbb C}^{m \times k}$ and $B \in {\mathbb C}^{k \times n}$ with $C = AB$. Suppose that  $\widehat{C}$ is the output of a floating-point implementation of \cref{alg:alt_basis_APA_exact} under the following conditions:
    \begin{enumerate}
        \item Both the entries of $A$ and $B$ and the evaluation parameters $\varepsilon_1, \ldots, \varepsilon_p \in {\mathbb C}$ are normalized floating-point numbers. In addition, the $\varepsilon_i$ are nonzero with $\varepsilon_i \neq \varepsilon_j$ if $i \neq j$. 
        \item \text{\normalfont \bf Vandermonde-Solve} satisfies the relative error bound \eqref{eqn: Vandermonde_error_assumption} for some small constant $\delta_1 > 0$. 
        \item \text{\normalfont \bf AltBasis-Recursive-APA}$_{\langle \phi,\psi,\nu \rangle}$ implements a recursive, alternative basis APA matrix multiplication algorithm, defined by encoding matrices $U_{\phi}(\varepsilon)$, $V_{\psi}(\varepsilon)$, $W_{\nu}(\varepsilon)$ and symbolic basis transformations $\phi_{\varepsilon}$, $\psi_{\varepsilon}$, $\nu_{\varepsilon}^{-1}$. The base case of this algorithm, which is reached after $L \in {\mathbb N}$ recursive steps for 
        \[ L \leq \min \left\{ \log_{m_0}(m), \log_{k_0}(k), \log_{n_0}(n) \right\}, \]
        defaults to classical matrix multiplication. We additionally assume that $m_0^L$, $k_0^L$, and $n_0^L$ divide $m$, $k$, and $n$, respectively.
        \item For each $\varepsilon_i$, the nonzero entries of $U_{\phi}(\varepsilon_i)$, $V_{\psi}(\varepsilon_i)$, $W_{\nu}(\varepsilon_i)$, $\phi_{\varepsilon_i}$, $\psi_{\varepsilon_i}$, and $\nu_{\varepsilon_i}^{-1}$ belong to the normalized range of the floating-point system and are computed to relative accuracy $\delta_2$ for some $0 < \delta_2 < 1$. 
        \item The sum reduction in line 5 is performed (entry-wise) via the recursive, pairwise algorithm for inner products (see \cite[Section 3.1]{Higham_book}). 
        \item No over/underflow occurs during the computation. 
    \end{enumerate}
    Then we have 
    \[ \|\widehat{C}-C\| \leq  (k/k_0^L)\left[\sum_{j=1}^p \left(\prod_{i \neq j} \frac{|\varepsilon_i|}{|\varepsilon_i-\varepsilon_j|} \right) \left( g(\delta_1,\delta_2)f(\varepsilon_j)\mathbf{u} + \left( g(\delta_1,\delta_2)-1 \right)\right) E(\varepsilon_j)^L \right] \|A\| \|B\| +O(\text{\normalfont \bf u}^2)\]
    for 
    \[ f(\varepsilon_j) =  (k/k_0^L) + Q(\varepsilon_j)L  +  \lceil \log_2p \rceil +1 \quad \text{and} \quad g(\delta_1,\delta_2) = (1+\delta_2)^{6L}(1+\delta_1),  \]
    where $\| \cdot \|$ is the maximum-entry norm, $\mathbf{u}$ is the unit roundoff,  and $E(\varepsilon)$ and $Q(\varepsilon)$ are the stability factor and prefactor, respectively, of \text{\normalfont \bf AltBasis-Recursive-APA}$_{\langle \phi,\psi,\nu \rangle}$. 
\end{thm}
\begin{proof}
    Let $\widehat{\alpha}_1, \ldots, \widehat{\alpha}_p \in {\mathbb C}$ and $\widehat{C}_j \in {\mathbb C}^{m \times n}$ be the floating-point output of \textbf{Vandermonde-Solve}$(\varepsilon_1,\ldots,\varepsilon_p)$ and \textbf{AltBasis-Recursive-APA}$_{\langle \phi, \psi, \nu \rangle}(A,B, \varepsilon_j)$, respectively. The latter corresponds to encoding matrices $\widetilde{U}_{\phi}(\varepsilon_j)$, $\widetilde{V}_{\psi}(\varepsilon_j)$, $\widetilde{W}_{\nu}(\varepsilon_j)$ and basis transformations $\tilde{\phi}_{\varepsilon_j}$, $\tilde{\psi}_{\varepsilon_j}$, and $\tilde{\nu}_{\varepsilon_j}^{-1}$, which can be obtained from $U_{\phi}(\varepsilon_j)$, $V_{\psi}(\varepsilon_j)$, $W_{\nu}(\varepsilon_j)$, $\phi_{\varepsilon_j}$, $\psi_{\varepsilon_j}$, and $\nu_{\varepsilon_j}^{-1}$ by applying entry-wise relative perturbations of size at most $\delta_2$. If $\sum_{j=1}^p \widehat{\alpha}_j \widehat{C}_j$ is the \textit{exact} linear combination of the floating-point quantities $\widehat{\alpha}_j$ and $\widehat{C}_j$, and  if $\widetilde{C}_j$ is the exact-arithmetic output of the recursive, alternative basis bilinear algorithm corresponding to $\widetilde{U}_{\phi}(\varepsilon_j)$, $\widetilde{V}_{\psi}(\varepsilon_j)$, $\widetilde{W}_{\nu}(\varepsilon_j)$, $\tilde{\phi}_{\varepsilon_j}$, $\tilde{\psi}_{\varepsilon_j}$, and $\tilde{\nu}_{\varepsilon_j}^{-1}$, then we have
    \begin{equation} \label{eqn: bound_one_triangle}
        \aligned 
        \|\widehat{C}-C \| &= \left\|\widehat{C} - \sum_{j=1}^p \widehat{\alpha}_j\widehat{C}_j + \sum_{j=1}^p \widehat{\alpha}_j\widehat{C}_j - \sum_{j=1}^p \alpha_j C_j\right\| \leq \left\| \widehat{C} - \sum_{j=1}^p \widehat{\alpha}_j \widehat{C}_j \right\| + \sum_{j=1}^p \|\widehat{\alpha}_j \widehat{C}_j - \alpha_j C_j \| ,
        \endaligned
\end{equation}
where the second term can be expanded as follows:
\begin{equation}\label{eqn: apa_expansion}
    \aligned
     \sum_{j = 1}^p \|\widehat{\alpha}_j \widehat{C}_j - \alpha_j C_j\| & =  \sum_{j=1}^p \|\widehat{\alpha}_j \widehat{C}_j -  \widehat{\alpha}_j \widetilde{C}_j + \widehat{\alpha}_j \widetilde{C}_j  - \alpha_j \widetilde{C}_j + \alpha_j \widetilde{C}_j - \alpha_j C_j \| \\
    & \leq   \sum_{j=1}^p \left( |\widehat{\alpha}_j| \|\widehat{C}_j - \widetilde{C}_j\| + |\widehat{\alpha}_j - \alpha_j| \|\widetilde{C}_j\|  + |\alpha_j| \|\widetilde{C}_j - C_j \| \right) \\
    & \leq  \sum_{j = 1}^p 
    \left( |\alpha_j| \left(  \|\widehat{C}_j - \widetilde{C}_j \| +\|\widetilde{C}_j - C_j \| \right) + |\widehat{\alpha}_j-\alpha_j| \left( \|\widetilde{C}_j\| + \|\widehat{C}_j - \widetilde{C}_j \| \right) \right) \\
    & \leq \sum_{j = 1}^p 
    \left( |\alpha_j| \left(  \|\widehat{C}_j - \widetilde{C}_j \| +\|\widetilde{C}_j - C_j \| + \delta_1 \left( \|\widetilde{C}_j\| + \|\widehat{C}_j - \widetilde{C}_j \| \right) \right) \right) \\
    & =  \sum_{j = 1}^p 
    \left( |\alpha_j| \left(  (1+\delta_1)\|\widehat{C}_j - \widetilde{C}_j \| +\|\widetilde{C}_j - C_j \| + \delta_1 \|\widetilde{C}_j\| \right) \right).
    \endaligned 
\end{equation}
To similarly bound the first term of \eqref{eqn: bound_one_triangle}, we note that each entry of $\widehat{C}$ can be interpreted as a floating-point inner product between the vector $\widehat{\alpha} \in {\mathbb C}^p$ of approximate coefficients and a vector of corresponding entries of $\widehat{C}_1, \ldots, \widehat{C}_p$. Since we assume that the recursive, pairwise algorithm is used to compute this inner product, we have 
\begin{equation} \label{eqn: triangle_part_two}
    \left\|\widehat{C} - \sum_{j=1}^p \widehat{\alpha}_j\widehat{C}_j \right\| \leq (\lceil \log_2p\rceil + 1) \left(\sum_{j=1}^p|\widehat{\alpha}_j| \|\widehat{C}_j\|\right) \textbf{u}  + O(\textbf{u}^2),
\end{equation}
where we have for all $1 \leq j \leq p$
\begin{equation}\label{eqn: term_one_bound}
    \aligned
    |\widehat{\alpha}_j|\|\widehat{C}_j\| & \leq \left( |\alpha_j| + |\widehat{\alpha}_j - \alpha_j| \right) \left( \|\widetilde{C}_j\| + \|\widehat{C}_j - \widetilde{C}_j\| \right) \leq  |\alpha_j| (1+\delta_1) \left( \|\widetilde{C}_j \| + \|\widehat{C}_j - \widetilde{C}_j \|  \right).
    \endaligned 
\end{equation}
Combining \eqref{eqn: apa_expansion} with \eqref{eqn: triangle_part_two} and \eqref{eqn: term_one_bound} yields 
\begin{equation}\label{eqn: putting_together}
    \resizebox{.9\textwidth}{!}{
    $\displaystyle
    \|\widehat{C} - C  \| \leq \sum_{j=1}^p |\alpha_j| \left( (1+\delta_1)\left( (\lceil \log_2p \rceil+1) \mathbf{u} + 1\right)  \|\widehat{C}_j - \widetilde{C}_j \|+ \left((1+\delta_1)(\lceil \log_2p \rceil + 1)\mathbf{u} + \delta_1 \right) \| \widetilde{C}_j \| + \|\widetilde{C}_j - C_j \|\right) + O(\mathbf{u}^2).
    $}
\end{equation}
Now \cref{thm: single_evaluation_bound} guarantees
\begin{equation}
    \| \widehat{C}_j-\widetilde{C}_j\| \leq (k/k_0^L+\widetilde{Q}(\varepsilon_j)L)(k/k_0^L)\widetilde{E}(\varepsilon_j)^L \|A\|\|B\| \textbf{u} + O(\textbf{u}^2)
\end{equation}
for $\widetilde{E}(\varepsilon_j)$ and $\widetilde{Q}(\varepsilon_j)$  the stability factor and prefactor corresponding to $\widetilde{U}_{\phi}(\varepsilon_j)$, $\widetilde{V}_{\psi}(\varepsilon_j)$, $\widetilde{W}_{\nu}(\varepsilon_j)$, $\tilde{\phi}$, $\tilde{\psi}$, and $\tilde{\nu}^{-1}$. Moreover, recalling \eqref{eqn: output_norm_bound}, $\|\widetilde{C}_j\| \leq (k/k_0^L)\widetilde{E}(\varepsilon_j)^L \|A\|\|B\| $. Finally, \cref{prop: evaluating an approximate bilinear algorithm} implies
\begin{equation}
    \|\widetilde{C}_j - C_j\| \leq \left( (1+\delta_2)^{6L} - 1 \right) (k/k_0^L)E(\varepsilon_j)^L \|A\|\|B\|.
\end{equation}
We complete the proof by applying these to  \eqref{eqn: putting_together} and simplifying, inserting the definition of $\alpha_j$, and noting $\widetilde{E}(\varepsilon_j) \leq (1+\delta_2)^6 E(\varepsilon_j)$ and $\widetilde{Q}(\varepsilon_j) = Q(\varepsilon_j)$ (since entry-wise relative perturbations of size $0 < \delta_2 < 1$ preserve the sparsity structure of the encoding matrices/basis transformations).
\end{proof}

\subsection{Parameter Selection for Improved Stability}

\cref{thm: main_APA_error_bound} suggests that, from the point of view of stability, a good choice for the evaluation parameters $\varepsilon_1,\ldots, \varepsilon_p$ is the $p$-th roots of unity -- i.e., $\varepsilon_j = e^{2 \pi i j/p}$. In this case, the factor $\prod_{i \neq j} \frac{|\varepsilon_i|}{|\varepsilon_i-\varepsilon_j|}$ simplifies to  $1/p$, and we do not have to worry about $E(\varepsilon_j)$ blowing up due to large positive/negative powers of $\varepsilon$ in the encoding matrices or basis transformations.  \\
\indent There is, however, one catch:\ the roots of unity, with the exception of $\pm 1$ and $\pm i$, are \textit{not} floating-point numbers, and hence \cref{thm: main_APA_error_bound} does not apply (in general) when $\varepsilon_j = e^{2 \pi i j/p}$. Instead, the best we can do is use the rounded roots of unity $\tilde{\varepsilon}_1, \ldots, \tilde{\varepsilon}_p$, which take the form $\tilde{\varepsilon}_j = e^{2 \pi i j/p} \cdot (1+\delta_j)$ for some $|\delta_j| \leq \textbf{u}$. Importantly, for $p$ not too large, these approximate roots satisfy $\prod_{i \neq j} \frac{|\tilde{\varepsilon}_i|}{|\tilde{\varepsilon}_i - \tilde{\varepsilon}_j|} = 1/p + O(\textbf{u})$. In particular, applying \cref{prop: vandermonde_perturbation} with $\delta = \textbf{u}$ and since the spacing between the $p$-th roots of unity is $2\sin(\pi/p)$, we have the bound
\begin{equation}
    \frac{| \tilde{\alpha}_j-1/p|}{1/p} \leq \frac{(p-1)\textbf{u}}{\sin(\pi/p)-p\textbf{u}} 
\end{equation}
for $\tilde{\alpha}_1, \ldots, \tilde{\alpha}_p$ the coefficients that solve the Vandermonde system corresponding to $\tilde{\varepsilon}_1, \ldots, \tilde{\varepsilon}_p$. \\
\indent This implies the following corollary of \cref{thm: main_APA_error_bound}, which assumes (1) that rounded roots of unity are used as evaluation points in \cref{alg:alt_basis_APA_exact}, and (2) that $\tilde{\alpha}_1, \ldots, \tilde{\alpha}_p$ and the entries of the encoding matrices/basis transformations at $\tilde{\varepsilon}_1, \ldots, \tilde{\varepsilon}_p$ are computed to high enough relative accuracy that $g(\delta_1,\delta_2) \leq 1 + cp \mathbf{u} + O(\mathbf{u}^2)$ for some small constant $c > 0$. This specific bound is motivated by \cref{prop: evaluate_vandermonde_directly}. 

\begin{cor}[Simplified version of \cref{thm: main_APA_error_bound}]\label{thm: simplified_APA_bound}
    Let $A, B \in {\mathbb C}^{n \times n}$ with $C = AB$ and let $ALG$ be a floating-point, alternative basis matrix multiplication routine, which implements recursively a degree $d$ base-case algorithm for $k \times k$ approximate matrix multiplication (where we assume $n$ is a power of $k$ and accordingly that the algorithm is applied over $L = \log_k(n)$ recursive steps). Suppose that $\widehat{C}$ is the output of the approximate-to-exact procedure  of Bini \cite{Bini1980} applied to $ALG$ in floating-point arithmetic, taking the evaluation parameters $\varepsilon_1, \ldots, \varepsilon_p$ to be floating-point approximations of the $p$-th roots of unity, where $p = d \log_kn + 1$. Assume further that all conditions of \cref{thm: main_APA_error_bound} hold, with in particular the solution to the corresponding Vandermonde system and the entries of the encoding matrices/basis transformations (at each $\varepsilon_j$) computed to sufficiently high relative accuracy to guarantee $g(\delta_1,\delta_2) \leq  1 + cp\mathbf{u} + O(\mathbf{u}^2)$ for a  constant $c > 0$ and $\mathbf{u}$ the unit roundoff.  In this case, we have
    \[ \|\widehat{C}-C\| \leq f_{ALG}(n)\|A\|\|B\| \text{\normalfont  \bf u} + O(\text{\normalfont \bf u}^2) \]
    for 
    \[f_{ALG}(n) =\frac{1}{p} \sum_{j=1}^p \left(2 + Q(\varepsilon_j)\log_kn + \lceil \log_2(p) \rceil + cp   \right)n^{\log_kE(\varepsilon_j)}, \]
    where $\|\cdot \|$ is the maximum entry norm and $Q(\varepsilon)$ and $E(\varepsilon)$ are the stability prefactor and factor, respectively, of $ALG$ (see \cref{defn:stability_factor} and \cref{defn:stability_prefactor}).
\end{cor}

This result suggests that the (optimal) stability of an approximate-to-exact APA matrix multiplication algorithm is characterized by the behavior of the corresponding stability factor/prefactor near the unit circle.  \\
\indent Before moving on, we note that using approximate roots of unity for $\varepsilon_1, \ldots, \varepsilon_p$ as suggested above will necessarily turn any APA algorithm into one with complex encoding matrices, which may considerably drive up the arithmetic cost of the algorithm (by a constant factor). This marks an efficiency/stability tradeoff for users to navigate. Of course, the bounds proved in this section are potentially loose, meaning that in practice, there may be little performance difference (in terms of accuracy) between \cref{alg:alt_basis_APA_exact} with approximate roots of unity and real $\varepsilon_1, \ldots, \varepsilon_p$. 

\subsection{Extensions}\label{section: extension}

A drawback of \cref{alg:alt_basis_APA_exact} is the need to evaluate \textbf{AltBasis-Recursive-APA}$_{\langle \phi,\psi,\nu \rangle}$ at $p = dL+1$ points, which ultimately leads to a logarithmic factor increase in its complexity. 
To get around this, we consider here an alternative, \textit{truncated} version of \cref{alg:alt_basis_APA_exact}, which evaluates \textbf{AltBasis-Recursive-APA}$_{\langle \phi,\psi,\nu \rangle}$ a fixed $r < p$ times and takes a linear combination of the outputs to cancel the first $r-1$ terms of the remainder (see \cref{alg:alt_basis_APA_exact_truncated}).

\begin{algorithm}[h]
\caption{Truncated Approximate-to-Exact Alternative Basis APA Matrix Multiplication\\
\textbf{Input:} $A \in {\mathbb C}^{m \times k}$, $B \in {\mathbb C}^{k \times n}$ and parameters $ \varepsilon_1, \ldots,\varepsilon_r \in {\mathbb C}$, where $|\varepsilon_j| < 1$ and $r \leq dL+1$. \\
\textbf{Output:} $C = AB + \left( \sum_{j=1}^r \alpha_j \varepsilon_j^r\right)R_r + \cdots + \left( \sum_{j=1}^r \alpha_j \varepsilon_j^{p-1} \right)R_{p-1}$. }\label{alg:alt_basis_APA_exact_truncated}
\begin{algorithmic}[1]
\STATE{$[\alpha_1, \ldots, \alpha_r] = $ \textbf{Vandermonde-Solve}$(\varepsilon_1, \ldots, \varepsilon_r)$}
\FOR{$j = 1:r$}
    \STATE $C_j = $ \textbf{AltBasis-Recursive-APA}$_{\langle \phi, \psi, \nu \rangle}(A,B,\varepsilon_j)$
\ENDFOR
\STATE{$C = \alpha_1C_1 + \alpha_2C_2 + \cdots + \alpha_rC_r$}
\RETURN{$C$}
\end{algorithmic}
\end{algorithm}

If the evaluation parameters $\varepsilon_1, \ldots, \varepsilon_r$ satisfy $|\varepsilon_j| < 1$ for all $j$, we can expect that the output of this algorithm will be close to $C$, as the portion of the remainder left over depends on powers of $\varepsilon$ that are $r$ or larger. Again, this trades stability for performance; in this setting, we have both a nonvanishing remainder \textit{and} evaluation parameters that are more tightly spaced (the latter increasing the error incurred by \textbf{Vandermonde-Solve} in floating point). Nevertheless, the improved efficiency is significant, particularly when $r = O(1)$. \\
\indent To derive an error bound for \cref{alg:alt_basis_APA_exact_truncated}, we first need the following technical lemma, which bounds the norms of the residual matrices $R_1, \ldots, R_{p-1}$.

\begin{lem}\label{lemma: residual_matrix_bound}
    Let $R_l \in {\mathbb C}^{m \times n}$, $1 \leq l \leq p-1$ be the residual matrices associated to a recursive, alternative basis APA matrix multiplication routine (see \cref{alg:alt_basis_APA}) with input matrices $A \in {\mathbb C}^{m \times k}$ and $B \in {\mathbb C}^{k \times n}$. Then 
    \[ \|R_l \| \leq \frac{(k/k_0^L)}{p} \left[ \sum_{j=1}^p E( e^{2 \pi ij/p})^L \right] \|A\|\|B\|, \]
    where $\|\cdot \|$ is the maximum entry norm and $E(\varepsilon)$ is the stability factor of the algorithm. 
\end{lem}
\begin{proof}
    Let $(\alpha_1^l,\ldots ,\alpha_p^l)^T \in {\mathbb C}^p$ be the solution to the Vandermonde system \eqref{eqn: vandermonde_system_d} with the right hand side replaced by the $(l+1)$-st standard basis vector and the evaluation parameters $\varepsilon_1, \ldots, \varepsilon_p$ taken to be the $p$-th roots of unity -- i.e., $\varepsilon_j = e^{2\pi ij/p}$. If $C_j$ is the exact-arithmetic output of one run of \cref{alg:alt_basis_APA} applied to $A$ and $B$ with evaluation parameter $\varepsilon_j$, then we have
    \[ R_l = \alpha_1^lC_1 + \alpha_2^l C_2 + \cdots + \alpha_p^l C_p. \]
    Recalling \eqref{eqn: output_norm_bound}, we can bound $\|C_j\|$ as $\|C_j\| \leq (k/k_0^L) E(\varepsilon_j)^L \|A\|\|B\|$ for each $1 \leq j \leq p$. Hence,
    \[ \|R_l \| \leq \sum_{j=1}^p |\alpha_j^l| \|C_j\| \leq \frac{(k/k_0^L)}{p} \left[\sum_{j=1}^p E(e^{2 \pi ij/p})^L \right] \|A\|\|B\|, \]
    where $|\alpha_j^l| = 1/p$ for all $1 \leq j \leq p$ since $\varepsilon_1, \ldots, \varepsilon_p$ are roots of unity. 
\end{proof}

If $C = AB$ and $\overline{C}$ is the exact-arithmetic output of \cref{alg:alt_basis_APA_exact_truncated} -- i.e., 
\begin{equation}
    \overline{C} = C + \left( \sum_{j=1}^r \alpha_j \varepsilon_j^r\right)R_r + \cdots + \left( \sum_{j=1}^r \alpha_j \varepsilon_j^{p-1} \right)R_{p-1},
\end{equation}
the previous lemma implies
\begin{equation}\label{eqn: residual_bound}
    \aligned 
    \|\overline{C} - C\| &= \left\| \left( \sum_{j=1}^r \alpha_j \varepsilon_j^r \right) R_r + \cdots +\left( \sum_{j=1}^r \alpha_j \varepsilon_j^{p-1} \right)R_{p-1} \right\| \\
    & \leq \left( \sum_{s = r}^{p-1} \sum_{j=1}^r |\alpha_j||\varepsilon_j|^s \right)\frac{(k/k_0^L)}{p} \left[ \sum_{j=1}^p E(e^{2 \pi ij/p})^L \right] \|A\| \|B\| \\
    & = \left( \sum_{s = r}^{p-1} \sum_{j=1}^r |\varepsilon_j|^s\prod_{i \neq j} \frac{|\varepsilon_i|}{|\varepsilon_i-\varepsilon_j|} \right)\frac{(k/k_0^L)}{p} \left[ \sum_{j=1}^p E(e^{2 \pi ij/p})^L \right] \|A\| \|B\|.
    \endaligned 
\end{equation}
The key insight now is that \cref{thm: main_APA_error_bound}, whose proof does not use the fact that the output of \cref{alg:alt_basis_APA_exact} is the exact product of $A$ and $B$, can be used to bound the error in a floating-point implementation of \cref{alg:alt_basis_APA_exact_truncated}. That is, if $\widehat{C}$ is the corresponding floating-point output (and assuming the same conditions as \cref{thm: main_APA_error_bound}), then 
\begin{equation}\label{eqn: putting_things_together}
    \aligned
    \|\widehat{C}&-C\|  \leq  \|\widehat{C} - \overline{C}\| + \|\overline{C} - C\| \\
    & \leq (k/k_0^L) \left[\sum_{j=1}^r \left(\prod_{i \neq j} \frac{|\varepsilon_i|}{|\varepsilon_i-\varepsilon_j|} \right) \left(g(\delta_1,\delta_2)f(\varepsilon_j) \mathbf{u} + (g(\delta_1,\delta_2)-1) \right)E(\varepsilon_j)^L \right] \|A\| \|B\| +O(\textbf{u}^2) \\
    & \hspace{1cm} + \left( \sum_{s = r}^{p-1} \sum_{j=1}^r |\varepsilon_j|^s\prod_{i \neq j} \frac{|\varepsilon_i|}{|\varepsilon_i-\varepsilon_j|} \right)\frac{(k/k_0^L)}{p} \left[ \sum_{j=1}^p E(e^{2 \pi ij/p})^L \right] \|A\| \|B\|.
    \endaligned 
\end{equation}
This decomposes the error in \cref{alg:alt_basis_APA_exact_truncated} as the sum of the loss of accuracy due to the truncation (term two) and the loss of accuracy due to floating-point arithmetic (term one). The latter can be made arbitrarily small by choosing $r$ and $\varepsilon, \ldots, \varepsilon_r$ appropriately. If, for example, we want 
\begin{equation}\label{eqn: epsilon_requirement}
    \sum_{s = r}^{p-1} \sum_{j=1}^r |\varepsilon_j|^s\prod_{i \neq j} \frac{|\varepsilon_i|}{|\varepsilon_i-\varepsilon_j|} = \Theta(\textbf{u})
\end{equation}
so that \eqref{eqn: putting_things_together} essentially implies (weak) forward stability for this algorithm, and additionally if we take $\varepsilon_j = \beta \cdot e^{2 \pi ij/r}$ for some $0 < \beta < 1$ (ignoring the fact that this is not a floating-point number) then \eqref{eqn: epsilon_requirement} reduces to
\begin{equation}
    \sum_{s = r}^{p-1} \beta^s = \Theta(\textbf{u}),
\end{equation}
which we can simplify further by noting
\begin{equation}
    \sum_{s = r}^{p-1} \beta^s \leq \sum_{s = r}^{\infty}\beta^s = \sum_{s = 0}^{\infty}\beta^s - \sum_{s=0}^{r-1} \beta^s = \frac{1}{1-\beta} - \frac{1-\beta^{r}}{1-\beta} = \frac{\beta^r}{1-\beta}.
\end{equation}
Hence, it is sufficient to choose $\beta$ and $r$ such that $\frac{\beta^r}{1-\beta} = \Theta( \textbf{u})$, a condition that we note is independent of $n$.

\section{Numerical Examples}\label{section: examples}

In this section, we present a collection of  numerical experiments aimed at verifying the theoretical results from the previous section. The first two explore the validity/tightness of our stability bounds (with and without alternative bases), while the third tests the effectiveness of the truncation idea discussed in \cref{section: extension}. All experiments were done in Matlab and follow a common setup described in \cref{subsection: setup}.  Importantly, these experiments do not demonstrate the performance of APA matrix multiplication in a high-performance environment, which we leave to future work. Takeaways include the following:
\begin{enumerate}
    \item \cref{thm: main_APA_error_bound} and \cref{thm: simplified_APA_bound}, though not tight, effectively gauge the (relative) stability of APA routines.
    \item While alternative basis algorithms have a weaker stability bound, they perform no worse in practice. 
    \item Truncation can significantly improve the run time of an APA matrix multiplication algorithm, trading mild accuracy loss for speed. 
\end{enumerate}


\subsection{Experimental Setup}\label{subsection: setup}
We implement several APA matrix multiplication routines, including some existing algorithms and some new ones found in our search. All implementations are recursive and run to the corresponding base case. That is, we do not default to classical or fast exact matrix multiplication at any point. Unless stated otherwise, we implement the full interpolation (i.e., \cref{alg:alt_basis_APA_exact}) using roots of unity to obtain an exact product (matching the setting of \cref{thm: simplified_APA_bound}). In each experiment, we measure (relative) forward error via $\|\widehat{C}-C\|/(\|A\|\|B\|)$, where $\widehat{C}$ is the computed value of $C = AB$ and $\|\cdot \|$ in the maximum-entry norm. 
Since we use random input matrices, each reported value of this error metric is the mean over ten independent draws of $A$ and $B$. \\
\indent All algorithms for computing $\widehat{C}$ are run in IEEE-754 binary64, whose corresponding unit roundoff is $\text{\normalfont \bf u} = 2^{-53}\approx 1.11\times 10^{-16}$. To get a surrogate for the exact product $C = AB$, we convert $A$ and $B$ into symbolic form and 
evaluate the product in $32$-digit arithmetic; relative error is then computed by evaluating $\widehat{C} - C$ symbolically before rounding to double precision. All measurements were taken on an Apple M4 running macOS 26.5.1, 
with experiments performed in MATLAB  version R2025b (25.2.0.3042426, Update 1) using the Symbolic Math Toolbox 25.2 and the Parallel Computing Toolbox 25.2. 
\subsection{Verifying \cref{thm: main_APA_error_bound}}\label{subsection: ex1}
We start by testing how accurately the stability bounds from \cref{section: stability} predict the practical performance of an APA algorithm. We specifically aim to answer the following questions: How tightly does \cref{thm: main_APA_error_bound} bound the measured error for different input distributions? Do these bounds ``rank" algorithms, in terms of stability, appropriately? \\ 
\indent For our first experiment, we test a $\langle 3,3,3; 21 \rangle$ APA algorithm found in our search against two existing APA algorithms for the same size -- i.e., a $\langle 3,3,3, 20 \rangle$ algorithm of Smirnov \cite{Smirnov2013} and a $\langle 3,3,3;21 \rangle$ algorithm of Schönhage \cite{Partial_to_Total} (all in the standard basis). The details of these algorithms are summarized in \cref{tab:exp1-schemes}. We also include the classical algorithm -- which is not just norm-wise stable but in fact entry-wise stable -- as a baseline. Input matrices $A$ and $B$ are drawn from the following three distributions: $\mathrm{Unif}(-1,1)$, $\mathrm{Unif}(0,1)$, and $\mathcal{N}(0,1)$.

\begin{table}[htbp]
  \centering
  \begin{tabular}{lccc}
    \toprule
    APA Algorithm & $\omega$ & Degree ($d$) \\
    \midrule
    Smirnov   $\langle 3,3,3,20 \rangle$ \cite{Smirnov2013}  & 2.7268  & 10 \\
    Sch\"onhage $\langle 3,3,3,21 \rangle$ \cite{Partial_to_Total}  & 2.7712  & 2 \\
    Ours  $\langle 3,3,3,21 \rangle$ & 2.7712       & 5 \\
    \bottomrule
  \end{tabular}
  \caption{Three APA algorithms for $\langle 3,3,3 \rangle $ matrix multiplication.}
  \label{tab:exp1-schemes}
\end{table}

\begin{figure}[t]
  \centering
  \includegraphics[width=.9\textwidth]{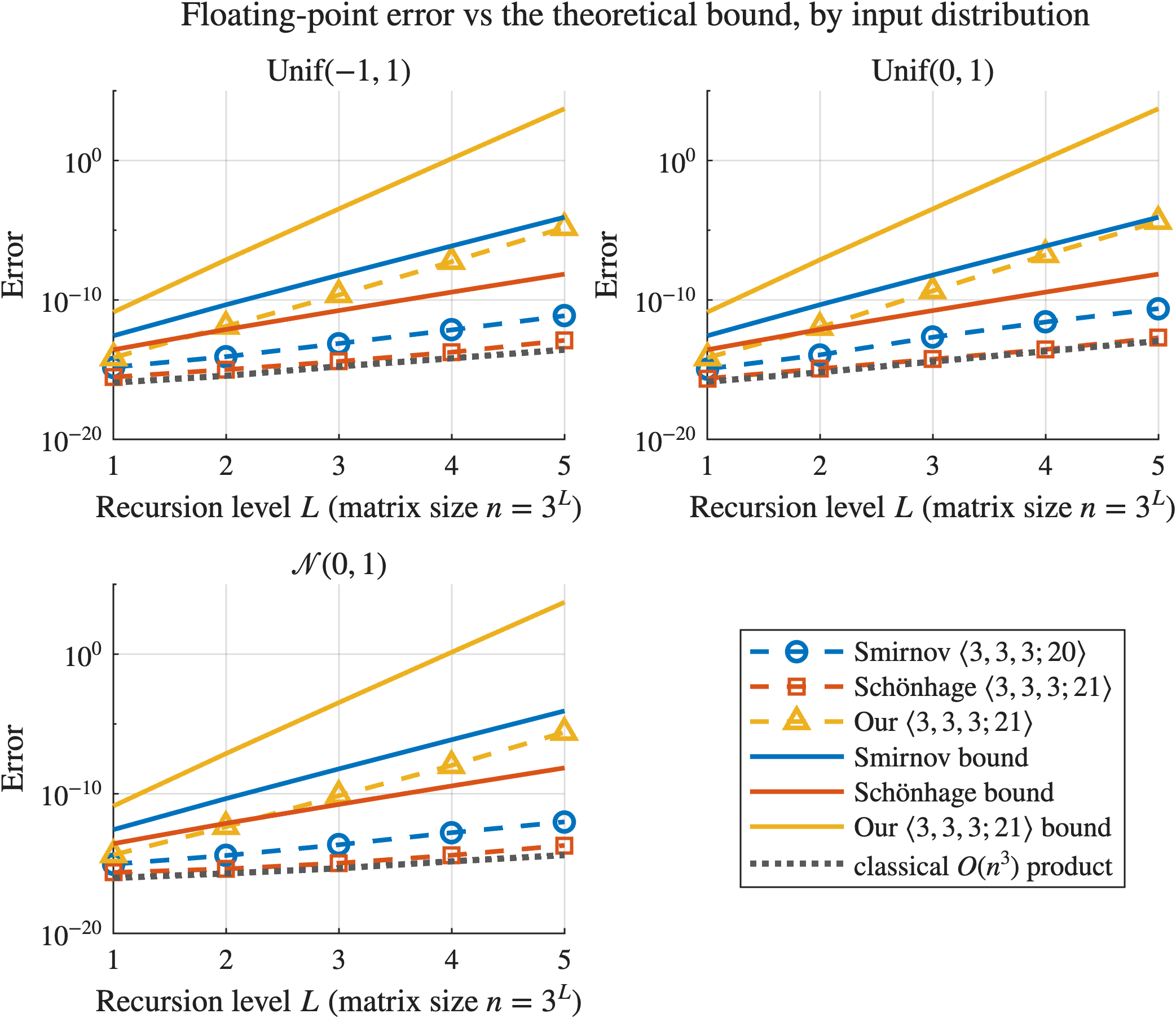}
  \caption{Measured and theoretical accuracy of three APA algorithms for $\langle 3,3,3 \rangle$ matrix multiplication applied to various input matrices, as measured by the normalized error $\|\widehat{C}-C\|/(\|A\|\|B\|)$ and compared to the bounds from \cref{thm: simplified_APA_bound}. The panels share a vertical scale, and the solid curves are identical in all three.}
  \label{fig:exp1-grid}
\end{figure}

\indent Results are shown in \cref{fig:exp1-grid}, where each panel corresponds to one input distribution. For each algorithm/input type, we plot both the empirical performance (with a dashed line) against the corresponding error bound from \cref{section: stability} (represented by a solid line). For reference, the accuracy of the classical $O(n^3)$ algorithm is shown with a dotted gray line. We see here that our main stability bound is descriptive and, moreover, that it accurately ranks the three algorithms in terms of expected accuracy.  \\
\indent An alternative version of these plots is presented in \cref{fig:exp1-dist}, where each panel now corresponds to a different algorithm and the plotted curves correspond to different choices of $A$ and $B$. In addition to the random input matrices discussed above, these plots include results for a ``worse-case" option.\footnote{We use the term ``worse-case" instead of ``worst-case" since we make no claim that these inputs are actually the most difficult for our algorithms.} Corresponding ``worse-case" inputs are algorithm dependent, constructed to maximize the error given a set of encoding matrices. To do this, we first construct (nonzero) inputs $A$ and $B$ with entries in $\{0,-1,1\}$ chosen to  maximize
\begin{equation}
  \label{eq:adv-objective}
  \max_{r} \; \sum_{k} \sum_{s}
      \bigl\lvert W_{rs}(\varepsilon_k) \bigr\rvert \;
      \bigl\lvert a^T U_s(\varepsilon_k) \bigr\rvert \;
      \bigl\lvert b^T V_s(\varepsilon_k) \bigr\rvert,
\end{equation}
where $U_s$ and $V_s$ are the $s$-th columns of the encoding matrices $U$ and $V$ (which describe the base-case algorithm) and $a$ and $b$ are vectorized copies of $A$ and $B$. Since the base case here is only $\langle 3,3,3 \rangle$, these matrices can be found via an exhaustive search. We then take Kronecker powers of $A$ and $B$ to obtain ``worse-case" inputs of arbitrary size. Recalling \eqref{eqn: encoding_relation}, the idea here is to encourage each intermediate product $C_j$ in \cref{alg:alt_basis_APA_exact} to have large entries despite the fact that $\|A\|\|B\|$ is fixed, thereby promoting potentially catastrophic cancellation in the sum reduction used to obtain $\widehat{C}$ from the intermediate products. Once again, we include the classical algorithm in \cref{fig:exp1-dist} as a baseline.

\begin{figure}[t]
  \centering
  \includegraphics[width=.9\textwidth]{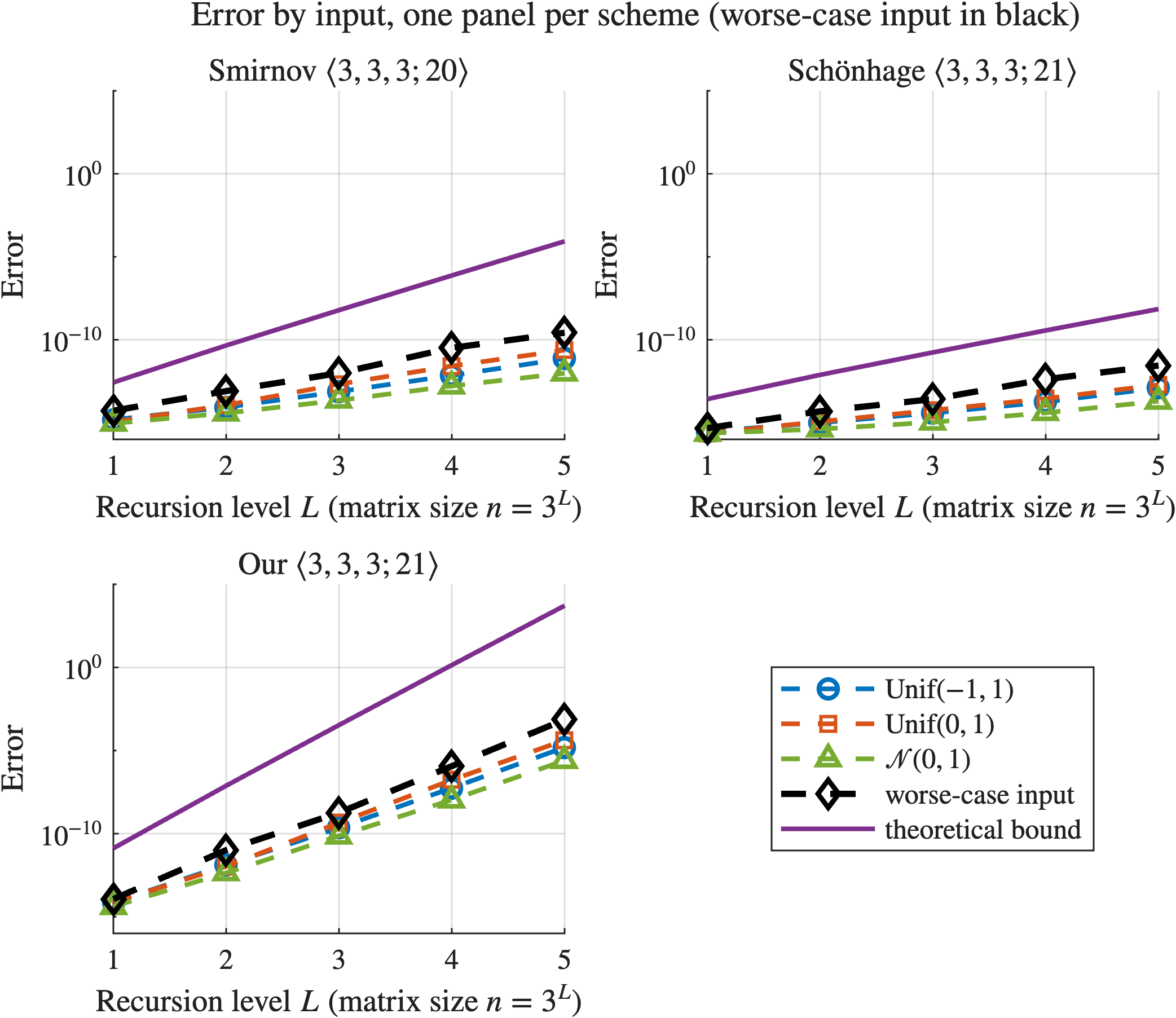}
  \caption{A repeat of \cref{fig:exp1-grid}, with results grouped by algorithm instead of input type. An additional ``worse-case" input is included here, constructed as described in the text. Note that all three panels share a vertical scale.}
  \label{fig:exp1-dist}
\end{figure}

\subsection{Stability with Alternative Bases}\label{subsection: ex2}
Alternative basis exact matrix multiplication is already known to be fast and stable \cite{Schwartz2026_journal}. \cref{section: stability} demonstrated that the theoretical bounds of \cite{Schwartz2026_journal} extend to APA matrix multiplication; here, we show that these algorithms are stable in practice as well.\\
\indent For this experiment, we test three algorithms corresponding to different sizes:\ Schönhage's $\langle 3,3,3; 21 \rangle$ and two algorithms found in our search: $\langle 2,2,3; 10\rangle$ and $\langle 2,2,5; 17 \rangle$.  
These were selected to show the range of potential stability loss predicted by \cref{thm: main_APA_error_bound}; the $\langle 2,2,5; 17 \rangle$ algorithm is the worst -- meaning its error bound grows the most when comparing with/without alternative bases -- while $\langle 3, 3,3; 21 \rangle$ shows the least growth and $\langle 2,2,3; 10 \rangle$ is somewhere in the middle.
For simplicity, we only compare each algorithm to itself with and without alternative bases, both run on the same set of input matrices $A$ and $B$.  As in the previous experiment, the entries of these matrices are drawn from $\mathcal{N}(0,1)$ and the reported error matches the setup described in \cref{subsection: setup}. \\
\indent \cref{fig:exp2} shows relative performance for all three algorithms. Note that the horizontal axes on these plots correspond to the inner dimension of the multiplication and not the number of recursive steps as in \cref{fig:exp1-grid,fig:exp1-dist}. Once again, empirical performance results are represented with dashed lines while the bounds implied by \cref{thm: main_APA_error_bound} are shown with solid lines. The gap between the latter indicates how much worse \cref{thm: main_APA_error_bound} predicts the stability of the algorithm to be under alternative bases. While this gap can be quite large (particularly for $\langle 2,2,5; 17 \rangle$, which, again, was chosen for this reason) we see essentially no difference in the practical performance. This is shown precisely in \cref{tab:exp2-ratio}, which confirms that the empirical results shown in \cref{fig:exp2} are not simply compressed due to the scale of the plots. In total this re-affirms an observation made in \cite{Schwartz2026_journal} for the fast exact setting:\ the predicted loss of stability under alternative bases is primarily an artifact of the analysis from \cref{section: stability}, and specifically the fact that the stability factor (i.e., \cref{defn:stability_factor}) is strictly larger when alternative bases are used.

\begin{figure}[t]
  \centering
  \includegraphics[width=0.9\textwidth]{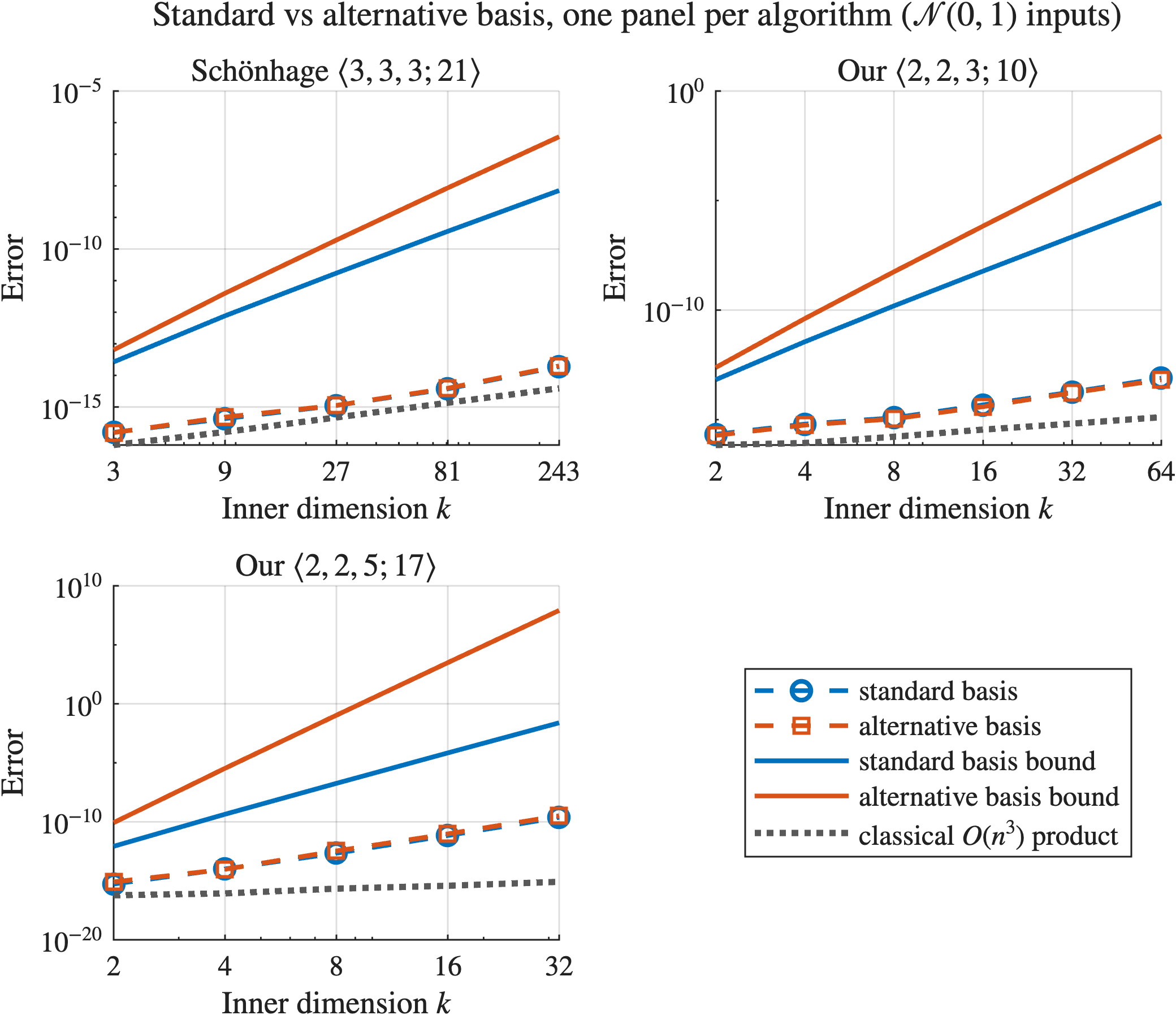}
  \caption{Empirical/predicted error for various APA matrix multiplications algorithms with and without alternative bases. As in \cref{fig:exp1-grid,fig:exp1-dist}, dashed lines mark practical results and solid lines represent the bound implied by \cref{thm: main_APA_error_bound}. Again, the error of the classical $O(n^3)$ algorithm is included for reference (in grey).}
  \label{fig:exp2}
\end{figure}

\begin{table}[htbp]
  \centering
  \begin{subtable}{0.48\textwidth}
    \centering
    \begin{tabular}{lcccccc}
      \toprule
      & \multicolumn{6}{c}{Inner Dimension ($k$)} \\
      &    3 &    9 &   27 &   81 &  243 &     \\
      \cmidrule(l){2-7}
      Bound     & 2.32 & 4.98 & 10.6 & 22.6 & 48.1 &     \\
      Empirical & 0.98 & 1.12 & 1.01 & 1.05 & 1.04 &     \\
      \bottomrule
    \end{tabular}
    \subcaption{$\langle 3,3,3;21 \rangle$}
  \end{subtable}
  \hfill
  \begin{subtable}{0.48\textwidth}
    \centering
    \begin{tabular}{lcccccc}
      \toprule
      & \multicolumn{6}{c}{Inner Dimension ($k$)} \\
      &    2 &    4 &    8 &   16 &   32 &   64 \\
      \cmidrule(l){2-7}
      Bound     & 3.22 & 10.1 & 31.7 & 99.2 &  311 &  974 \\
      Empirical & 0.90 & 0.90 & 0.90 & 0.83 & 0.93 & 0.85 \\
      \bottomrule
    \end{tabular}
    \subcaption{$\langle 2,2,3;10 \rangle$}
  \end{subtable}

  \vspace{1.7em}

  \begin{subtable}{0.6\textwidth}
    \centering
    \begin{tabular}{lcccccc}
      \toprule
      & \multicolumn{6}{c}{Inner Dimension ($k$)} \\
      &    2 &    4 &    8 &   16 &   32 &     \\
      \cmidrule(l){2-7}
      Bound     & 88.0 & $6.6{\times}10^{3}$ & $5.0{\times}10^{5}$
                & $3.8{\times}10^{7}$ & $2.8{\times}10^{9}$ &     \\
      Empirical & 1.52 & 0.93 & 1.56 & 1.38 & 1.39 &     \\
      \bottomrule
    \end{tabular}
    \subcaption{$\langle 2,2,5;17\rangle$}
  \end{subtable}

  \caption{Error ratios (either empirical or predicted by \cref{thm: main_APA_error_bound}) for the APA algorithms shown in \cref{fig:exp2}, obtained by dividing the error with alternative bases by the error in the standard basis. A value of one indicates the algorithm is equally accurate with/without alternative bases. 
    }
  \label{tab:exp2-ratio}
\end{table}

\subsection{Effectiveness of Truncation}\label{subsection: ex3}
For our final experiment, we explore the impact of the truncation approach (outlined in \cref{section: extension}) on the stability of an alternative basis APA algorithm. Accordingly, we run \cref{alg:alt_basis_APA_exact_truncated}, evaluating an APA algorithm a fixed $r$ times and taking $\varepsilon_j = \beta \cdot e^{2 \pi ij/r}$ for some $0 < \beta < 1$. The idea here is that, if $\beta$ is small enough to make the uncanceled remainder negligible but not so small that it introduces instability, we can get a decent approximation to the true product without executing the full interpolation procedure of \cref{alg:alt_basis_APA_exact}. \\
\indent We test \cref{alg:alt_basis_APA_exact_truncated} on three APA routines:\ Schönhage's $\langle3,3,3;21\rangle$ algorithm (which has highest degree two in its base-case remainder), Smirnov's $\langle3,3,3;20\rangle$ algorithm (which has highest degree ten), and our $\langle2,2,3;10\rangle$ algorithm (which has highest degree six). Input matrices $A$ and $B$ are chosen to be $m_0^4\times k_0^4$ and $k_0^4\times n_0^4$ with entries sampled independently from ${\mathcal N}(0,1)$. \cref{fig:exp3-schonhege} shows accuracy results for each algorithm under varying choices of $r$ and $\beta$. In each case, we see the anticipated tradeoff. If $\beta$ is too small, even evaluating enough times to get an ``exact" product (i.e., running \cref{alg:alt_basis_APA_exact}) gives poor results. On the other hand, if $\beta$ is close to one -- and in particular if only a modestly-accurate product is desired -- we can get away with a constant factor reduction in the number of evaluations. \\
\indent We also observe in these plots the phenomenon discussed in \cref{rem: degree}. That is, while the highest degree in the base-case remainder of our $\langle 2,2,3; 10 \rangle$ algorithm is six (and recalling that we're testing with inputs that require four recursive steps) we do not need $25$ evaluations to get the exact product. This is a result of the fact that the remainder for this algorithm has the special form 
\begin{equation}
    AB + \varepsilon^3R_3 + \varepsilon^6 R_6. 
\end{equation}
Hence, when we apply the algorithm recursively with $L = 4$, only degrees 3, 6, 9, 12, 15, 18, 21, and 24 appear in the remainder. Since we are using scaled $r$-th roots of unity as evaluation points, we  obtain the exact product whenever $r$ does not divide any degree in the remainder; this happens here when $r=25$ but also occurs at much lower values of $r$ since only specific exponents are present. This corresponds to the bands of blue in the $\langle 2,2,3; 10 \rangle$ panel of \cref{fig:exp3-schonhege}. We see similar (though less pronounced) behavior for Smirnov's $\langle 3,3,3; 20 \rangle$ algorithm since its base-case remainder omits the exponent nine. 

\begin{figure}[htbp]
  \centering
  \includegraphics[width=.9\textwidth]{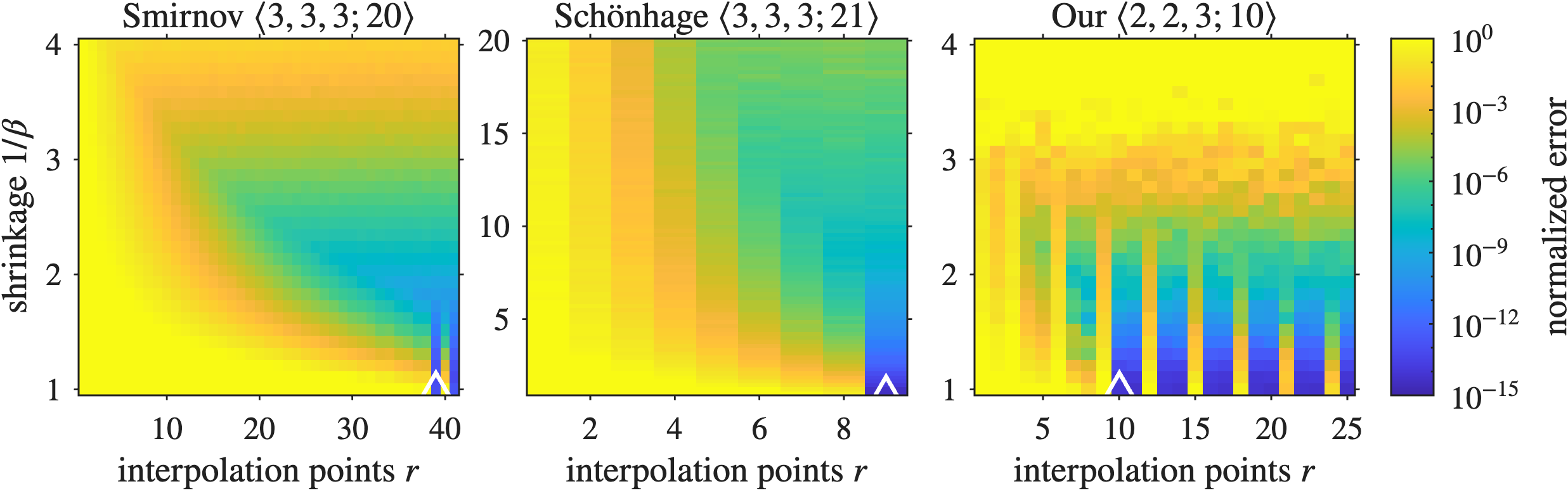}
  \caption{\cref{alg:alt_basis_APA_exact_truncated} runs with various APA routines. For each choice of $r$, we use the evaluation points $\varepsilon_j = \beta \cdot e^{2 \pi i j/r}$ for several choices of $0 < \beta < 1$. The colorbar corresponds to normalized forward error $\|\widehat{C}-C\|/(\|A\|\|B\|)$ in the maximum-entry norm. A triangle marks the smallest error achieved on each plot. }
  
  \label{fig:exp3-schonhege}
\end{figure}

\section{Conclusion}\label{subsection: conclusion}
This paper presented a deep dive into arbitrary precision approximate (APA) matrix multiplication. We provided both a new search algorithm for APA schemes and numerical analysis results aimed at understanding the practical performance of an APA algorithm. As mentioned in the introduction, the goal of this work was to renew interest in APA matrix multiplication, both from the theory and implementation perspectives. \\
\indent In future work, we plan to pursue high-performance implementations of APA matrix multiplication algorithms with the goal of demonstrating results competitive with fast exact matrix multiplication (possibly in applications where lower precision is required). We have already identified several tools for accelerating APA algorithms -- e.g., sparsification, truncation, etc.\ -- but these have not yet been tested in a high-performance environment. It also remains to examine other kinds of approximate matrix multiplication. In particular, it is still open to prove numerical stability for an approximate matrix multiplication algorithm obtained from a floating-point, least-squares solution to the tensor decomposition problem \eqref{eqn: MM_tensor}; several such algorithms exist in the literature \cite{Kaporin2024,Kaporin2024_2,Smirnov2013} and should, like APA matrix multiplication,  be considered for practical use. 

\section{Acknowledgements}\label{section: acknowledgments} 
This work was done in part while the authors were visiting the Simons Institute for the Theory of Computing. Support was provided by a UC Multicampus Research Program Award M25PL9060. The project also received funding from the European Research Council (ERC) under the European Union’s Horizon 2020 research and innovation program (grant agreements 818252 and 101113120) and was supported by grant 1354/23 of the Israel Science Foundation (founded by the Israel Academy of Sciences and Humanities). Additional support was provided by the Science Accelerator and by the Frontiers in Science initiative of the Ministry of Innovation, Science and Technology. Finally, R.S. is supported by NSF grant MSPRF 2402027.  Special thanks to James Demmel for suggesting the topic of APA matrix multiplication and to Dan Romik for helpful comments. \\

\noindent \textbf{AI Declaration:} Artificial intelligence (AI) tools assisted in the production of the code used to run our search algorithm in \cref{section: search} and the numerical examples in \cref{section: examples}. An LLM also provided feedback on an earlier draft of this paper, suggesting the special case for sparsification discussed in \cref{subsec: section_two_special_case} and highlighting an error in the proof of \cref{thm: main_APA_error_bound}, which was corrected by a human. The authors take full responsibility for the correctness of all results presented in this work.  

\bibliographystyle{abbrv}{}
\bibliography{bib}

\appendix


\section{Technical Results for \cref{section: stability}}
This appendix collects technical results used in the proof of our main stability bound (i.e., \cref{thm: main_APA_error_bound}).
\subsection{Vandermonde Bounds}
For a set of (distinct and nonzero) points $\varepsilon_1, \ldots, \varepsilon_p \in {\mathbb C}$, consider the $p \times p$ Vandermonde system 
\begin{equation}\label{eqn: vandermonde_system_app}
        \begin{pmatrix} 1 & 1 & \cdots & 1 \\ 
        \varepsilon_1 & \varepsilon_2 & \cdots & \varepsilon_{p} \\
       \vdots & \vdots & & \vdots \\
       \varepsilon_1^{p-1} & \varepsilon_2^{p-1} & \cdots & \varepsilon_p^{p-1}
        \end{pmatrix} \begin{pmatrix} \alpha_1 \\ \alpha_2 \\ \vdots \\ \alpha_p \end{pmatrix} = \begin{pmatrix} 1 \\ 0 \\ \vdots \\ 0 \end{pmatrix}.
    \end{equation}
Recall from \cref{section: stability} that the exact solution to this system is
\begin{equation}\label{eqn: exact_vandermonde_soln}
    \alpha_j = \prod_{k \neq j} \frac{\varepsilon_k}{\varepsilon_k-\varepsilon_j}, \; \; \; 1 \leq j \leq p.
\end{equation}
Below, we derive the following:
\begin{enumerate}
    \item A (relative) error bound for evaluating \eqref{eqn: exact_vandermonde_soln} in floating-point arithmetic (\cref{prop: evaluate_vandermonde_directly}).
    \item An (exact-arithmetic) perturbation bound for the $\alpha_j$ (\cref{prop: vandermonde_perturbation}).
\end{enumerate}

\indent We start with the floating-point error bound, which assumes the same model of floating-point arithmetic used in \cref{section: stability} (i.e, \cite{Overton,Higham_book}). Informally, \cref{prop: evaluate_vandermonde_directly} implies that, under some mild assumptions, the $\alpha_j$ can be computed accurately (in a relative sense) by simply evaluating \eqref{eqn: exact_vandermonde_soln} directly. 
\begin{prop}\label{prop: evaluate_vandermonde_directly}
    Given distinct, nonzero $\varepsilon_1, \ldots, \varepsilon_p \in {\mathbb C}$, let 
    \[ \alpha_j= \prod_{k \neq j} \frac{\varepsilon_k}{\varepsilon_k-\varepsilon_j} \; \; \; \text{for} \; \; \; 1 \leq j \leq p.  \]
    Suppose that the approximate quantities $\widehat{\alpha}_1, \ldots, \widehat{\alpha}_p$ are the result of evaluating this expression in floating-point arithmetic, where we assume the following: 
    \begin{enumerate}
        \item Each $\varepsilon_i$ is a normalized floating-point number. 
        \item The quotients $\frac{\varepsilon_k}{\varepsilon_k-\varepsilon_j}$ are evaluated directly and multiplied in an order that guarantees no over/underflow occurs. 
        \item $\frac{2(p-1)\textbf{u}}{1-\textbf{u}} < 1$, where \text{\normalfont \bf u} is the unit roundoff. 
    \end{enumerate}
    In this case, we have
    \[ \frac{|\widehat{\alpha}_j - \alpha_j|}{|\alpha_j|} \leq 3p \text{\normalfont \bf u} + O(\text{\normalfont \bf u}^2). \]
\end{prop}
\begin{proof}
    For fixed $j$, let $y_k = \frac{\varepsilon_k}{\varepsilon_k-\varepsilon_j}$. If $\widehat{y}_k$ is the value of this quantity in floating-point arithmetic, then 
\begin{equation}
  \hat{y}_k = \frac{\varepsilon_k(1+\delta')}{(\varepsilon_k - \varepsilon_j)(1+\delta)} \; \; \; \;  \text{for some } \; \; \; |\delta|,|\delta'| \leq \textbf{u}. 
\end{equation}
Equivalently, $\widehat{y}_k = y_k \left( 1 + \mu_k \right)$ where $|\mu_k| = \left| \frac{\delta' - \delta}{1+\delta} \right| \leq \frac{2\textbf{u}}{1-\textbf{u}}$. Now ${\widehat{\alpha}}_j$ is just a floating-point product of the $\widehat{y}_k$. Assuming without loss of generality that this product can be computed from left to right without encountering over/underflow (the argument is the same for any ordering), the corresponding floating-point partial products $\widehat{\Pi}_i$ satisfy
\begin{equation}
    \aligned 
    \widehat{\Pi}_1 &= \widehat{y}_1 = y_1(1+\mu_1) \\
    \widehat{\Pi}_2 &= \widehat{y}_1\widehat{y}_2(1+\delta_1) = y_1y_2(1+\mu_1)(1+\mu_2)(1+\delta_1) \\
    \widehat{\Pi}_3 & = \widehat{\Pi}_2 \widehat{y}_3(1+\delta_2) = y_1y_2y_3(1+\mu_1)(1+\mu_2)(1+\mu_3)(1+\delta_1)(1+\delta_2) \\
    & \; \;  \vdots \\
    \endaligned
\end{equation}
and therefore
\begin{equation}
    \widehat{\alpha}_j = \widehat{\Pi}_{p-1} =  \alpha_j \cdot \prod_{k \neq j} (1+\mu_k) \cdot \prod_{k = 1}^{p-2}(1+\delta_k) ,
\end{equation}
where each $\delta_k$ satisfies $|\delta_k| \leq \textbf{u}$. Since $\frac{2(p-1)\textbf{u}}{1-\textbf{u}} < 1$, standard bounds -- e.g., \cite[Lemma 3.1]{Higham_book} -- imply
$\prod_{k \neq j}(1+\mu_k) = 1 + \theta_1$ and $\prod_{k=1}^{p-2}(1+\delta_i) = 1 + \theta_2$ for 
\begin{equation}\label{eqn: theta_bounds}
    |\theta_1| \leq 2p\textbf{u} + O(\textbf{u}^2) \; \; \; \text{and} \; \; \; |\theta_2| \leq p\textbf{u} + O(\textbf{u}^2).
\end{equation}
Hence,
\begin{equation}
    \frac{|\widehat{\alpha}_j-\alpha_j|}{|\alpha_j|} = |\theta_1 + \theta_2 + \theta_1\theta_2| \leq |\theta_1| + |\theta_2| + |\theta_1||\theta_2|.
\end{equation}
The result now follows by applying \eqref{eqn: theta_bounds}.
\end{proof}

We next prove a perturbation bound for \eqref{eqn: vandermonde_system_app}. 

\begin{prop}\label{prop: vandermonde_perturbation}
   Let $\varepsilon_1, \ldots, \varepsilon_p \in {\mathbb C}$ be distinct and nonzero and define
   \[ \text{\normalfont gap}_{\varepsilon} = \min_{i \neq j} |\varepsilon_i-\varepsilon_j| \; \; \;  \text{and} \; \; \; \;   M_{\varepsilon} = \max_{j} |\varepsilon_j|.\]
   Suppose that $\tilde{\varepsilon}_1$, \ldots, $\tilde{\varepsilon}_p$ satisfy \[\tilde{\varepsilon}_j = \varepsilon_j(1+\delta_j) \; \; \text{for} \; \; |\delta_j| \leq \delta, \] 
   where $\delta > 0$ is the ``size" of the relative perturbation. Let $(\alpha_1, \ldots, \alpha_p)$ and $(\tilde{\alpha}_1, \ldots, \tilde{\alpha}_p)$ be the exact-arithmetic solutions to the Vandermonde system \eqref{eqn: vandermonde_system_app} at the points $\varepsilon_1, \ldots, \varepsilon_p$ and $\tilde{\varepsilon}_1, \ldots, \tilde{\varepsilon}_p$, respectively.
   If $2(p-1)M_{\varepsilon}\delta < \text{\normalfont gap}_{\varepsilon}-2M_{\varepsilon} \delta$, then
    \[ \frac{|\tilde{\alpha}_j- \alpha_j|}{|\alpha_j|} \leq \frac{2(p-1)M_{\varepsilon} \delta}{\text{\normalfont gap}_{\varepsilon} -2pM_{\varepsilon}\delta}.  \]
\end{prop}
\begin{proof}
    Let $y_k = \frac{\varepsilon_k}{\varepsilon_k-\varepsilon_j}$ and $\tilde{y}_k = \frac{\tilde{\varepsilon}_k}{\tilde{\varepsilon}_k - \tilde{\varepsilon}_j}$. We have
    \begin{equation}
        \tilde{y}_k = \frac{\varepsilon_k (1+\delta_k)}{\varepsilon_k(1+\delta_k)-\varepsilon_j(1+\delta_j)} = \frac{\varepsilon_k}{\varepsilon_k-\varepsilon_j} \cdot \frac{1+\delta_k}{1 + \frac{\varepsilon_k \delta_k - \varepsilon_j \delta_j}{\varepsilon_k - \varepsilon_j}} = y_k \left( 1 + \frac{\varepsilon_j(\delta_j-\delta_k)}{\varepsilon_k-\varepsilon_j + \varepsilon_k\delta_k - \varepsilon_j\delta_j}\right),
    \end{equation}
    where $\mu_k = \frac{\varepsilon_j(\delta_j-\delta_k)}{\varepsilon_k-\varepsilon_j+\varepsilon_k\delta_k - \varepsilon_j\delta_j}$ satisfies $|\mu_k| \leq \frac{2 M_{\varepsilon}\delta}{\text{gap}_{\varepsilon}-2M_{\varepsilon}\delta}$. Hence,
    \begin{equation}
        \tilde{\alpha}_j = \prod_{k \neq j}\tilde{y}_k = \prod_{k \neq j}y_k(1+\mu_k) = \alpha_j \prod_{k \neq j}(1+\mu_k).
    \end{equation}
    Assuming $\frac{2(p-1)M_{\varepsilon}\delta}{\text{gap}_{\varepsilon}-2M_{\varepsilon}\delta} < 1$ and again applying \cite[Lemma 3.1]{Higham_book}, we can rewrite $\prod_{k \neq j}(1+\mu_k)$ as follows:
    \begin{equation}
        \prod_{k \neq j} (1 + \mu_k) = 1 + \theta \; \; \;  \text{for} \; \; \; |\theta| \leq \frac{2(p-1)M_{\varepsilon} \delta}{\text{gap}_{\varepsilon} - 2pM_{\varepsilon}\delta}.
    \end{equation}
    Inserting this above and rearranging yields the bound. 
\end{proof}

\subsection{Perturbation Bound for Bilinear Algorithms}
Finally, we prove a perturbation bound for evaluating a recursive, alternative-basis bilinear algorithm in exact arithmetic. 
\begin{prop}\label{prop: evaluating an approximate bilinear algorithm}
    Let $U_{\phi} \in {\mathbb C}^{m_0 k_0 \times t}$, $V_{\psi} \in {\mathbb C}^{k_0n_0 \times t}$, $W_{\nu} \in {\mathbb C}^{m_0n_0 \times t}$, $\phi \in {\mathbb C}^{m_0k_0 \times m_0k_0}$, $\psi \in {\mathbb C}^{k_0n_0 \times k_0n_0}$, and $\nu^{-1} \in {\mathbb C}^{m_0n_0 \times m_0n_0}$ denote the encoding matrices and basis transformations defining an alternative basis bilinear map
    \[ \langle \cdot , \cdot \rangle: {\mathbb C}^{m_0 \times k_0} \times {\mathbb C}^{k_0 \times n_0} \rightarrow {\mathbb C}^{m_0 \times n_0} \]
    evaluated as
    \[ c = \nu^{-1} \left[ W_{\nu}\left(U_{\phi}^T(\phi^Ta) \odot V_{\psi}^T(\psi^Tb) \right) \right], \]
    where $a$ and $b$ are vectorized copies of input matrices $A \in {\mathbb C}^{m_0 \times k_0}$ and $B \in {\mathbb C}^{k_0 \times n_0}$ and $c$ is a vectorized copy of the corresponding output $ \langle A, B \rangle = C \in {\mathbb C}^{m_0 \times n_0}$. 
    Suppose that a corresponding ``perturbed" bilinear form is defined by encoding matrices  $\widetilde{U}_{\phi} \in {\mathbb C}^{m_0k_0 \times t}$, $\widetilde{V}_{\psi} \in {\mathbb C}^{k_0n_0 \times t}$, and $\widetilde{W}_{\nu} \in {\mathbb C}^{m_0n_0 \times t}$ and basis transformations $\tilde{\phi} \in {\mathbb C}^{m_0k_0 \times m_0k_0}$, $\tilde{\psi} \in {\mathbb C}^{k_0n_0 \times k_0n_0}$, and $\tilde{\nu}^{-1} \in {\mathbb C}^{m_0n_0 \times m_0n_0}$, which are obtained by applying entry-wise, relative perturbations to $U_{\phi}$, $V_{\psi}$, $W_{\nu}$, $\phi$, $\psi$, and $\nu^{-1}$, respectively, of size at most $\delta > 0$ \text{\normalfont (}e.g., $\widetilde{U}_{\phi}(i,j) = U_{\phi}(i,j)(1+\Delta_{ij})$ for some $|\Delta_{ij}| \leq \delta$\text{\normalfont )}. Let $C, \widetilde{C} \in {\mathbb C}^{m \times n}$ be the result of evaluating recursive versions of these bilinear algorithms at input matrices $A \in {\mathbb C}^{m \times k}$ and $B \in {\mathbb C}^{k \times n}$ in exact arithmetic, where we assume that both default to matrix multiplication after $L$ recursive steps for $L \in {\mathbb N}$ satisfying 
    \[ L \leq \min \left\{  \log_{m_0}(m), \log_{n_0}(n), \log_{k_0}(k) \right\}.\]
    In this case,
    \[ \|C - \widetilde{C}\| \leq \left( (1+\delta)^{6L} - 1 \right)(k/k_0^L)E^L \|A\|\|B\|\]
    where $\| \cdot \|$ is the maximum-entry norm and $E$ is the stability factor corresponding to $U_{\phi}$, $V_{\psi}$, $W_{\nu}$, $\phi$, $\psi$, and $\nu^{-1}$ {\normalfont (}see \cref{defn:stability_factor} {\normalfont )}.
\end{prop}
\begin{proof}
    We prove the result via induction on $L$, starting with the inductive step. Assume that the given bound holds for $L-1$ and consider $L \geq 2$. To evaluate our bilinear forms, we partition the input matrices into blocks as follows (assuming $m_0$, $k_0$, and $n_0$ divide $m$, $k$, and $n$):
    \begin{equation}
        A = \begin{pmatrix} A_{11} & \cdots & A_{1k_0} \\
        \vdots & & \vdots \\
        A_{m_01} & \cdots & A_{m_0k_0} \end{pmatrix} \quad \text{and} \quad B = \begin{pmatrix} 
        B_{11} & \cdots & B_{1n_0} \\
        \vdots & & \vdots \\
        B_{k_01} & \cdots & B_{k_0n_0}
        \end{pmatrix}.
    \end{equation}
    Each $A_{ij}$ is $(m/m_0) \times (k/k_0)$ while each $B_{ij}$ is $(k/k_0)$ by $(n/n_0)$. For this block partition, let ${\mathfrak a}$ and ${\mathfrak b}$ be \textit{block} vectorized copies of $A$ and $B$ -- e.g., 
    \begin{equation}
        {\mathfrak a} = \begin{pmatrix} A_{11} \\  A_{12} \\ \vdots \\  A_{m_0k_0} \end{pmatrix}  \in {\mathbb C}^{mk_0 \times (k/k_0)}\quad \text{and} \quad {\mathfrak b} =  \begin{pmatrix} B_{11} \\ B_{12}\\  \vdots \\ B_{k_0n_0} \end{pmatrix} \in {\mathbb C}^{kn_0 \times (n/n_0)}.
    \end{equation}
    Additionally, let $\widehat{C}$ be the output of a hybrid, recursive bilinear algorithm applied to $A$ and $B$, which uses $\widetilde{U}_{\phi}$, $\widetilde{V}_{\psi}$,  $\widetilde{W}_{\nu}$, $\tilde{\phi}$, $\tilde{\psi}$, and $\tilde{\nu}^{-1}$ at the highest level and $U_{\phi}$, $V_{\psi}$, $W_{\nu}$, $\phi$, $\psi$, $\nu^{-1}$ thereafter. In this case, we have 
    \begin{equation}\label{eqn: c_hat_triangle}
        \|C - \widetilde{C}\| \leq \| C - \widehat{C}\| + \|\widehat{C} - \widetilde{C} \|.
    \end{equation}
    To bound the second term in \eqref{eqn: c_hat_triangle}, let $\hat{{\mathfrak c}}$ be a block-vectorized version of $\widehat{C}$, which can be written as
    \begin{equation}\label{eqn: block_c_hat}
        \hat{\mathfrak{c}} = \left( \tilde{\nu}^{-1} \otimes I_{m/m_0} \right)\left( \widetilde{W}_{\nu} \otimes I_{m/m_0} \right) \begin{pmatrix} 
        \widetilde{M}_1 \\
        \vdots \\
        \widetilde{M}_t \end{pmatrix} 
    \end{equation}
    for $\widetilde{M}_j = \langle (\widetilde{U}_{\phi}(:,j)^T \otimes I_{m/m_0})( \tilde{\phi}^T \otimes I_{m/m_0}) \mathfrak{a},( \widetilde{V}_{\psi}(:,j)^T \otimes I_{k/k_0}) ( \tilde{\psi}^T \otimes I_{k/k_0} ) \mathfrak{b} \rangle_{L-1} $, where  $\langle \cdot , \cdot \rangle_{L-1}$ is the bilinear map obtained by applying $\langle \cdot, \cdot \rangle$ recursively $L-1$ times (and again defaulting to matrix multiplication at that point). To obtain \eqref{eqn: block_c_hat}, we exploit the fact that, in exact arithmetic, this expression is equivalent to evaluating the perturbed alternative basis algorithm as done in \cref{section: stability} -- e.g., where the basis transformations are applied recursively to $A$ and $B$ first before evaluating the bilinear algorithm corresponding to $\widetilde{U}_{\phi}$, $\widetilde{V}_{\psi}$, and $\widetilde{W}_{\nu}$ and finally applying $\tilde{\nu}^{-1}$. This is implied by the following observation. Suppose that 
   \begin{equation}
       \left( \tilde{\phi}^T \otimes I_{m/m_0} \right) \mathfrak{a} = \mathfrak{a}' = \begin{pmatrix} A_{11}' \\ A_{12}' \\
       \vdots \\
       A_{m_0k_0}'
       \end{pmatrix}
   \end{equation}
   for some $(m/m_0) \times (k/k_0)$ matrices $A_{11}', \ldots, A_{m_0k_0}'$. In this case, the left input to $\langle \cdot , \cdot \rangle_{L-1}$ in block row $j$ of \eqref{eqn: block_c_hat} is 
   \begin{equation}
       A_j \coloneqq \left( \widetilde{U}_{\phi}(:,j)^T \otimes I_{m/m_0} \right) \mathfrak{a'} = \widetilde{U}_{\phi}(1,j) A_{11}' + \cdots + \widetilde{U}_{\phi}(m_0k_0,j) A_{m_0k_0}' \in {\mathbb C}^{(m/m_0) \times (k/k_0)}.
   \end{equation}
  The subsequent recursive call begins by multiplying $\mathfrak{a}_j$, the block-vectorized copy of $A_j$ obtained by breaking it into submatrices of size $(m/m_0^2) \times (k/k_0^2)$, by $(\tilde{\phi}^T \otimes I_{m/m_0^2})$. If $\mathfrak{a}_{11}', \ldots, \mathfrak{a}_{m_0k_0}'$ are similarly block-vectorized copies of $A_{11}', \ldots, A_{m_0k_0}'$, then 
   \begin{equation}
       \mathfrak{a}_j = \widetilde{U}_{\phi}(1,j) \mathfrak{a}_{11}' + \cdots + \widetilde{U}_{\phi}(m_0k_0,j) \mathfrak{a}_{m_0k_0}'
   \end{equation}
   and therefore
   \begin{equation}
        \aligned 
       \left( \tilde{\phi}^T \otimes I_{m/m_0^2} \right)  \mathfrak{a}_j &= \widetilde{U}_{\phi}(1,j) \left( \tilde{\phi}^T \otimes I_{m/m_0^2} \right) \mathfrak{a}_{11}' + \cdots + \widetilde{U}_{\phi}(m_0k_0,j) \left( \tilde{\phi}^T \otimes I_{m/m_0^2} \right) \mathfrak{a}_{m_0k_0}' \\
       & = \left( \widetilde{U}_{\phi}(:,j)^T \otimes I_{m/m_0} \right) \left( I_{m_0k_0} \otimes \left( \tilde{\phi}^T \otimes I_{m/m_0^2} \right) \right) \begin{pmatrix} \mathfrak{a}_{11}' \\
       \vdots \\
       \mathfrak{a}_{m_0k_0}'
       \end{pmatrix},
       \endaligned
   \end{equation}
   where $\left( I_{m_0k_0} \otimes \left( \tilde{\phi}^T \otimes I_{m/m_0^2} \right)\right) \begin{pmatrix} \mathfrak{a}_{11}' \\
       \vdots \\
       \mathfrak{a}_{m_0k_0}'
       \end{pmatrix}$ is exactly what we get if we break $A$ into $(m/m_0^2) \times (k/k_0^2)$ blocks and multiply the corresponding block vectorization by $((\tilde{\phi}^{\otimes 2})^T \otimes I_{m/m_0^2})$. Continuing recursively -- and applying the same argument to the other encoding matrices/basis transformations -- implies the claimed equivalence. \\
   \indent We now continue bounding $\|\widehat{C} - \widetilde{C}\|$. By our induction hypothesis, the block-vectorized copy of $\widetilde{C}$ satisfies 
    \begin{equation}
        \tilde{\mathfrak{c}} = \left( \tilde{\nu}^{-1} \otimes I_{m/m_0} \right)\left(\widetilde{W}_{\nu} \otimes I_{m/m_0} \right) \begin{pmatrix}
            \widetilde{M}_1 + E_1 \\ \vdots \\ \widetilde{M}_t + E_t
        \end{pmatrix},
    \end{equation}
    for $E_1, \ldots, E_t$ with 
    \begin{equation}
        \resizebox{.9\textwidth}{!}{
        $\displaystyle
        \|E_j\| \leq \left((1+\delta)^{6(L-1)} - 1 \right)(k/k_0^L)E^{L-1} \left\| \left(\widetilde{U}_{\phi}(:,j)^T \otimes I_{m/m_0}\right) \left( \tilde{\phi}^T \otimes I_{m/m_0}  \right) \mathfrak{a} \right\| \left\|\left( \widetilde{V}_{\psi}(:,j)^T \otimes I_{k/k_0} \right) \left( \tilde{\psi}^T \otimes I_{k/k_0} \right) \mathfrak{b} \right\|.
        $}
        \end{equation}
    Accordingly, we have
    \begin{equation}\label{eqn: appendix_term_one}
        \resizebox{.9\textwidth}{!}{$\displaystyle
        \aligned 
         \| &\widehat{C}  - \widetilde{C}\| = \| \hat{\mathfrak{c}} - \tilde{\mathfrak{c}}\| = \left\| \left(\tilde{\nu}^{-1} \otimes I_{m/m_0} \right)\left( \widetilde{W}_{\nu} \otimes I_{m/m_0} \right) \begin{pmatrix} E_1 \\ \vdots \\ E_t \end{pmatrix} \right\|  \leq \left\| |\tilde{\nu}^{-1}| \cdot |\widetilde{W}_{\nu}|  \begin{pmatrix} \|E_1\| \\
        \vdots \\
        \| E_t \| \end{pmatrix} \right\|_{\infty} \\
        & \leq \left( (1+\delta)^{6(L-1)} - 1\right) (k/k_0^L) E^{L-1} \left\| | \tilde{\nu}^{-1}| \cdot |\widetilde{W}_{\nu}| \begin{pmatrix} \left(\sum_{i=1}^{m_0k_0} |\widetilde{U}_{\phi}(i,1) | \|\tilde{\phi}(:,i)\|_1 \right) \left( \sum_{i=1}^{k_0n_0} | \widetilde{V}_{\psi}(i,1) | \| \tilde{\psi}(:,i) \|_1 \right) \\
        \vdots \\
        \left(\sum_{i=1}^{m_0k_0} |\widetilde{U}_{\phi}(i,t) | \|\tilde{\phi}(:,i) \|_1 \right)\left( \sum_{i=1}^{k_0n_0} | \widetilde{V}_{\psi}(i,t) | \| \tilde{\psi}(:,i)\|_1\right)
        \end{pmatrix} \right\|_{\infty} \|A\| \|B\| \\
        & \leq \left( (1+\delta)^{6(L-1)} - 1\right) (k/k_0^L) E^{L-1} \left\| \left( | \tilde{\nu}^{-1} |\cdot |\widetilde{W}_{\nu}|\right) \left( (|\tilde{\phi}| \cdot |
        \widetilde{U}_{\phi}|)^T\mathds{1} \odot (|\tilde{\psi}| \cdot |\widetilde{V}_{\psi}|)^T\mathds{1} \right) \right\|_{\infty} \|A\| \|B\| \\
        & \leq \left( (1+\delta)^{6L} - (1+\delta)^6\right) (k/k_0^L) E^L \|A\| \|B\|
        \endaligned 
        $}
    \end{equation}
    for $\mathds{1}$ the vector of all ones. Here, we note that $ \| ( | \tilde{\nu}^{-1} |\cdot |\widetilde{W}_{\nu}|) ( (|\tilde{\phi}| \cdot |
        \widetilde{U}_{\phi}|)^T\mathds{1} \odot (|\tilde{\psi}| \cdot |\widetilde{V}_{\psi}|)^T\mathds{1} ) \|_{\infty}$ is the stability factor corresponding to $\widetilde{U}_{\phi}$, $\widetilde{V}_{\psi}$, $\widetilde{W}_{\nu}$, $\tilde{\phi}$, $\tilde{\psi}$ and $\tilde{\nu}^{-1}$, which is at most $(1+\delta)^6 E$. \\
        \indent To similarly bound the first term of \eqref{eqn: c_hat_triangle}, we write
    \begin{equation}
      \mathfrak{c} = \left(\nu^{-1} \otimes I_{m/m_0} \right) \left( W_{\nu} \otimes I_{m/m_0} \right) \begin{pmatrix}
          M_1 \\ \vdots \\ M_t 
      \end{pmatrix} 
    \end{equation}
    for $M_j = \left\langle \left(U_{\phi}(:,j)^T \otimes I_{m/m_0} \right)  \left( \phi^T \otimes I_{m/m_0} \right) \mathfrak{a}, \left( V_{\psi}(:,j)^T \otimes I_{k/k_0} \right) \left( \psi^T \otimes I_{k/k_0} \right) \mathfrak{b} \right\rangle_{L-1}$. 
    Since $\langle \cdot, \cdot \rangle_{L-1}$ is bilinear, we can express $M_j - \widetilde{M}_j$ as
    \begin{equation}\label{eqn: ugly_indexing}
        \resizebox{.91\textwidth}{!}{$\displaystyle
        \sum_{\substack{1 \leq i_1 \leq m_0k_0 \\ 1 \leq i_2 \leq k_0n_0}} \left[ \left( \sum_{i_1' = 1}^{m_0k_0} U_{\phi} (i_1',j) \phi(i_1,i_1') \right)
        \left( \sum_{i_2' = 1}^{k_0n_0} V_\psi(i_2',j) \psi(i_2,i_2') \right) - \left( \sum_{i_1' = 1}^{m_0k_0} \widetilde{U}_{\phi} (i_1',j) \tilde{\phi}(i_1,i_1') \right)
        \left( \sum_{i_2' = 1}^{k_0n_0} \widetilde{V}_\psi(i_2',j) \tilde{\psi}(i_2,i_2') \right) \right]\langle \mathfrak{a}(i_1), \mathfrak{b}(i_2) \rangle_{L-1},
        $}
    \end{equation}
    for any $1 \leq j \leq t$, where $\mathfrak{a}(i_1)$ and $\mathfrak{b}(i_2)$
    are $(m/m_0) \times (k/k_0)$ and $(k/k_0) \times (n/n_0)$ subblocks of $\mathfrak{a}$ and $\mathfrak{b}$ with block row indices $i_1$ and $i_2$, respectively. Now by \eqref{eqn: output_norm_bound} 
    \begin{equation}\label{eqn: sub-call bound}
       \| \langle \mathfrak{a}(i_1), \mathfrak{b}(i_2) \rangle_{L-1} \| \leq (k/k_0^L)E^{L-1} \|A\|\|B\|, 
    \end{equation}
    so we have (combining the sums in \eqref{eqn: ugly_indexing})
    \begin{equation}\label{eqn:correct ref}
        \resizebox{.9\textwidth}{!}{$\displaystyle
        \begin{split}
        \| M_j - \widetilde{M}_j\| & \leq   (k/k_0^L)E^{L-1} \|A\| \|B\| \sum_{\substack{1 \leq i_1,i_1' \leq m_0k_0 \\ 1 \leq i_2,i_2' \leq k_0n_0}} \left| U_{\phi}(i_1',j)\phi(i_1,i_1')V_{\psi}(i_2',j)\psi(i_2,i_2') - \widetilde{U}_{\phi}(i_1',j)  \tilde{\phi}(i_1,i_1')\widetilde{V}_{\psi}(i_2',j) \tilde{\psi}(i_2,i_2') \right| \\
        & \leq ((1+\delta)^4 - 1) (k/k_0^L)E^{L-1} \|A\| \|B\| \sum_{\substack{1 \leq i_1,i_1' \leq m_0k_0 \\ 1 \leq i_2,i_2' \leq k_0n_0}} \left| U_{\phi}(i_1',j)\phi(i_1,i_1')V_{\psi}(i_2',j)\psi(i_2,i_2') \right| \\
        &  = ((1+\delta)^4 - 1) (k/k_0^L) E^{L-1} \|A\| \|B\| \sum_{\substack{1 \leq i_1' \leq m_0k_0 \\ 1 \leq i_2' \leq k_0n_0}} |U_{\phi}(i_1',j) | \|\phi(:,i_1')\|_1 |V_{\psi}(i_2',j)| \|\psi(:,i_2')\|_1  \\
        &  = ((1+\delta)^4 - 1) (k/k_0^L) E^{L-1} \|A\| \|B\| \left( ( |\phi| \cdot |U_{\phi}|)^T \mathds{1} \odot (|\psi| \cdot |V_{\psi}|)^T \mathds {1} \right)_{j}
        \end{split}
        $}
    \end{equation}
    for $( ( |\phi| \cdot |U_{\phi}|)^T \mathds{1} \odot (|\psi| \cdot |V_{\psi}|)^T \mathds {1})_{j}$ the $j$-th entry of $(( |\phi| \cdot |U_{\phi}|)^T \mathds{1} \odot (|\psi| \cdot |V_{\psi}|)^T \mathds {1})$. 
    In this case, we have
    \begin{equation}
        \aligned 
        \| \mathfrak{c} - \hat{\mathfrak{c}} \| &= \left\|\left( \nu^{-1} \otimes I_{m/m_0} \right) \left(W_{\nu} \otimes I_{m/m_0}\right)\begin{pmatrix} M_1 - \widetilde{M}_1 \\ 
        \vdots \\
        M_t - \widetilde{M}_t \end{pmatrix} + \Delta W \begin{pmatrix} \widetilde{M}_1 \\ \vdots \\ \widetilde{M}_t \end{pmatrix}\right\| \\
        &\leq \left\| \left( \nu^{-1} \otimes I_{m/m_0} \right) \left( W_{\nu} \otimes I_{m/m_0} \right)\begin{pmatrix} M_1 - \widetilde{M}_1 \\ 
        \vdots \\
        M_t  - \widetilde{M}_t \end{pmatrix}  \right\| + \left\| \Delta W \begin{pmatrix} \widetilde{M}_1 \\ \vdots \\ \widetilde{M}_t \end{pmatrix}\right\|,
        \endaligned
    \end{equation}
    where $\Delta W = (\nu^{-1} \otimes I_{m/m_0})(W_{\nu} \otimes I_{m/m_0}) - (\tilde{\nu}^{-1} \otimes I_{m/m_0})(\widetilde{W}_{\nu} \otimes I_{m/m_0})$.  Applying \eqref{eqn:correct ref} to the first term here yields
    \begin{equation}\label{eqn: part_two_part_one}
        \aligned
        \left\| \left( \nu^{-1} \otimes I_{m/m_0} \right) \left( W_{\nu} \otimes I_{m/m_0}\right) \begin{pmatrix} M_1 - \widetilde{M}_1 \\ 
        \vdots \\
        M_t  - \widetilde{M}_t \end{pmatrix}  \right\| & \leq \left\|\left(  |\nu^{-1}| \cdot |W_{\nu}| \right)\begin{pmatrix} \left\| M_1 - \widetilde{M}_1 \right\| \\ \vdots \\
        \left\| M_t - \widetilde{M}_t \right\| \end{pmatrix} \right\|_{\infty}  \\
        & \hspace{-5cm}\leq \left( (1+\delta)^4 - 1 \right) (k/k_0^L) \left\| \left( | \nu^{-1} |\cdot |W_{\nu}|\right) \left( (|\phi| \cdot |
        U_{\phi}|)^T\mathds{1} \odot (|\psi| \cdot |V_{\psi}|)^T\mathds{1} \right) \right\|_{\infty}  E^{L-1} \|A\|\|B\| \\
        & \hspace{-5cm} = \left((1+\delta)^4-1 \right) (k/k_0^L) E^L \|A\|\|B\|.
        \endaligned 
    \end{equation}
    Similarly,
    \begin{equation}\label{eqn: part_two_part_two}
        \resizebox{.9\textwidth}{!}{$\displaystyle
        \aligned 
        \left\|  \Delta W  \begin{pmatrix} \widetilde{M}_1 \\
        \vdots \\
        \widetilde{M}_t \end{pmatrix} \right\| &\leq \left\||\Delta W| \begin{pmatrix} \|\widetilde{M}_1\| \\
        \vdots \\
        \| \widetilde{M}_t \| \end{pmatrix}
        \right\|_{\infty} \\
        & \hspace{-2cm} \leq (2\delta + \delta^2) \left\|\left(|\nu^{-1} | \cdot |W_{\nu}| \right) \begin{pmatrix} \|M_1\| + ((1+\delta)^4 - 1) (k/k_0^L) E^{L-1} \|A\| \|B\| \left( ( |\phi| \cdot |U_{\phi}|)^T \mathds{1} \odot (|\psi| \cdot |V_{\psi}|)^T \mathds {1} \right)_{1}\\
        \vdots \\
        \|M_t\| + ((1+\delta)^4 - 1) (k/k_0^L) E^{L-1} \|A\| \|B\| \left( ( |\phi| \cdot |U_{\phi}|)^T \mathds{1} \odot (|\psi| \cdot |V_{\psi}|)^T \mathds {1} \right)_{t}\end{pmatrix} \right\|_{\infty} \\
        & \hspace{-2cm}\leq (2\delta+\delta^2) \left\| \left( |\nu^{-1}| \cdot |W_{\nu}| \right) \begin{pmatrix} \|M_1\| \\ \vdots \\ \|M_t \| \end{pmatrix} \right\|_{\infty} + (2\delta + \delta^2)\left((1+\delta)^4 - 1\right) (k/k_0^L) \left\| \left( | \nu^{-1} |\cdot |W_{\nu}|\right) \left( (|\phi| \cdot |
        U_{\phi}|)^T\mathds{1} \odot (|\psi| \cdot |V_{\psi}|)^T\mathds{1} \right) \right\|_{\infty}E^{L-1} \|A\| \|B\|\\
        & \hspace{-2cm} \leq (2\delta + \delta^2)   (k/k_0^L)\left\|\left( | \nu^{-1} |\cdot |W_{\nu}|\right) \left( (|\phi| \cdot |
        U_{\phi}|)^T\mathds{1} \odot (|\psi| \cdot |V_{\psi}|)^T\mathds{1} \right) \right\|_{\infty} E^{L-1} \|A\| \|B\| + (2\delta + \delta^2)\left( (1+\delta)^4 - 1\right) (k/k_0^L) E^{L} \|A\|\|B\| \\
        & \hspace{-2cm} = \left( (1+\delta)^6 - (1+\delta)^4 \right) (k/k_0^L) E^L \|A\| \|B\|
        \endaligned
        $}
    \end{equation}
    where we bound $\|M_j\|$ by repeating the argument made above. Summing \eqref{eqn: part_two_part_one} and \eqref{eqn: part_two_part_two} yields  
    \begin{equation}
        \| C-\widehat{C}\| \leq \left( (1+\delta)^6 - 1 \right) (k/k_0^L) E^L \|A\|\|B\|.
    \end{equation}
    Combining this result with \eqref{eqn: appendix_term_one} implies that the stated bound holds for $L$ if it holds for $L-1$. To complete the proof, we note that the base case (when $L = 1$) follows from the argument used to bound the first term $\|C - \widehat{C} \|$, substituting \eqref{eqn: sub-call bound} for the simpler bound $(k/k_0)\|A\| \|B\|$ (since in this case the recursive call is replaced by matrix  multiplication).
\end{proof}

\section{Encoding Matrices for $\langle 2,4,5; 31 \rangle$ }
\setcounter{MaxMatrixCols}{31}
This appendix contains encoding/residual matrices for a new APA $\langle 2,4,5; 31 \rangle$ scheme, which establishes 31 as an upper bound on the border rank of ${\mathcal M}_{\langle 2,4,5 \rangle}$. The corresponding algorithm approximates the product of $A \in {\mathbb C}^{2 \times 4}$ and $B \in {\mathbb C}^{4 \times 5}$ -- i.e., 
\begin{equation}
    \begin{pmatrix} a_{11} & a_{12} & a_{13} & a_{14} \\
    a_{21} & a_{22} & a_{23} & a_{24} \end{pmatrix} \times \begin{pmatrix}
        b_{11} & b_{12} & b_{13} & b_{14} & b_{15} \\
         b_{21} & b_{22} & b_{23} & b_{24} & b_{25} \\
          b_{31} & b_{32} & b_{33} & b_{34} & b_{35} \\
           b_{41} & b_{42} & b_{43} & b_{44} & b_{45} \\
    \end{pmatrix} = \begin{pmatrix} c_{11} & c_{12} & c_{13} & c_{14} & c_{15} \\
    c_{21} & c_{22} & c_{23} & c_{24} & c_{25} \end{pmatrix},
\end{equation}
and outputs
\begin{equation}
    AB + \varepsilon^3R_3 + \varepsilon^4R_4 + \varepsilon^5R_5 + \varepsilon^6R_6 + \varepsilon^7R_7 + \varepsilon^8R_8
\end{equation}
for a set of residual matrices $R_3, R_4, R_5,R_6, R_7, R_8$ that depend on $A$ and $B$. Note that the search algorithm discussed in \cref{section: search} actually found several $\langle 2,4,5; 31 \rangle$ schemes; the one shown here achieves the largest improvement to its leading coefficient via sparsification (see row 15 of \cref{tab:results_summary2}).

\subsection{Standard Basis Encoding Matrices }\label{ap: encoding_matrices}
We first provide the encoding matrices for this scheme in the standard basis: 
\vspace{2mm}
\begin{equation*}
\resizebox{\textwidth}{!}{
$\displaystyle
    U(\varepsilon)^T = \begin{pmatrix}
0 &0 &\varepsilon^2 &0 &\varepsilon^{-2} &\varepsilon^{-2} &0 &-\varepsilon \\
\varepsilon^{-1} &0 &-\varepsilon^2 &0 &-\varepsilon^{-2} &0 &\varepsilon &0 \\
0 &0 &0 &0 &\varepsilon^{-2} &0 &-\varepsilon &0 \\
0 &\varepsilon^{-1} &0 &0 &\varepsilon^{-2} &0 &-\varepsilon &0 \\
0 &0 &0 &0 &\varepsilon^{-2} &\varepsilon^{-2} &0 &-\varepsilon \\
0 &0 &-\varepsilon^2 &\varepsilon^2 &-\varepsilon^{-2} &0 &\varepsilon &-\varepsilon \\
0 &0 &-\varepsilon^2 &\varepsilon^2 &0 &0 &\varepsilon &-\varepsilon \\
0 &0 &\varepsilon^2 &0 &0 &0 &-\varepsilon &0 \\
0 &0 &-\varepsilon^2 &\varepsilon^2 &-\varepsilon^{-2} &-\varepsilon^{-2} &0 &0 \\
0 &0 &\varepsilon^2 &0 &-\varepsilon^{-2} &-\varepsilon^{-2} &0 &0 \\
0 &0 &-\varepsilon^2 &\varepsilon^2 &0 &0 &0 &0 \\
\varepsilon^{-1} &\varepsilon^{-1} &0 &-\varepsilon^2 &-\varepsilon^{-2} &-\varepsilon^{-2} &0 &\varepsilon \\
0 &\varepsilon^{-1} &0 &0 &0 &0 &0 &0 \\
0 &\varepsilon^{-1} &0 &0 &-\varepsilon^{-2} &-\varepsilon^{-2} &0 &\varepsilon \\
\varepsilon^{-1} &\varepsilon^{-1} &0 &0 &-\varepsilon^{-2} &-\varepsilon^{-2} &0 &0 \\
0 &0 &-\varepsilon^2 &\varepsilon^2 &-\varepsilon^{-2} &0 &\varepsilon &0 \\
\varepsilon^{-1} &0 &0 &0 &0 &0 &0 &0 \\
0 &0 &0 &0 &-\varepsilon^{-2} &0 &-\varepsilon &\varepsilon \\
0 &0 &0 &0 &\varepsilon^{-2} &\varepsilon^{-2} &0 &0 \\
0 &0 &0 &0 &\varepsilon^{-2} &\varepsilon^{-2} &\varepsilon &-\varepsilon \\
0 &\varepsilon^{-1} &0 &0 &0 &-\varepsilon^{-2} &0 &0 \\
0 &\varepsilon^{-1} &0 &0 &-\varepsilon^{-2} &-\frac{2}{\varepsilon^2} &-\varepsilon &\varepsilon \\
0 &0 &0 &0 &-\varepsilon^{-2} &0 &0 &\varepsilon \\
0 &0 &\varepsilon^2 &0 &0 &0 &0 &0 \\
0 &0 &\varepsilon^2 &0 &\varepsilon^{-2} &0 &-\varepsilon &0 \\
0 &0 &-\varepsilon^2 &\varepsilon^2 &\varepsilon^{-2} &\varepsilon^{-2} &0 &-\varepsilon \\
0 &0 &0 &0 &\varepsilon^{-2} &0 &0 &0 \\
-\varepsilon^{-1} &0 &-\varepsilon^2 &\varepsilon^2 &\varepsilon^{-2} &0 &\varepsilon &-\varepsilon \\
0 &0 &\varepsilon^2 &0 &\varepsilon^{-2} &0 &-\varepsilon &-\varepsilon \\
0 &0 &0 &0 &\varepsilon^{-2} &\varepsilon^{-2} &-\varepsilon &0 \\
\varepsilon^{-1} &\varepsilon^{-1} &-\varepsilon^2 &0 &0 &0 &0 &0
\end{pmatrix} \quad W(\varepsilon)^T = \begin{pmatrix}
2\varepsilon &3\varepsilon^2 &-\varepsilon &0 &-1 &-2\varepsilon &0 &-\varepsilon^{-2} &-\varepsilon^{-3} &-\varepsilon^{-2} \\
0 &0 &0 &0 &-1 &0 &-\varepsilon^{-3} &0 &0 &0 \\
0 &0 &0 &0 &-1 &-\varepsilon &-\varepsilon^{-3} &-\varepsilon^{-2} &0 &0 \\
0 &-\varepsilon^2 &0 &0 &1 &\varepsilon &\varepsilon^{-3} &\varepsilon^{-2} &0 &0 \\
2\varepsilon &2\varepsilon^2 &\varepsilon &\varepsilon^2 &-1 &-\varepsilon &0 &0 &-\varepsilon^{-3} &-\varepsilon^{-2} \\
0 &-2\varepsilon^2 &-\varepsilon &-2\varepsilon^2 &0 &0 &0 &0 &0 &\varepsilon^{-2} \\
0 &2\varepsilon^2 &2\varepsilon &2\varepsilon^2 &0 &0 &0 &0 &0 &-\varepsilon^{-2} \\
0 &-3\varepsilon^2 &2\varepsilon &0 &0 &2\varepsilon &0 &\varepsilon^{-2} &0 &\varepsilon^{-2} \\
-3\varepsilon &-2\varepsilon^2 &-3\varepsilon &-2\varepsilon^2 &1 &0 &\varepsilon^{-3} &0 &\varepsilon^{-3} &\varepsilon^{-2} \\
3\varepsilon &2\varepsilon^2 &\varepsilon &\varepsilon^2 &-1 &-\varepsilon &-\varepsilon^{-3} &-\varepsilon^{-2} &-\varepsilon^{-3} &-\varepsilon^{-2} \\
3\varepsilon &3\varepsilon^2 &\varepsilon &\varepsilon^2 &-1 &-\varepsilon &0 &0 &-\varepsilon^{-3} &-\varepsilon^{-2} \\
2\varepsilon &0 &\varepsilon &0 &-1 &0 &0 &0 &-\varepsilon^{-3} &0 \\
\varepsilon &\varepsilon^2 &0 &0 &-1 &-\varepsilon &-\varepsilon^{-3} &-\varepsilon^{-2} &0 &0 \\
2\varepsilon &-\varepsilon^2 &\varepsilon &0 &-1 &\varepsilon &0 &\varepsilon^{-2} &-\varepsilon^{-3} &0 \\
-3\varepsilon &0 &-\varepsilon &0 &1 &0 &\varepsilon^{-3} &0 &\varepsilon^{-3} &0 \\
0 &3\varepsilon^2 &0 &\varepsilon^2 &1 &-\varepsilon &\varepsilon^{-3} &0 &0 &-\varepsilon^{-2} \\
2\varepsilon &0 &\varepsilon &0 &0 &0 &0 &0 &-\varepsilon^{-3} &0 \\
0 &0 &\varepsilon &\varepsilon^2 &0 &0 &0 &0 &0 &0 \\
3\varepsilon &3\varepsilon^2 &\varepsilon &\varepsilon^2 &-1 &-\varepsilon &-\varepsilon^{-3} &-\varepsilon^{-2} &-\varepsilon^{-3} &-\varepsilon^{-2} \\
0 &-\varepsilon^2 &2\varepsilon &\varepsilon^2 &0 &\varepsilon &0 &\varepsilon^{-2} &0 &0 \\
0 &\varepsilon^2 &0 &0 &0 &-\varepsilon &0 &-\varepsilon^{-2} &0 &0 \\
0 &-\varepsilon^2 &\varepsilon &0 &0 &\varepsilon &0 &\varepsilon^{-2} &0 &0 \\
0 &3\varepsilon^2 &0 &\varepsilon^2 &0 &-2\varepsilon &0 &-\varepsilon^{-2} &0 &-\varepsilon^{-2} \\
-2\varepsilon &-2\varepsilon^2 &-\varepsilon &-\varepsilon^2 &1 &\varepsilon &\varepsilon^{-3} &\varepsilon^{-2} &\varepsilon^{-3} &\varepsilon^{-2} \\
0 &2\varepsilon^2 &0 &\varepsilon^2 &-1 &-\varepsilon &-\varepsilon^{-3} &-\varepsilon^{-2} &0 &-\varepsilon^{-2} \\
-2\varepsilon &-3\varepsilon^2 &-\varepsilon &-\varepsilon^2 &1 &\varepsilon &0 &0 &\varepsilon^{-3} &\varepsilon^{-2} \\
0 &-2\varepsilon^2 &0 &-\varepsilon^2 &0 &0 &0 &0 &0 &\varepsilon^{-2} \\
0 &0 &-\varepsilon &0 &0 &0 &0 &0 &0 &0 \\
0 &3\varepsilon^2 &-\varepsilon &0 &0 &-2\varepsilon &0 &-\varepsilon^{-2} &0 &-\varepsilon^{-2} \\
0 &\varepsilon^2 &0 &0 &0 &0 &0 &0 &0 &0 \\
\varepsilon &0 &0 &0 &0 &0 &0 &0 &0 &0
\end{pmatrix}
$}
\end{equation*}

\begin{equation*}
    \resizebox{\textwidth}{!}{
    $\displaystyle
    V(\varepsilon)^T = \begin{pmatrix}
0 &0 &-\varepsilon &\varepsilon^4 &0 &0 &0 &0 &0 &0 &0 &0 &-\varepsilon^{-2} &\varepsilon &0 &0 &0 &-\varepsilon^{-2} &\varepsilon &0 \\
0 &0 &0 &-\varepsilon^4 &0 &0 &0 &-\varepsilon &\varepsilon^4 &\varepsilon^4 &0 &0 &0 &0 &0 &0 &0 &0 &0 &0 \\
-1 &0 &0 &0 &-2\varepsilon^4 &1 &0 &\varepsilon &0 &\varepsilon^4 &0 &0 &\varepsilon^{-2} &0 &-\varepsilon &0 &0 &0 &0 &0 \\
-1 &0 &0 &0 &-2\varepsilon^4 &1 &0 &\varepsilon &0 &\varepsilon^4 &0 &0 &0 &0 &0 &0 &0 &0 &0 &0 \\
0 &0 &2\varepsilon &-2\varepsilon^4 &0 &0 &0 &-\varepsilon &\varepsilon^4 &0 &0 &0 &\varepsilon^{-2} &-\varepsilon &0 &\varepsilon^{-3} &0 &\varepsilon^{-2} &-\varepsilon &3\varepsilon \\
0 &1 &0 &0 &\varepsilon^4 &0 &0 &0 &0 &0 &0 &0 &0 &0 &0 &0 &\varepsilon^{-3} &0 &0 &\varepsilon \\
0 &1 &\varepsilon &-\varepsilon^4 &\varepsilon^4 &0 &0 &0 &0 &0 &0 &0 &0 &0 &0 &0 &\varepsilon^{-3} &-\varepsilon^{-2} &0 &3\varepsilon \\
0 &0 &-\varepsilon &0 &0 &0 &0 &-\varepsilon &\varepsilon^4 &\varepsilon^4 &0 &\varepsilon^{-3} &-\varepsilon^{-2} &0 &\varepsilon &0 &\varepsilon^{-3} &-\varepsilon^{-2} &0 &\varepsilon \\
0 &0 &0 &0 &0 &0 &0 &0 &0 &0 &0 &0 &0 &0 &0 &0 &0 &-\varepsilon^{-2} &\varepsilon &\varepsilon \\
1 &0 &0 &0 &2\varepsilon^4 &0 &0 &-\varepsilon &\varepsilon^4 &\varepsilon^4 &\varepsilon^{-3} &0 &0 &0 &2\varepsilon &\varepsilon^{-3} &0 &0 &0 &2\varepsilon \\
0 &0 &0 &0 &0 &0 &0 &0 &0 &0 &0 &0 &0 &0 &0 &\varepsilon^{-3} &0 &-\varepsilon^{-2} &\varepsilon &3\varepsilon \\
0 &0 &-\varepsilon &\varepsilon^4 &0 &0 &0 &0 &0 &0 &0 &0 &0 &0 &0 &0 &0 &0 &0 &0 \\
-1 &0 &0 &0 &-2\varepsilon^4 &1 &0 &0 &0 &2\varepsilon^4 &0 &0 &0 &0 &0 &0 &0 &0 &0 &0 \\
0 &0 &\varepsilon &-\varepsilon^4 &0 &0 &0 &-\varepsilon &\varepsilon^4 &0 &0 &0 &0 &0 &0 &0 &0 &0 &0 &0 \\
0 &0 &0 &0 &0 &0 &0 &-\varepsilon &\varepsilon^4 &\varepsilon^4 &0 &0 &0 &0 &0 &0 &0 &0 &0 &0 \\
0 &0 &0 &0 &0 &0 &0 &0 &0 &0 &0 &0 &0 &0 &0 &0 &0 &\varepsilon^{-2} &0 &-\varepsilon \\
0 &0 &\varepsilon &-\varepsilon^4 &-\varepsilon^4 &0 &0 &-\varepsilon &\varepsilon^4 &\varepsilon^4 &0 &0 &0 &0 &0 &0 &0 &0 &0 &0 \\
0 &0 &0 &0 &0 &0 &1 &0 &0 &\varepsilon^4 &0 &-\varepsilon^{-3} &0 &0 &-\varepsilon &0 &0 &0 &0 &0 \\
1 &0 &0 &0 &2\varepsilon^4 &0 &0 &0 &0 &0 &\varepsilon^{-3} &0 &0 &0 &2\varepsilon &\varepsilon^{-3} &0 &\varepsilon^{-2} &-\varepsilon &\varepsilon \\
0 &0 &-\varepsilon &\varepsilon^4 &0 &0 &1 &0 &0 &\varepsilon^4 &0 &0 &-\varepsilon^{-2} &\varepsilon &0 &0 &0 &0 &0 &-\varepsilon \\
0 &0 &\varepsilon &-\varepsilon^4 &0 &0 &1 &0 &\varepsilon^4 &0 &0 &0 &0 &0 &0 &0 &0 &0 &0 &0 \\
0 &0 &0 &0 &0 &0 &1 &0 &0 &\varepsilon^4 &0 &0 &0 &0 &0 &0 &0 &0 &0 &0 \\
0 &0 &\varepsilon &-\varepsilon^4 &0 &0 &0 &-\varepsilon &\varepsilon^4 &0 &0 &\varepsilon^{-3} &0 &0 &\varepsilon &0 &\varepsilon^{-3} &-\varepsilon^{-2} &0 &2\varepsilon \\
1 &0 &0 &0 &2\varepsilon^4 &0 &0 &-\varepsilon &\varepsilon^4 &\varepsilon^4 &\varepsilon^{-3} &0 &-\varepsilon^{-2} &\varepsilon &3\varepsilon &\varepsilon^{-3} &0 &-\varepsilon^{-2} &\varepsilon &3\varepsilon \\
0 &0 &0 &-\varepsilon^4 &0 &0 &0 &-\varepsilon &\varepsilon^4 &\varepsilon^4 &0 &0 &-\varepsilon^{-2} &0 &\varepsilon &0 &0 &-\varepsilon^{-2} &0 &\varepsilon \\
0 &0 &\varepsilon &-\varepsilon^4 &0 &0 &0 &0 &0 &0 &0 &0 &0 &0 &0 &\varepsilon^{-3} &0 &0 &0 &3\varepsilon \\
0 &1 &-\varepsilon &0 &2\varepsilon^4 &0 &0 &0 &0 &0 &0 &0 &-\varepsilon^{-2} &0 &\varepsilon &0 &\varepsilon^{-3} &-\varepsilon^{-2} &0 &2\varepsilon \\
0 &1 &0 &0 &\varepsilon^4 &0 &0 &0 &0 &0 &0 &0 &0 &0 &0 &0 &0 &0 &0 &0 \\
0 &0 &0 &0 &0 &0 &0 &0 &0 &0 &0 &\varepsilon^{-3} &0 &0 &\varepsilon &0 &\varepsilon^{-3} &0 &0 &\varepsilon \\
-1 &0 &\varepsilon &-\varepsilon^4 &-2\varepsilon^4 &1 &0 &\varepsilon &0 &\varepsilon^4 &-\varepsilon^{-3} &0 &0 &0 &-2\varepsilon &0 &0 &0 &0 &\varepsilon \\
1 &0 &0 &0 &2\varepsilon^4 &0 &0 &-\varepsilon &\varepsilon^4 &\varepsilon^4 &0 &0 &0 &0 &0 &0 &0 &0 &0 &0
\end{pmatrix}
$}
\end{equation*}

\pagebreak
\subsection{Residual Matrices }
The remainder matrices $R_3,R_4,R_5,R_6,R_7,R_8$ for the corresponding algorithm are given below:

\begin{equation*}
\resizebox{\textwidth}{!}{
$\displaystyle
R_3 = \begin{pmatrix}
    0& 2a_{23}(b_{13} + b_{23} + b_{34} + b_{45})- 2a_{24}(b_{13} - b_{44} + b_{45})
  & 0& 0& 0\\
    0& 0& 
   a_{23}(b_{13} + b_{23} + b_{34} + b_{45}) - a_{24}(b_{13} - b_{44} + b_{45}) & 0& 0
\end{pmatrix} 
$}
\end{equation*}

\begin{equation*}
\resizebox{\textwidth}{!}{
$\displaystyle
R_4 = 
\begin{pmatrix}
    0& -2a_{13}(b_{13} + b_{23} + b_{34} + b_{45})+ 2a_{14}(b_{13} - b_{44} + b_{45})
    & 0& 0& 0\\
   -a_{23}(b_{13} + b_{23} + b_{34} + b_{45}) + a_{24}(b_{13} - b_{44} + b_{45})
    & a_{23}(b_{13} + b_{23} + b_{34} + b_{45}) - a_{24}(b_{13} - b_{44} + b_{45})
    &
    -a_{13}(b_{13} + b_{23} + b_{34} + b_{45})+ a_{14}(b_{13} - b_{44} + b_{45})
    & 0& 0
\end{pmatrix}
$}
\end{equation*}

\begin{equation*}
\resizebox{\textwidth}{!}{
$\displaystyle
R_5 = 
\begin{pmatrix}
    0& 0& 0& 0& 0\\ a_{13}(b_{13} + b_{23} + b_{34} + b_{45})- a_{14}(b_{13} - b_{44} + b_{45})
    &  -a_{13}(b_{13} + b_{23} + b_{34} + b_{45})+ a_{14}(b_{13} - b_{44} + b_{45})
    & 0& 0& 0
\end{pmatrix}
$}
\end{equation*}

\vspace{4mm}

\begin{equation*}
R_6 = \begin{pmatrix}
    0& -2a_{23}(b_{24} +b_{25}) + 2a_{24}b_{14}& 0& 0& 0\\
    0& 0& -a_{23}(b_{24} + b_{25}) + a_{24}b_{14}& 0& 0
\end{pmatrix}
\end{equation*}

\vspace{4mm}

\begin{equation*}
R_7 = \begin{pmatrix}
    0& 2a_{13}(b_{24} + b_{25}) - 2a_{14}b_{14}& 0& 0& 0\\
    a_{23}(b_{24} + b_{25}) - a_{24}b_{14}& -a_{23}(b_{24} + b_{25}) + a_{24}b_{14}& a_{13}(b_{24} + b_{25}) - a_{14}b_{14}& 0& 0
\end{pmatrix}
\end{equation*}

\vspace{4mm}

\begin{equation*}
R_8 = \begin{pmatrix}
    0& 0& 0& 0& 0\\
    -a_{13}(b_{24} + b_{25}) + a_{14}b_{14}& a_{13}(b_{24} + b_{25}) - a_{14}b_{14}& 0& 0& 0
\end{pmatrix}
\end{equation*}

\subsection{Sparsified Encoding Matrices}\label{ap: sparsification_example}
Finally, sparsifying these algorithms as described in \cref{section: alternative_basis}, noting that the specialization of \cref{subsec: section_two_special_case} applies, yields the following basis transformations and (constant) encoding matrices:

\vspace{2mm}

\begin{equation*}
    \phi_{\varepsilon}^{-1}=\begin{pmatrix}
   \varepsilon^{-1} & \varepsilon^{-1} & 0 & 0&-\varepsilon^{-2}&-\varepsilon^{-2}&0&0\\
0&0&\varepsilon^2&-\varepsilon^2&-\varepsilon^{-2}&-\varepsilon^{-2}&0&\varepsilon \\ 
0&-\varepsilon^{-1}&0&0&-\varepsilon^{-2}&0&\varepsilon&0 \\
0&0&\varepsilon^2&0&-\varepsilon^{-2}&-\varepsilon^{-2}&0&0 \\
0&0&0&0&-\varepsilon^{-2}&-\varepsilon^{-2}&0&0 \\
0&0&0&0&\varepsilon^{-2}&0&-\varepsilon&0 \\
0&0&0&0&-\varepsilon^{-2}&-\varepsilon^{-2}& \varepsilon&0 \\
0&0&0&0&\varepsilon^{-2}&\varepsilon^{-2}&0&-\varepsilon \\
\end{pmatrix}
\end{equation*}

\vspace{4mm}

\begin{equation*}
\nu^{-1}_{\varepsilon} =\begin{pmatrix}
    0&0&\varepsilon&\varepsilon^2&0&0&0&0&0&0\\
-\varepsilon&-\varepsilon^2&0&0&0&0&0&0&0&0\\
0&-\varepsilon^2&0&0&0&0&0&0&0&0\\
0&3\varepsilon^2&0&\varepsilon^2&0&-2\varepsilon&0&-\varepsilon^{-2}&0&-\varepsilon^{-2}\\
-2\varepsilon&-2\varepsilon^2&-\varepsilon&-\varepsilon^2&1&\varepsilon&0&0&\varepsilon^{-3}&\varepsilon^{-2}\\
0&2\varepsilon^2&0&\varepsilon^2&0&0&0&0&0&-\varepsilon^{-2}\\
-3\varepsilon&-3\varepsilon^2&-\varepsilon&-\varepsilon^2&1&\varepsilon&\varepsilon^{-3}&\varepsilon^{-2}&\varepsilon^{-3}&\varepsilon^{-2}\\
0&0&0&0&1&\varepsilon&\varepsilon^{-3}&\varepsilon^{-2}&0&0\\
0&\varepsilon^2&0&0&0&-\varepsilon&0&-\varepsilon^{-2}&0&0\\
0&\varepsilon^2&-2\varepsilon&-\varepsilon^2&0&-\varepsilon&0&-\varepsilon^{-2}&0&0\\
\end{pmatrix}
\end{equation*}

\begin{equation*}
\resizebox{\textwidth}{!}{
$\displaystyle
\psi^{-1}_{\varepsilon} =\begin{pmatrix}
    -1&0&0&-\varepsilon^4&-2\varepsilon^4&1&-1&0&\varepsilon^4&\varepsilon^4&0&\varepsilon^{-3}&0&0&\varepsilon&0&\varepsilon^{-3}&-\varepsilon^{-2}&0&2\varepsilon\\
0&1&\varepsilon&-\varepsilon^4&\varepsilon^4&0&0&0&0&0&0&0&0&0&0&0&\varepsilon^{-3}&0&-\varepsilon&2\varepsilon\\
0&0&0&0&0&0&0&\varepsilon&0&-\varepsilon^4&0&0&0&0&0&0&0&0&0&0\\
-1&0&0&0&-2\varepsilon^4&0&0&\varepsilon&-\varepsilon^4&-\varepsilon^4&0&0&0&0&0&0&0&0&0&0\\
0&0&\varepsilon&0&0&0&0&\varepsilon&-\varepsilon^4&-\varepsilon^4&0&-\varepsilon^{-3}&\varepsilon^{-2}&0&-\varepsilon&0&-\varepsilon^{-3}&\varepsilon^{-2}&0&-\varepsilon\\
-1&0&0&0&-2\varepsilon^4&0&0&0&0&0&-\varepsilon^{-3}&0&0&0&-2\varepsilon&-\varepsilon^{-3}&0&-\varepsilon^{-2}&\varepsilon&-\varepsilon\\
0&0&0&0&0&0&-1&0&0&-\varepsilon^4&0&0&\varepsilon^{-2}&-\varepsilon&0&0&0&\varepsilon^{-2}&-\varepsilon&0\\
0&0&0&0&0&0&-1&0&0&-\varepsilon^4&0&\varepsilon^{-3}&0&0&\varepsilon&0&0&0&0&0\\
0&0&\varepsilon&-\varepsilon^4&-\varepsilon^4&0&0&-\varepsilon&\varepsilon^4&\varepsilon^4&0&0&0&0&0&0&0&0&0&0\\
0&0&\varepsilon&-\varepsilon^4&0&0&0&0&0&0&0&0&0&0&0&\varepsilon^{-3}&0&0&0&3\varepsilon\\
0&0&0&0&0&0&0&0&0&0&0&0&0&0&0&0&0&\varepsilon^{-2}&-\varepsilon&-\varepsilon\\
0&0&0&0&0&0&1&0&0&\varepsilon^4&0&0&0&0&0&0&0&0&0&0\\
-1&0&0&0&-2\varepsilon^4&0&0&\varepsilon&-\varepsilon^4&-\varepsilon^4&-\varepsilon^{-3}&0&0&0&-2\varepsilon&-\varepsilon^{-3}&0&0&0&-2\varepsilon\\
0&0&0&0&0&0&1&0&0&\varepsilon^4&0&-\varepsilon^{-3}&0&0&-\varepsilon&0&-\varepsilon^{-3}&\varepsilon^{-2}&0&-2\varepsilon\\
0&1&0&0&\varepsilon^4&0&1&0&0&\varepsilon^4&0&-\varepsilon^{-3}&0&0&-\varepsilon&0&0&0&0&0\\
0&0&0&0&0&0&-1&0&0&-\varepsilon^4&0&\varepsilon^{-3}&0&0&\varepsilon&0&\varepsilon^{-3}&0&0&\varepsilon\\
0&0&-\varepsilon&\varepsilon^4&0&0&1&0&0&\varepsilon^4&0&0&-\varepsilon^{-2}&\varepsilon&0&0&0&-\varepsilon^{-2}&\varepsilon&0\\
0&0&0&-\varepsilon^4&0&0&-1&-\varepsilon&\varepsilon^4&0&0&\varepsilon^{-3}&-\varepsilon^{-2}&0&2\varepsilon&0&\varepsilon^{-3}&-\varepsilon^{-2}&0&2\varepsilon\\
0&0&\varepsilon&-\varepsilon^4&0&0&1&-\varepsilon&\varepsilon^4&\varepsilon^4&0&0&0&0&0&0&0&0&0&0\\
-1&0&0&0&-2\varepsilon^4&1&0&\varepsilon&0&\varepsilon^4&0&0&0&0&0&0&0&0&0&0\\
\end{pmatrix}
$}
\end{equation*}

\vspace{2mm}

\begin{equation*}
    \resizebox{\textwidth}{!}{
    $\displaystyle
    U_{\phi}^T = \begin{pmatrix}
    0&0&0&1&-1&0&0&1\\
1&0&1&-1&0&0&0&0\\
0&0&0&0&0&1&0&0\\
0&0&-1&0&0&0&0&0\\
0&0&0&0&0&0&0&1\\
0&-1&0&0&1&-1&0&0\\
0&-1&0&0&0&0&1&0\\
0&0&0&1&0&0&-1&0\\
0&-1&0&0&1&0&0&-1\\
0&0&0&1&0&0&0&0\\
0&-1&0&0&0&0&0&-1\\
1&1&0&-1&0&0&0&0\\
0&0&-1&0&0&-1&0&0\\
0&0&-1&0&0&-1&0&-1\\
1&0&0&0&0&0&0&0\\
0&-1&0&0&0&-1&0&-1\\
1&0&1&0&-1&1&0&0\\
0&0&0&0&1&-1&-2&-1\\
0&0&0&0&-1&0&0&0\\
0&0&0&0&-1&0&1&1\\
0&0&-1&0&0&0&1&0\\
0&0&-1&0&1&0&0&-1\\
0&0&0&0&0&-1&-1&-1\\
0&0&0&1&-1&0&0&0\\
0&0&0&1&-1&1&0&0\\
0&-1&0&0&0&0&0&0\\
0&0&0&0&-1&1&1&0\\
-1&-1&-1&0&0&0&2&0\\
0&0&0&1&0&1&0&1\\
0&0&0&0&0&0&-1&0\\
1&0&0&-1&0&0&0&0\\
    \end{pmatrix} \quad W_{\nu}^T = \begin{pmatrix}
        0&0&0&0&-1&0&0&0&0&1\\
0&0&-1&0&0&0&0&-1&-1&0\\
0&0&0&0&0&0&0&-1&0&0\\
0&0&1&0&0&0&0&1&0&0\\
0&0&0&0&-1&0&0&0&0&0\\
-1&0&0&0&0&-1&0&0&0&0\\
0&0&0&0&0&1&0&0&1&-1\\
0&0&0&-1&0&0&0&0&1&-1\\
0&0&0&0&0&0&1&0&0&1\\
0&0&1&0&0&0&-1&0&0&0\\
0&-1&0&0&-1&0&0&0&0&0\\
0&0&0&-1&-1&0&0&0&1&0\\
0&-1&0&0&0&0&0&-1&0&0\\
0&0&0&-1&-1&0&0&0&0&0\\
0&0&0&0&0&1&1&0&1&0\\
0&0&0&1&0&0&0&1&0&0\\
0&1&0&0&0&-1&-1&1&0&0\\
1&0&0&0&0&0&0&0&0&0\\
0&0&0&0&0&0&-1&0&0&0\\
0&0&0&0&0&0&0&0&0&-1\\
0&0&0&0&0&0&0&0&1&0\\
-1&0&0&0&0&0&0&0&0&-1\\
0&0&0&1&0&0&0&0&0&0\\
0&-1&0&0&0&0&1&0&0&0\\
0&0&0&0&0&1&0&-1&0&0\\
0&0&1&0&1&0&0&0&0&0\\
0&0&0&0&0&-1&0&0&0&0\\
1&0&0&0&0&0&0&0&-1&1\\
-1&0&0&1&0&0&0&0&0&0\\
0&0&-1&0&0&0&0&0&0&0\\
0&-1&1&0&0&0&0&0&0&0\\
    \end{pmatrix}
    $}
\end{equation*}

\begin{equation*}
    \resizebox{\textwidth}{!}{
    $\displaystyle
    V_{\psi}^T = \begin{pmatrix}
        0&0&0&0&0&0&0&0&0&0&0&-1&0&0&0&0&1&0&0&0\\
1&0&0&0&0&0&0&0&0&0&0&0&0&1&0&0&0&0&0&-1\\
1&0&0&0&0&0&0&0&0&0&0&0&0&0&0&0&0&-1&0&0\\
0&0&0&0&0&0&0&0&0&0&0&0&0&0&0&0&0&0&0&1\\
0&0&0&0&0&0&1&0&0&1&0&0&0&0&0&0&0&0&1&0\\
0&0&0&0&0&0&0&0&0&0&0&0&0&0&1&1&0&0&0&0\\
0&1&0&0&0&0&0&0&0&0&-1&0&0&0&0&0&0&0&0&0\\
0&0&0&0&-1&0&0&0&0&0&0&0&0&0&0&0&0&0&0&0\\
0&0&0&0&0&0&0&0&0&0&-1&0&0&0&0&0&0&0&0&0\\
0&0&0&0&0&0&0&0&0&0&0&0&-1&0&0&0&0&0&0&0\\
0&-1&0&0&0&0&0&0&0&1&0&0&0&-1&1&0&0&0&0&0\\
0&0&0&0&0&0&1&0&0&0&0&0&0&0&0&0&1&0&0&0\\
0&0&-1&0&0&0&0&0&0&0&0&0&0&0&0&0&0&0&0&1\\
0&0&0&0&0&0&0&0&0&0&0&-1&0&0&0&0&0&0&1&0\\
0&0&0&0&0&1&0&0&0&0&1&0&-1&0&0&0&0&0&0&0\\
0&0&0&0&0&0&0&0&0&0&0&0&0&1&0&1&0&0&0&0\\
0&0&0&0&0&0&0&0&1&0&0&0&0&0&0&0&0&0&0&0\\
0&0&0&0&0&0&0&-1&0&0&0&0&0&0&0&0&0&0&0&0\\
0&0&0&0&0&-1&0&0&0&0&0&0&0&0&0&0&0&0&0&0\\
0&0&0&0&0&0&0&0&0&0&1&0&0&0&0&0&1&0&0&0\\
0&0&1&0&0&0&0&0&0&0&0&0&0&0&0&0&0&0&1&0\\
0&0&0&0&0&0&0&0&0&0&0&1&0&0&0&0&0&0&0&0\\
0&0&0&0&0&0&0&0&0&0&0&0&0&-1&0&0&0&0&1&0\\
0&0&0&0&1&0&0&0&0&0&0&0&-1&0&0&0&1&1&0&0\\
0&0&0&0&0&0&0&0&0&0&0&0&0&0&0&-1&0&1&0&0\\
0&0&0&0&0&0&0&0&0&1&0&0&0&0&0&0&0&0&0&0\\
0&0&0&0&0&0&0&0&-1&0&0&0&0&0&1&0&0&1&0&0\\
0&0&0&0&0&0&0&1&0&0&0&0&0&0&1&0&0&0&0&0\\
0&0&0&0&0&0&0&0&0&0&0&1&0&0&0&1&0&0&0&0\\
0&0&0&-1&0&0&0&0&0&1&0&0&1&0&0&0&0&0&0&1\\
0&0&0&-1&0&0&0&0&0&0&0&0&0&0&0&0&0&0&0&0\\
    \end{pmatrix}
    $}
\end{equation*}

\end{document}